\documentclass[13pt]{amsart}
\usepackage{amsmath,amssymb,amsthm,enumerate,multicol,graphicx,mathtools,enumitem,xcolor,hyperref}

\usepackage{appendix} 
\usepackage{listings} 

\usepackage{soul} 

\usepackage[margin=1in]{geometry}
\theoremstyle{plain}
\newtheorem{theorem}{Theorem}
\numberwithin{theorem}{section}

\newtheorem{proposition}[theorem]{Proposition}
\newtheorem{lemma}[theorem]{Lemma}
\newtheorem{conjecture}[theorem]{Conjecture}
\newtheorem{question}[theorem]{Question}
\newtheorem{problem}[theorem]{Problem}

\theoremstyle{definition}
\newtheorem{definition}[theorem]{Definition}
\newtheorem{procedure}[theorem]{Procedure}
\newtheorem{remark}[theorem]{Remark}
\newtheorem{example}[theorem]{Example}
\newcommand{\Wrain}{\mathrm{weak}\text{-}\mathrm{rain}\#}
\newcommand{\rain}{\mathrm{rainbow}}
\newcommand{\bridge}{\mathrm{bridge}}
\newcommand{\braid}{\mathrm{braid}}
\newcommand{\mrk}{\mathrm{mrk}}

\newcommand{\wt}{\widetilde}
\newcommand{\Z}{\mathbb{Z}}

\newcommand{\R}{\mathbb{R}}

\newcommand{\Tcal}{\mathcal{T}}
\newcommand{\Scal}{\mathcal{S}}
\newcommand{\Pcal}{\mathcal{P}}

\newcommand{\Acal}{\mathcal{A}}
\newcommand{\Dcal}{\mathcal{D}}
\newcommand{\Ccal}{\mathcal{C}}
\newcommand{\Ucal}{\mathcal{U}}
\newcommand{\T}{\overline{T}}
\newcommand{\2}{\overline{2}}
\newcommand{\3}{\overline{3}}

\title{Trisected Rainbows and Braids}
\author{Rom\'an Aranda}
\address{Department of Mathematics, University of Nebraska-Lincoln}
\email{jarandacuevas2@unl.edu}
\author{Scott Carter} 
\address{Department of Mathematics and Statistics, University of South Alabama, Mobile, AL}
\email{carter@southalabama.edu}
\author{Julia Courtney}
\address{Department of Mathematics, University of Nebraska-Lincoln}
\email{jcourtney2@huskers.unl.edu}
\author{Puttipong Pongtanapaisan}
\address{Mathematics, Pitzer College, Claremont, CA}
\email{puttip@pitzer.edu}
\date{}

\begin{document}

\begin{abstract}
New explicit procedures for passing among  triplane diagrams, braid movies, and braid charts for knotted surfaces in $\R^4$ are presented. To this end,  rainbow diagrams, which lie between braid charts and triplanes, are introduced. Inequalities relating the braid index and the bridge index of 2-knots are obtained via these procedures. Another consequence is a 4-dimensional version of the classical result that ``the minimal number of Seifert circles equals the braid index of a link'' due to Yamada. The procedures are exemplified for the spun trefoil, the 2–twist spun trefoil,  and other related examples. 

Of independent interest,  an appendix is included that describes a procedure for drawing a triplane diagram for a satellite surface with a 2-sphere companion. Thus, larger families of surfaces for which we know specific triplane diagrams are obtained.

\end{abstract}

\maketitle

\section{Introduction}
Knotted $2$-dimensional spheres in $4$-dimensional space have been known to exist for more than a century \cite{Artin}. During this era, numerous techniques have been brought to bear upon their study. For example, in the mid-twentieth century, algebraic topology and surgery techniques were the major tools for topological investigations; see, for example, \cite{Kervaire65,Levine}. In addition, Fox~\cite{QuickTrip} and Yajima \cite{YajimaSimple} initiated diagrammatic methods. These were amplified via Zeeman's twist spinning construction \cite{Zeeman}. See also \cite{GoldKauff} and \cite{Litherland}. Other diagrammatic techniques include marked vertex diagrams \cite{KSS1,Yoshikawa},  surface diagrams in $3$-dimensions \cite{CS:Book,Giller,Roseman-diag} and surface braids \cite{KamBook}. A more modern approach to studying knotted spheres, or indeed embedded surfaces in general, was initiated by means of triplane diagrams \cite{MZ17Trans}.

\subsection{Rosetta stone with braids, triplanes, and banded unlinks} 
This paper studies relationships among these representations. Specifically, there are some easy-to-define, yet hard-to-compute, invariants for embedded surfaces in $S^4$. The \emph{braid index} is the minimal number of sheets in a braid surface description taken over all isotopic diagrams. The \emph{bridge index} is the minimal number of arcs needed to describe a surface using a triplane diagram. 
\begin{theorem}\label{thm:bridge_and_braid_index}
Let $F$ be an orientable surface in $S^4$, then 
\[ 
\bridge(F) \leq 5\cdot \braid(F) - 2\cdot \chi(F) -2,
\] 
where $\chi(F)$ is the Euler characteristic of $F$. 
\end{theorem}
To relate braid charts and triplane diagrams, we introduce a new concept. 

\begin{definition} 
A \emph{$b$-stranded rainbow diagram} or \emph{$b$-braided triplane diagram} of an oriented surface $F\subset S^4$ is a triple $\Tcal=(T_1,T_2,T_3)$ such that 
\begin{enumerate} 
\item each $T_i$ is a $b$-string braided tangle, and the axes of all three tangles coincide (see Figure~\ref{fig:rainbow1}), 
\item for all $i,j$, the union $T_i \cup \T_j$ is a closed braid that corresponds to a stabilization of a trivial braid. 
\end{enumerate}
\end{definition} 
\begin{figure}[h]
\centering
\includegraphics[width=.6\textwidth]{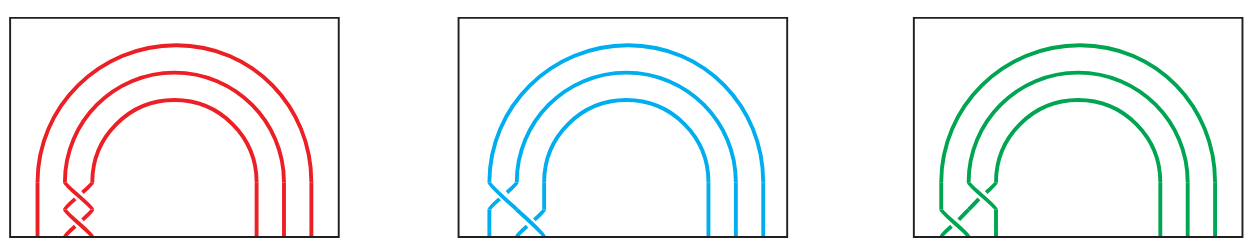}
\caption{A rainbow diagram for the standard unknotted torus.}
\label{fig:rainbow1}
\end{figure}
The \emph{rainbow number} of $F$ is the smallest $b$ such that $F$ admits a $b$-stranded rainbow diagram. We will see that in the same way triplane diagrams yield broken surface diagrams of the underlying surface~\cite{Joseph_Classical_knot_theory}, rainbow diagrams will extend to braid charts. Thus proving the following relation. 

\newtheorem*{thm_main_ineq}{Theorem \ref{thm:main_ineq}}
\begin{thm_main_ineq}
Let $F$ be an orientable surface, then 
\[ 
\braid(F) \leq \rain(F)\leq 5\cdot\braid(F) - 2\cdot\chi(F) -2, 
\]
where $\chi(F)$ is the Euler characteristic of $F$. 
\end{thm_main_ineq}

The main course of this work can be found upon a menu of concrete methods for putting a triplane diagram in rainbow position. We show how to go from a braided rainbow to form a surface braid, and we review methods to move among other diagrammatic descriptions of a fixed embedding. In particular, we have developed several key examples and demonstrated, via a sequence of procedures, how to move cyclically among these representations; see Figure~\ref{fig:Rosetta} for reference. This paper, then, is a state-of-the-art cookbook in which many recipes for the standard fare of knotted surfaces and $2$-knots are described. \textexclamdown Buen provecho!  

\begin{figure}[h]
\centering
\includegraphics[width=7cm]{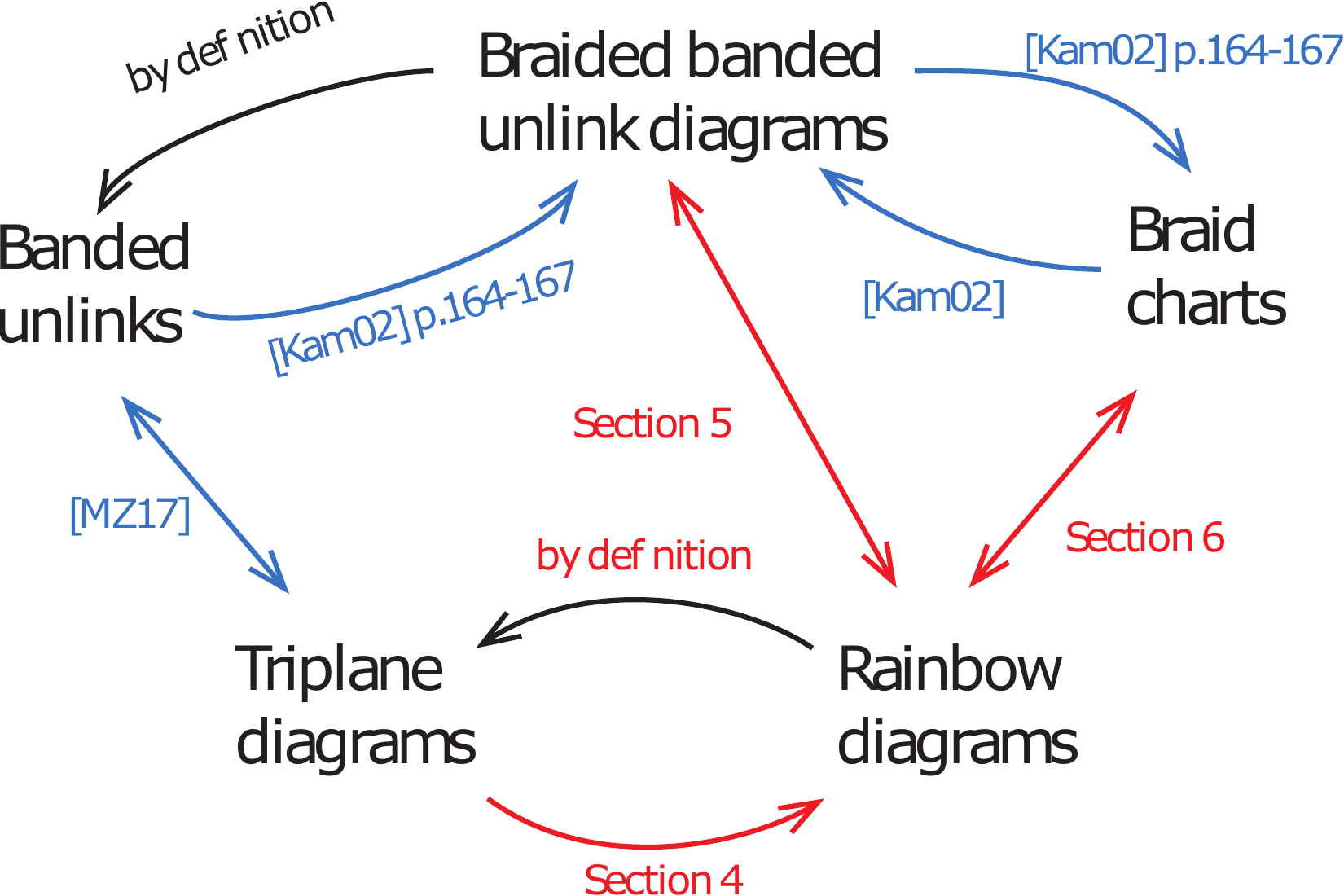}
\caption{How the various diagrammatic descriptions of knotted surfaces are interconnected.}
\label{fig:Rosetta}
\end{figure}

\subsection{Seifert circles count rainbow numbers}
It is a classical result of Yamada that the minimal number of Seifert circles equals the braid index of a link. One of the menu items of this work is a version of Yamada's method for triplane diagrams. We introduce \emph{weak-rainbow diagrams} which are triplane diagrams where each tangle $T_i$ is braided. In Section 3, the distinctions between rainbow diagrams and weak-rainbow diagrams will be emphasized. 

The first step of Seifert's algorithm for oriented link diagrams is to resolve all crossings in an orientation-preserving way to get a disjoint union of embedded loops in the plane. Given a diagram of an oriented $b$-stranded tangle diagram in a disk,  each crossing can be resolved, in an orientation-preserving manner, to get $b$ embedded arcs and some number $s$ of Seifert circles. We say that an oriented triplane diagram is \emph{clustered} if all its negative punctures are separated from the positive punctures. Note that a rainbow diagram is, by construction, clustered. Let  $\Wrain(F)$ ---{\emph{the weak rainbow number}}--- indicate the minimum number of arcs in a weak rainbow diagram for $F$.

\newtheorem*{thm_Seifert_circles}{Theorem \ref{thm:seifertcircles}}
\begin{thm_Seifert_circles}
Let $F$ be an orientable surface in $S^4$. Then, 
\[ \Wrain(F) = \min\left\{b(\Tcal)+\sum_i s_i(\Tcal): \Tcal \text{ is clustered triplane for } F\right\}, \]
where $s_i\left(\Tcal\right)$ is the number of Seifert circles of the $i$-th tangle of $\Tcal$. 
\end{thm_Seifert_circles}

\subsection{Rainbows in 4-dimensional symplectic topology} 
Islambouli and Starkston introduced the notion of bisections with divides for a class of symplectic 4-manifolds called Weinstein domains~\cite{Islambouli_divides}. These are determined by three cut systems, which satisfy some extra conditions on the contact topology of each pair of cut systems. Still, a crucial ingredient is another curve $d$ intersecting each curve in the cut system twice. If the contact topology condition is not met, one still gets an achiral bisection with divides, which in turn has connections to achiral Lefschetz fibrations (see~\cite[Rem 5.5]{contactcutgraphweinstein}. As rainbow diagrams lift to achiral bisections with divides, where the braid axis lifts to $d$, the results in this work could be used to build achiral bisections with divides in the branched cover. On the other hand, one can use rainbow diagrams (with some positivity conditions) to construct properly embedded symplectic surfaces in Weinstein domains; see~\cite{bridge_divides}.  

\subsection*{Outline of paper} 
Section~\ref{sec:background} reviews the necessary background of triplane diagrams needed for this work. Section~\ref{sec:rainbow_diags} is an introduction to rainbow diagrams, where we explain moves between rainbow diagrams and provide examples of rainbows for spun knots and satellite surfaces. Section~\ref{sec:rainbows_and_triplanes} presents three ways to rainbow a triplane diagram; each method generalizes a different algorithm to braid knots in 3-space. Sections~\ref{sec:rains_and_movies} and~\ref{sec:rainbows_and_charts} translate between rainbows, banded unlink diagrams, and braid charts. Although each section includes a brief review of the necessary background on surface diagrams, we encourage the beginner reader to see \cite{CS:Book} for more details. Section~\ref{sec:Ex12} is an example of the methods in Sections~\ref{sec:rains_and_movies} and~\ref{sec:rainbows_and_charts} performed on one surface: the 2-twist spun trefoil. To increase the list of examples of triplane and rainbow diagrams, Appendix~\ref{appen:satellite} provides a procedure for drawing a triplane diagram of satellite surfaces. 

\subsection*{Acknowledgments}
This work began at the Trisectors Workshop 2025 in Austin, Texas; the authors are grateful to the organizers of this event for hosting it. RA would like to thank Patricia Cahn and Agniva Roy for conversations that inspired some of this work. 

\tableofcontents

\section{Background on triplane diagrams}\label{sec:background}
A \emph{trivial tangle} is a tuple $(B,T)$, or simply $T\subset B$, where $B$ is a 3-ball containing properly embedded arcs $T$ such that, fixing the endpoints of $T$, we may isotope $T$ into $\partial B$. We will be interested in trivial tangles up to isotopy rel boundary. 
For a link $L\subset S^3$, a decomposition $(S^3,L)=(B_1,T_1) \cup_\Sigma \overline{(B_2,T_2)}$, where each $(B_i,T_i)$ is a trivial tangle, is called a \emph{bridge splitting}. If each tangle $T_i$ has $b$ strands, we call this decomposition a \emph{$b$-bridge splitting} of $L$. 
The symbol $\T$ denotes the mirror image of a trivial tangle $T$.

A trivial disk system is a collection of properly embedded disks in a 4-ball $\Dcal\subset B^4$ that can be isotoped rel-boundary to lie in $\partial B^4$. Meier and Zupan gave a four-dimensional analog of bridge splittings for surfaces in $S^4$ where every surface can be decomposed as the union of three trivial disk systems. In Definition~\ref{def:bridge_trisection}, we use the convention on the indices from \cite{Joseph_Classical_knot_theory}.

\begin{definition}[\cite{MZ17Trans}]\label{def:bridge_trisection}
Let $F$ be an embedded surface in $S^4$. A $(b;c_1,c_2,c_3)$-bridge trisection of $F$ is a decomposition $(S^4,F)=(X_1,\Dcal_1)\cup (X_2,\Dcal_2)\cup (X_3,\Dcal_3)$ satisfying the following conditions for each $i\in \Z/3\Z$: 
\begin{enumerate}
\item each $(X_i,\Dcal_i)$ is a trivial disk system with $c_i=|\Dcal_i|$, 
\item $(B_i,T_i)=(X_i,\Dcal_i)\cap (X_{i-1},\Dcal_{i-1})$ is a $b$-stranded trivial tangle, and 
\item $(\Sigma,\{p_1,\dots, p_{2b}\})=\cap_{i=1}^3 (X_i,\Dcal_i)$ is a 2-sphere with $2b$-punctures. 
\end{enumerate}
\end{definition}
The surface $\Sigma$ is called the \emph{central surface} or \emph{bridge sphere}. It follows from the definition that each \mbox{$T_i\cup \T_{i+1}$} is a $c_i$-component unlink in $S^3=X_i\cup X_{i+1}$. Thus, by a result of Livingston, the union of the tangles $\bigcup_{i=1}^3 (B_i,T_i)$ determines the embedding of $F$ up to isotopy; see \cite{Livingston} or \cite[Lem 2.5]{MZ17Trans}.

\begin{definition}\begin{sloppypar} 
A triplane diagram is an ordered tuple of trivial tangles with the same endpoints \mbox{$\Tcal=(T_1,T_2,T_3)$} such that, for each $i\in \Z/3\Z$, $T_i\cup \T_{i+1}$ is an unlink in $S^3$. \end{sloppypar}
\end{definition}

As triplane diagrams determine the bridge trisected surface, we will often abuse notation and use $\Tcal$ to denote a triplane diagram and a bridge trisection simultaneously. In this case, the bridge trisected surface described by $\Tcal$ is denoted by $F_\Tcal$. 
Any two triplane diagrams $\Tcal$ and $\Tcal'$ representing istopic surfaces are related by a sequence of \emph{triplane moves} \cite[Thm 1.7]{MZ17Trans}.
\begin{enumerate}
\item \emph{Interior Reidemeister moves}: isotopies on each tangle away from its endpoints. 
\item \emph{Mutual braid transpositions}: to apply the same braid move $\sigma_j^\pm$ (point swap) between the same adjacent strands of each tangle. 
\item \emph{Perturbation/deperturbation moves}: local moves that increase/decrease the number of strands in each tangle by one. This move is depicted in Figure~\ref{fig:perturbs}, up to permutation of the tangles. 
\end{enumerate}

\begin{figure}[ht]
\centering
\includegraphics[width=.8\textwidth]{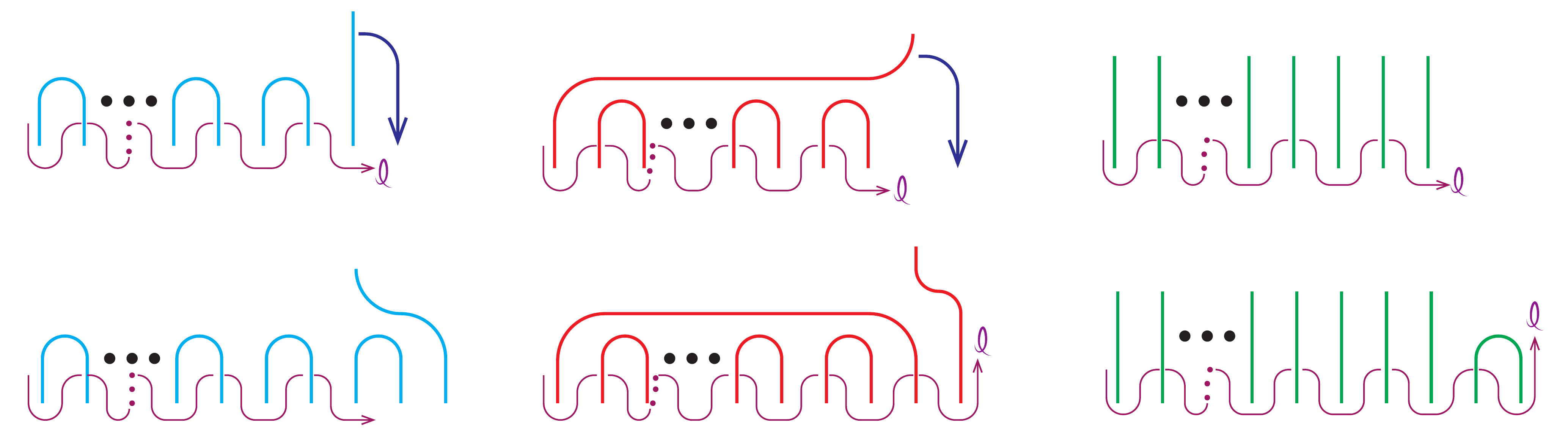}
\caption{Local model for perturbation/deperturbation of triplane diagrams.}
\label{fig:perturbs}
\end{figure}

One can interpret a perturbation move in the following way: the tangles of $\Tcal$ change by pressing two tangles towards the bridge surface and popping a small trivial arc in the third tangle. For example, the perturbation in Figure~\ref{fig:perturbs} can be achieved by pressing/popping each tangle along the blue arcs in the figure. For the time being, ignore the snaking purple arrows that is labeled $\ell$ at the bottom of each triplane. With this interpretation in mind, a special type of perturbation, called a \emph{0-sector perturbation}, can be defined: 
\begin{enumerate}
\item[(4)] 0-sector perturbation: consider $\rho$ to be an arc in $B_i$ the interior of which is disjoint from the tangle $T_i$ and which connects an interior point of $T_i$ with a point in $\Sigma$ that is away from $\{p_1,\dots, p_{2b}\}$. Let $T'_i$ be the tangle with $(b+1)$ strands that results from pressing $T_i$ along $\rho$ with some number of twists (possibly zero), and for $j\neq i$, let $T'_j$ be the result of adding a small trivial arc to $T_j$ with endpoints equal to the two new endpoints of $T'_i$. Then $\Tcal'=(T'_1,T'_2,T'_3)$ is the result of a 0-sector perturbation of $\Tcal$ along $\rho$. See columns 4 and 6 of Figure~\ref{fig:AlexMethod_9_1_part2} for an example of a 0-sector perturbation. 
\end{enumerate}
It was shown in Theorem 4.10 of \cite{AE24} that 0-sector perturbations preserve the isotopy type of the bridge trisected surface. They correspond to isotopies of $F$, supported in a neighborhood of the 3-ball $B_i$ containing $\rho$, that drag $F$ through the bridge sphere to create a new pair of intersections.

\begin{remark}[Warning on 0-sector perturbations]
The result of a 0-perturbation on a triplane diagram may not be a triplane diagram. The reason for this is that the tangle $T'_i$ may not be trivial anymore, as the arc $\rho$ may create some local knotting in $T_i$. 
That said, one can see that the tuple $\Tcal'$ still satisfies that each pair $T'_j\cup \overline{T'}_{j-1}$ is a bridge splitting for an unlink in $S^3$; in fact, $T'_i\cup \overline{T'}_{i\pm 1}$ is isotopic to $T_i\cup \T_{i\pm 1}$. Thus, the union of the tangles $\bigcup_{i=1}^3 (B_i,T'_i)$ still determines an embedding of $F$. This remark is relevant to this work as some of the intermediate tuples in the procedures in Section~\ref{sec:rainbows_and_triplanes} may not be triplane diagrams in the strict sense, while the final tuples will.
\end{remark}

The quantity $b$ in Definition~\ref{def:bridge_trisection} is called the \emph{bridge index} of the decomposition. The minimal $b$ such that $F$ admits a $b$-bridge trisection is the \emph{bridge number of $F$}, denoted by $\bridge(F)$. To estimate bridge numbers of surfaces, we will utilize the meridional rank of surface-links $\mrk(F)$, which is the minimum number of meridians of $F$ needed to generate $\pi_1(S^4-F)$. The following result can be found in~\cite{joseph2025meridional}.

\begin{lemma}\label{lem:mrk_ineq}
Let $F$ be a surface-link in $S^4$. Then, $$3\cdot \mrk(F) -\chi(F) \leq \bridge(F)$$ where $\chi(F)$ denotes the Euler-characteristic of $F$.
\end{lemma}
\begin{proof}
Let $\Tcal$ be a $(\bridge(F);c_1,c_2,c_3)$-bridge trisection of $F$. We know that $\mrk(F)$ is bounded from above by $\min\{c_1,c_2,c_3\}$ \cite[Cor 5.3]{MZ17Trans}. Thus, 
\[ 
3\cdot \mrk(F) \leq \left( c_1+c_2+c_3 -b\right) +b = \chi(F) +b.
\]
\end{proof}

\section{Rainbow diagrams}\label{sec:rainbow_diags}

We will consider trivial tangles which are ``braided'' with respect to a fixed axis. Precisely, a \emph{braided trivial tangle} is a trivial tangle $(B,T)$ with a fixed homeomorphism \mbox{$(B,T)\to (D^2,\{p_1,\dots, p_b\})\times [0,1]$}, where $\{p_i\}_{i=1}^b$ lie in the interior of $D^2$. The curve $\ell=\partial D^2\times \{1/2\}$ is called the \emph{braid axis}. 
A braided trivial tangle is determined by an element of the braid group in $b$ strands, $B_b$. Two braided trivial tangles are considered equal if they are determined by the same braid in $B_b$. By definition, if $T_1$ and $T_2$ are braided trivial tangles represented by braid elements $x_1, x_2\in B_b$, then the union $T_1\cup \T_2$ is the braid closure of the braid element $x_1\cdot x_2^{-1}\in B_b$. 
\begin{definition}
A \emph{braided-bridge trisection} of a closed surface $F\subset S^4$ is a bridge trisection $F=\Dcal_1\cup \Dcal_2\cup \Dcal_3$ such that 
\begin{enumerate}
    \item each $T_i=\Dcal_i\cap \Dcal_{i-1}$ is a braided trivial tangle with respect to the same axis, and 
    \item each $\Dcal_i$ is a braided collection of trivial disks. 
\end{enumerate}
\end{definition}
The spine of a braided-bridge trisection is called a \emph{rainbow diagram}. We say that a $b$-stranded braid $L$ in $S^3$ is \emph{fully destabilizable} if it is braid isotopic to $(b-|L|)$ Markov stabilizations of the $|L|$-stranded crossingless braid. By definition, rainbow diagrams are triplane diagrams satisfying that $T_i\cup \T_{i-1}$ is a fully destabilizable braid. The braid movie of these destabilizations traces the braided disk system $\Dcal_i$. 
\emph{A weak-rainbow diagram} is a triplane diagram such that each tangle is braided around the same axis. Weak-rainbow diagrams can be turned into rainbows via a finite set of moves (Lemma~\ref{lem:strong_rainbow}).

\begin{remark}
Rainbow and weak-rainbow diagrams are distinct notions, as there exist braid representatives of the unlink that are not stabilized. For example, Morton gave a 4-braid whose closure is unknotted but not stabilized in \cite{Morton_1986}, and Rudolph gave a 4-braid for a 2-component unlink that is not stabilized \cite[Ex 5.2]{Rudolph}. Thus, weak-rainbow diagrams containing these braids will not be rainbows. 
\end{remark}

We can define a surface invariant by minimizing among all (weak) rainbow diagrams of a fixed orientable surface $F\subset S^4$. By definition, the following holds. 
\begin{equation}\label{eq:inequality1}
    \bridge(F)\leq \Wrain(F) \leq \rain(F).
\end{equation}

\begin{question}
    Can the quantities in Equation~\eqref{eq:inequality1} be arbitrarily far apart?
\end{question}

\subsection{Rainbow moves}
For the following descriptions, it will be beneficial to represent tangles in a rainbow diagram $\Tcal=(T_1,T_2,T_3)$ as elements in the $b$-stranded braid group $(\beta_1,\beta_2,\beta_3)\in B_b^3$, or as words in the standard generators of the braid group $(w_1,w_2,w_3)$. With this nomenclature in mind, each triplane move has a rainbow version we now describe; we will refer to them as \emph{rainbow moves}.
\begin{enumerate}
    \item {Interior Reidemeister moves} correspond to relations in the braid group; i.e., we replace $(w_1,w_2,w_3)$ with $(w'_1,w'_2,w'_3)$ using the standard braid relations in $B_b$. 
    \item {Mutual braid moves} get replaced with mutual braid moves that keep the tangles braided. These \emph{mutual braided-braid moves} correspond to replacing $(\beta_1,\beta_2,\beta_3)$ with $(x\beta_1y, x\beta_2y,x\beta_3y)$ for some braid elements $x,y\in B_b$. 
    \item {Perturbations} correspond to local modifications on the tangles as in Figure~\ref{fig:rain_perturbs} and will be called \emph{braided perturbations}. They are obtained from the model in Figure~\ref{fig:perturbs} by tightening the axis $\ell$ which no longer should be ignored. There are two ways of tightening $\ell$: so that the arrow points into or out of the page. There is also a second choice for $\ell$ obtained by changing the crossings with the tangles in Figure~\ref{fig:perturbs}. These choices lead us to four local models for a \emph{braided-perturbation} in Figure~\ref{fig:rain_perturbs}. One can see that braided perturbation changes two of the three braids $T_i\cup \T_{i+1}$ by a Markov stabilization, while the third one gains a fully destabilizable unknotted component. 
\end{enumerate}
\begin{figure}[ht]
\centering
\includegraphics[width=.7\textwidth]{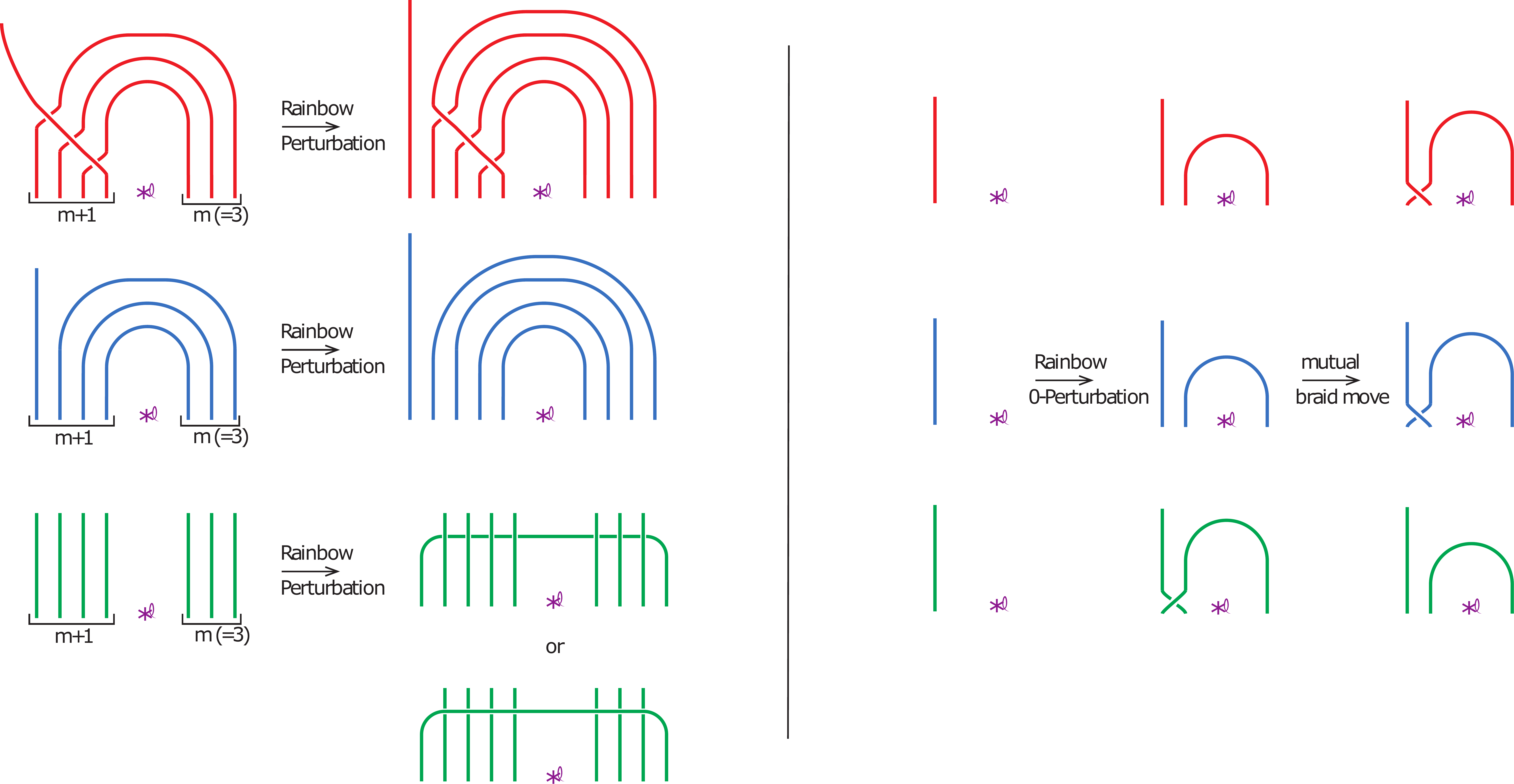}
\caption{Local models of braided-perturbations. Notice how the resulting rainbow diagram can be chosen so that the new strand goes under or over all the old strands. The integer $m$ is the number of crossings involved in the local model before perturbation. For $m>0$, all the over-crossings in the left diagrams can also be chosen to be under-crossings. The particular case $m=0$ is drawn on the right.}
\label{fig:rain_perturbs}
\end{figure}

A valuable kind of braided-perturbation is the case $m=0$ in Figure~\ref{fig:rain_perturbs}, which corresponds to a \emph{braided version of a 0-sector perturbation}. In practice, this can be described modulo permutation on the indices by replacing the tuples $(P_1Q_1,\beta_2,\beta_3)$ in $B_b^3$ with $(P_1\sigma_b^{\pm1}Q_1,\beta_2,\beta_3)$ in $B_{b+1}^3$, where $\beta_1=P_1Q_1\in B_b$. 

\begin{lemma}\label{lem:strong_rainbow}
Let $\Tcal$ be a weak-rainbow diagram. Then $\Tcal$ can be turned into a rainbow diagram after finitely many 0-sector braided perturbations.
\end{lemma}
\begin{proof}
Notice that 0-sector braided perturbations change two of the three braids $L_i=T_i\cup \T_{i+1}$ by a Markov stabilization and the third one by the addition of a 1-braided unknot. Thus, they preserve the property of any $L_i$ to be fully destabilizable. 
Now, assume that $L_1$ is not fully destabilizable. It is a fact that any braid representative of the unlink, like $L_1$, admits a sequence of Markov stabilizations that make it fully destabilizable \cite{Birman}. This sequence can be upgraded to a sequence of 0-sector perturbations of $\Tcal$ (that press only on $T_1$) by adding some interior Reidemeister moves and mutual rainbow-braid moves to mimic the isotopies in the braid $L_1=T_1\cup \T_2$. We can repeat this process with the other two braids $L_2$ and $L_3$. 
\end{proof}

\begin{conjecture}\label{conj:uniqueness_rains}
Any two rainbow diagrams representing isotopic surfaces are related by a sequence of rainbow moves. 
\end{conjecture}

Rainbows correspond to braid movies or braided banded unlink diagrams; see Section~\ref{sec:rains_and_movies} for details. So one may expect to adapt the Proof of Theorem 1.6 from \cite{MZ17Trans} to rainbows. 

It is a fact that $(b;1,b,1)$-triplane diagrams are always perturbed, in fact unknotted~\cite[Prop 4.1]{MZ17Trans}. On the other hand, the following lemma shows that this is now always the case for rainbow diagrams. 

\begin{lemma}\label{lem:unperturbed_rain}
There exists a $(4;1,4,1)$-rainbow diagram $\Tcal_{Morton}$ of an unknotted 2-sphere that is not a braided-perturbation.
\end{lemma}
\begin{proof}
Let $T_1$ be Morton's unstabilized 4-braid of the unknot~\cite{Morton_1986}. In standard braid generators, $T_1$ is given by the braid word $\3\32\32111\21\2$. Let $T_2=T_3$ be crossingless rainbows. As $T_1\cup \T_j$ is the braid closure of Morton's braid, it is unstabilized for $j=2,3$. In particular, the rainbow diagram $\Tcal_{Morton}=(T_1,T_2,T_3)$ cannot be a rainbow-perturbation. 
\end{proof}

\subsection{Exchange moves}
The braid aficionado may recall another move that modifies the braid without changing its link type or its braid index. This move, known as an \emph{exchange move}, was used by Birman and Menasco to prove a version of Markov's theorem for the unlink that does not increase the braid index \cite{BM_theunlink}. For rainbow diagrams, we define an \emph{exchange move} by the modification near the braid axis given in Figure~\ref{fig:exchange_move}.

\begin{figure}[ht]
\centering
\includegraphics[width=.7\textwidth]{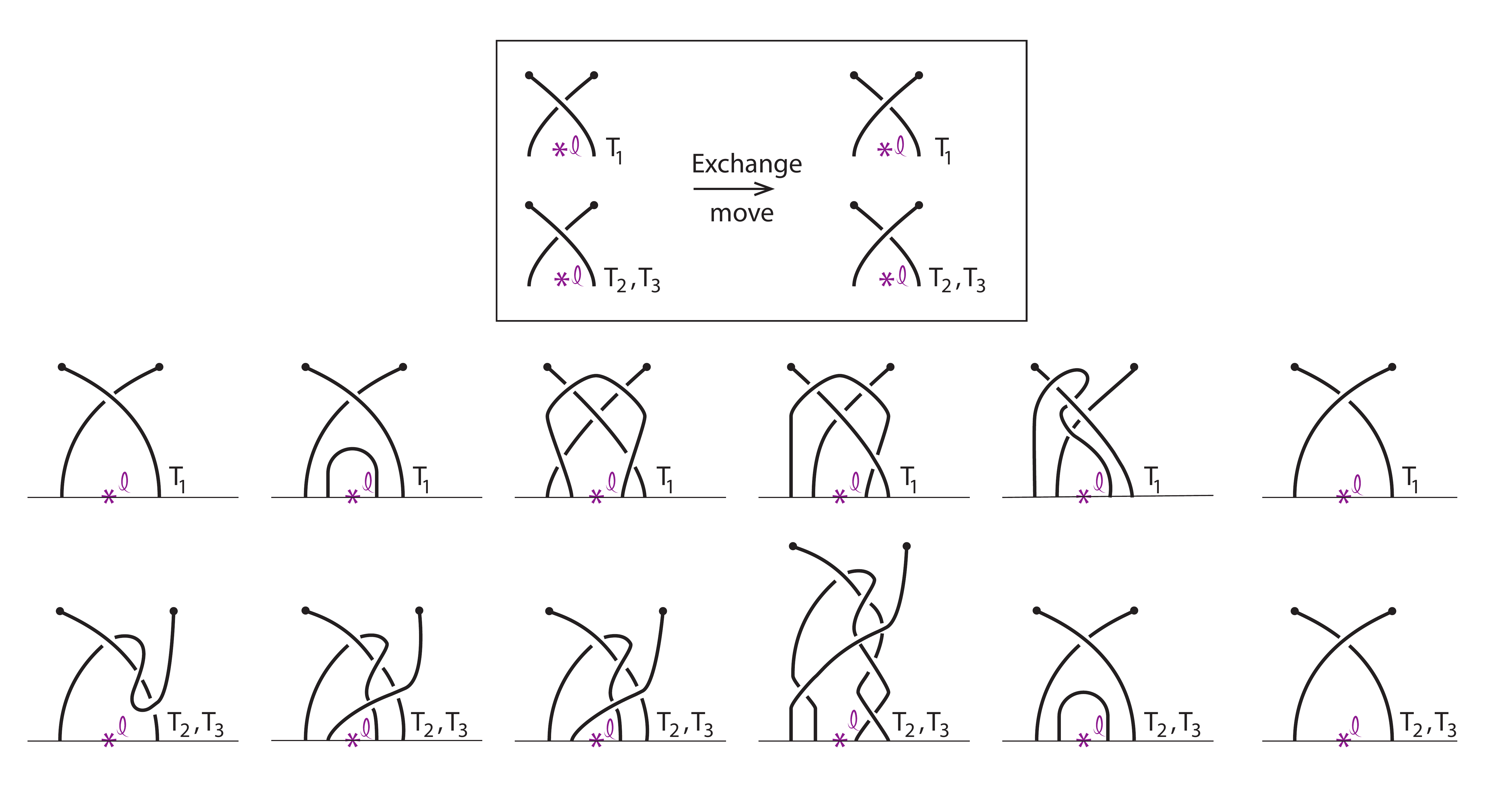}
\caption{(Top) Rainbow version of an exchange move. (Bottom) A sequence of rainbow moves that adds up to an exchange move.}
\label{fig:exchange_move}
\end{figure}

Exchange moves on (weak) rainbow diagrams do not change the underlying bridge trisected surface, as they are achieved by mutual braid moves that interfere with the braid axis. As in dimension three, exchange moves on rainbows replace a sequence of four rainbow moves: braid isotopy, perturbation, braid isotopy, and deperturbation. This sequence is described in Figure~\ref{fig:exchange_move}; the reader may compare it with Figure 13 of \cite{BB_BraidsSurvey}. 

Birman and Menasco used one exchange move to turn Morton's unknot into a fully destabilizable braid. The same can be done with the 4-braid rainbow diagram from Lemma~\ref{lem:unperturbed_rain}. More precisely, one exchange move is enough to make $\Tcal_{{\mbox{\rm Morton}}}$ into a completely decomposable rainbow diagram: rainbow-perturbations of the 1-braid rainbow diagram. Such an exchange move can be extracted from page 588 of \cite{BM_theunlink}. We leave it to the reader to draw the sequence of rainbow moves. The following conjecture is the rainbow version of the Markov Theorem Without Stabilizations for the unlink \cite[Thm 1]{BM_theunlink}.

\begin{conjecture}[MTWS]
Let $\Tcal$ be a rainbow diagram of an $r$-component unlink of 2-spheres. Let $\mathcal{U}_r$ be the crossingless rainbow diagram with $r$ strands. There is a finite sequence of rainbow diagrams 
\[ 
\Tcal=\Tcal_0 \to \Tcal_1 \to \cdots \to \Tcal_m=\mathcal U_r, 
\]
such that each $\Tcal_{k+1}$ is obtained from $\Tcal_k$ by (1) a braided-deperturbation, (2) an exchange move, or (3) interior Reidemeister moves and replacements $(\beta_1,\beta_2,\beta_3)\mapsto(x\beta_1y, x\beta_2y,x\beta_3y)$ for some braid words $x,y$.
\end{conjecture}

\subsection{Artin spun knots}
\begin{lemma}\label{lem:spunknot_rain}
The spin of $K$ admits a rainbow diagram with \mbox{$(3\cdot \braid(K)-2)$} strands. 
\end{lemma}
\begin{proof}
Meier and Zupan described how to turn a $b$-bridge plat projection of a knot into a $(3b-2)$-bridge triplane diagram for its spin surface \cite[Sec 5.1]{MZ17Trans}. Their method is applied to a $b$-bridge plat projection  that is obtained from a $b$-stranded braid of presentation of $K$ as follows. In the left side of Figure~\ref{fig:spunknot_rain1},  the crossingless strands  of the braid closure  are tucked in the back of the braid box. This way,  a plat projection for $K$ is obtained. The corresponding triplane diagram is depicted  on the right side of Figure~\ref{fig:spunknot_rain1}, where  the tangles are braided with respect to the purple axis. Figure~\ref{fig:spunknot_rain2} depicts the tangles $T_i\cup \T_{i+1}$ which are fully destabilizable. Hence, the triplane in Figure~\ref{fig:spunknot_rain1} is a rainbow diagram for the spun of $K$.
\end{proof}
\begin{figure}[ht]
\centering
\includegraphics[width=.7\textwidth]{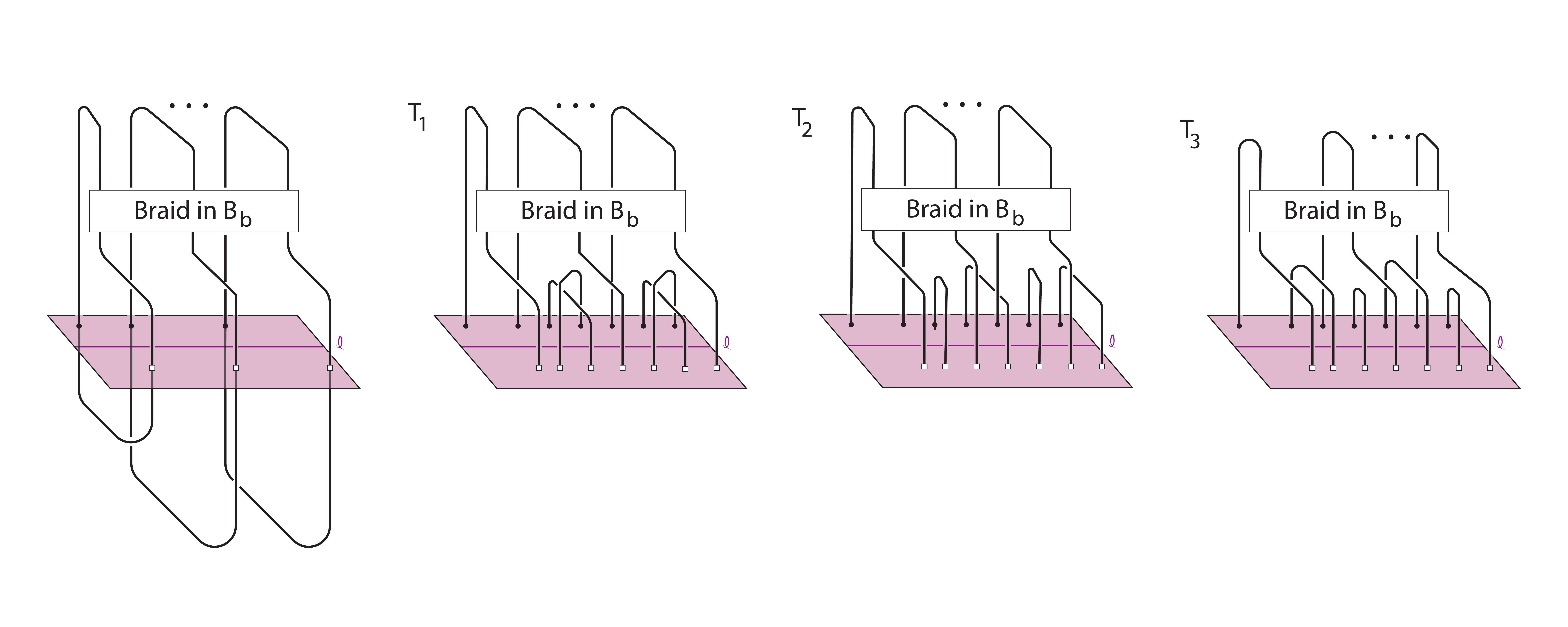}
\caption{(Left) A plat projection of a braided knot $K$ with $b$ strands. The box describes a $b$-stranded braid which goes over $(b-1)$ strands of $K$. (Right) A triplane diagram of the spin of $K$. The purple lines represent the braid axis. Hence, the triplane is ``one mutual braid move away'' from being braided on the nose.}
\label{fig:spunknot_rain1}
\end{figure}

\begin{figure}[ht]
\centering
\includegraphics[width=.7\textwidth]{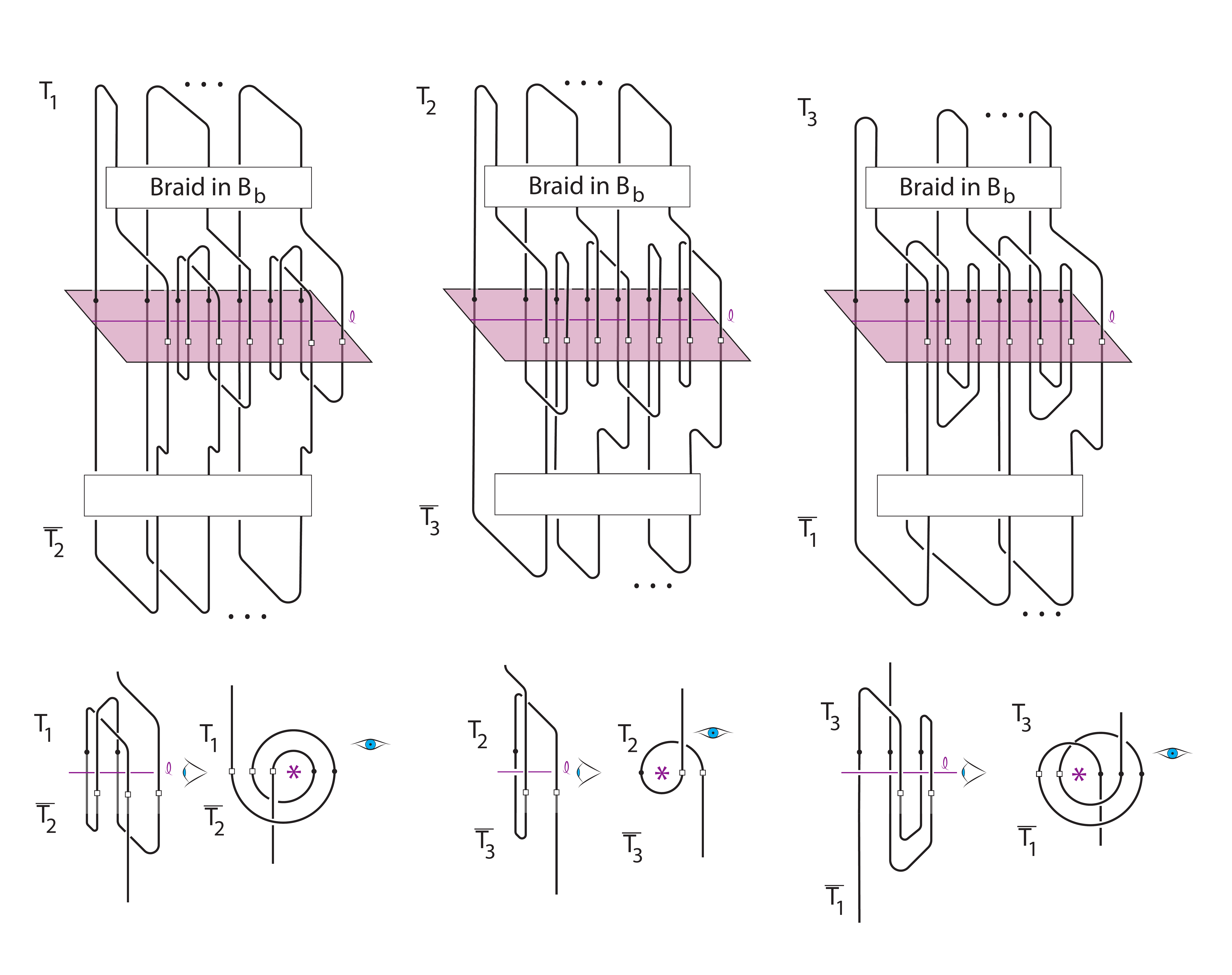}
\caption{Proof-by-picture of why the triplane diagram in Figure~\ref{fig:spunknot_rain1} is a rainbow; that is, each braid $T_i\cup \T_{i+1}$ is fully destabilizable. The bottom rows are close-ups of the short strands in the respective tangles: notice that they are Markov stabilizations of shorter strands. The blue symbols represent the eyes of the reader and how they see the braided strands.}
\label{fig:spunknot_rain2}
\end{figure}

Diao, Ernst, and Reiter studied knots with equal bridge and braid index \cite{diao2021knots}. These are called \emph{BB knots}. The following proposition shows that the spun of some BB knots has the same bridge, weak-rainbow, and rainbow index. 

\begin{proposition}
Let $K$ be a BB knot that satisfies the meridional rank conjecture. Then the following hold for $F=\Scal(K)$,
\[ 
3\cdot \braid(K) -2 = \bridge(F) = \Wrain(F) = \rain(F).
\]
\end{proposition}
\begin{proof}
Lemmas~\ref{lem:mrk_ineq} and~\ref{lem:spunknot_rain}, together with Inequality~\ref{eq:inequality1}, give us the following 
\[ 
3\cdot \mrk(\Scal(K)) -2 \leq \bridge(\Scal(K))\leq \Wrain (\Scal(K)) \leq \rain(\Scal(K))\leq 3\cdot \braid(K)-2.
\]
By Proposition 5.5 of \cite{MZ17Trans}, we know that the left-hand side is equal to $3\cdot \bridge(K)-2$. This is also equal to the right-hand side as $K$ is a BB knot. 
\end{proof}

\subsection{Rainbows of Satellite 2-knots}

This generalizes the well-known construction of satellites for classical links. Let $F$ be an orientable surface. Let $P$ be an embedded surface in $F\times D^2$ and let $C:F\rightarrow S^4$ be an embedding. Consider a diffeomorphism $\wt C:F\times D^2 \rightarrow N(C)$ such that the restriction to $F\times \{0\}$ is $C$. We call the image of $\wt C$ a \emph{satellite knot} $\Sigma(P;C)$ with pattern $P$ and companion $C$. For more information on satellite surfaces, see \cite[Sec 2.4.2]{CKS_Surfaces_in_4space}.

In Appendix~\ref{appen:satellite}, the authors described how to draw a triplane diagram for $\Sigma(P;C)$ when $C$ is a 2-knot. The first step is the observation that surfaces in $S^2\times D^2$ can be described with a triplane diagram $\Pcal$ together with a simple closed curve $\ell$, called an axis \emph{axis}, in the bridge sphere that avoids the endpoints of the tangles of $\Pcal$; see Lemma~\ref{lem:axed_triplanes}. In this situation, one can think of $\ell$ as the binding of a ``half-open book'' of the 3-ball containing each tangle; i.e. there is a diffeomorphism $B\to D^2\times [0,1]$ with $\ell\mapsto \partial D^2\times \{1/2\}$. The second step is to use this product structure on each 3-ball (containing the tangles of the pattern) to insert the tangles of $\Pcal$ into a framed neighborhood of the tangles of a triplane diagram for the companion 2-knot. See Appendix~\ref{appen:satellite} for details. 

In our setup, (weak) rainbow diagrams come equipped with an axis $\ell$. Thus, they already describe embeddings of surfaces in $S^2\times D^2$. The following is a direct consequence of Theorem~\ref{thm:sat} and Lemma~\ref{lem:frames_existence}.

\begin{proposition}\label{prop:rainbow_sat}
    Let $P$ and $C$ surfaces in $S^4$. If $C$ is a 2-knot and $P$ can be described as a weak-rainbow diagram, then $\Sigma(P,C)$ admits a weak-rainbow diagram with $\Wrain(P)\cdot \Wrain(C)$ strands. 
\end{proposition}

\begin{question}\label{ques:rainbow_satel}
    Under what conditions is the (weak) rainbow number multiplicative with respect to satellite operations? In other words, when does the following hold? $$\rain\left(\Sigma(P, C)\right)=\rain(P)\cdot \rain(C).$$
\end{question}

Candidate surfaces to answer Question~\ref{ques:rainbow_satel} are the cable 2-knots. Defined by Kim in \cite{kim20glucktwist}, a \emph{cable 2-knot} is a satellite surface with a 2-knot as a companion and a pattern equal to an unknotted 2-sphere in $S^4$. The phrase ``unknotted companion $P$'' makes sense as $S^4=(S^2\times D^2)\cup (B^3\times S^1)$; thus, patterns in $S^2\times D^2$ can be thought of as patterns in $S^4$ up to isotopy away from a fixed loop $\ell\subset S^4$. Kim showed that the set of cables of a fixed 2-knot $C$ is parametrized by the integers~\cite[Lem 2.8]{kim20glucktwist}. In fact, they are determined by the homotopy class $[l]=m$ in $\pi_1(S^4-N(P))=\Z$. 

\begin{proposition}
    Let $C$ be an oriented 2-knot and let $C_m$ be its $m$-cable for some $m\in \Z-\{0\}$. Then 
    \[ 
    rainbow(C_m)\leq |m| \cdot rainbow(C).
    \]
\end{proposition}
\begin{proof}
    By Proposition~\ref{prop:rainbow_sat}, it remains to show that the $C_m$ has a pattern with rainbow number $|m|$. One can check that the desired triplane diagram $\Pcal_m$ is given by the triplet of elements in the $|m|$-stranded braid group $\left(\sigma_1\sigma_2\cdots \sigma_{|m|-1}, id, id\right)$. 
\end{proof}

\section{Rainbows and triplanes}\label{sec:rainbows_and_triplanes}
We provide three procedures to turn a triplane diagram into a weak-rainbow diagram. Each of these generalizes an idea in 3-dimensional braid knot theory. In Subsection~\ref{sec:Alexanders_method}, we adapt a classical method due to Alexander that braids a knot diagram around a given braid axis. In Subsection~\ref{sec:threading_shadows}, we use shadows and an algorithm due to Morton that finds a convenient braid axis for the knot. In Subsection~\ref{sec:Yamada}, we discuss an algorithm due to Yamada to braid a knot using Seifert circles for which Theorem~\ref{thm:seifertcircles} is a consequence.

\subsection{Alexander's method for triplanes}\label{sec:Alexanders_method}
The idea of Alexander's method to braid knots in $\R^3$ is to manipulate the given knot diagram so that every strand is oriented in a particular direction around a fixed origin (the braid axis). We now describe how to do the same for triplane diagrams. 
For this procedure, we will treat each tangle of $\Tcal$ as a tangle diagram in the upper-half plane with crossing information and boundary on the boundary of the plane. In particular, the boundary of each plane is the same $x$-axis. The goal is to manipulate the triplane diagram so that every strand of $\Tcal$ is oriented in a particular direction around a choice of origin. 

\begin{procedure}[Alexander's method for triplane diagrams]\label{alg:Alex_Method}
$\quad$

\textbf{Step 0.} Choose an \emph{orientation} of $\Tcal$; that is, a consistent choice of labels $\{+,-\}$ for the punctures of $\Sigma$. The strands of $\Tcal$ are oriented so that negative punctures are sources and positive ones are sinks. This is possible as $F$ is orientable \cite[Lem 2.1]{MTZ_cubic_graphs}.

\textbf{Step 1.} Apply mutual braid moves to $\Tcal$ so that all the negative punctures lie on the left of the positive punctures. Mark with a $\star$ a point in the $x$-axis separating the negative punctures from the positive ones. After some Reidemeister moves, the strands of $\Tcal$ are divided by segments which travel in a clockwise or anti-clockwise path around $\star$. In Figure~\ref{fig:Alex2}, we shaded the \emph{bad segments}: those traveling anti-clockwise. 

\textbf{Step 2.} Choose a bad segment from one of the tangles of $\Tcal$. As in Alexander's method, we want to push it through the axis $\star$ so that it is braided. This is achieved by some 0-sector perturbations as in Figure~\ref{fig:Alex2}. Notice that you may need more than one perturbation, as the rest of the tangle may get in the way. 

\begin{figure}[h]
\centering
\includegraphics[width=.7\textwidth]{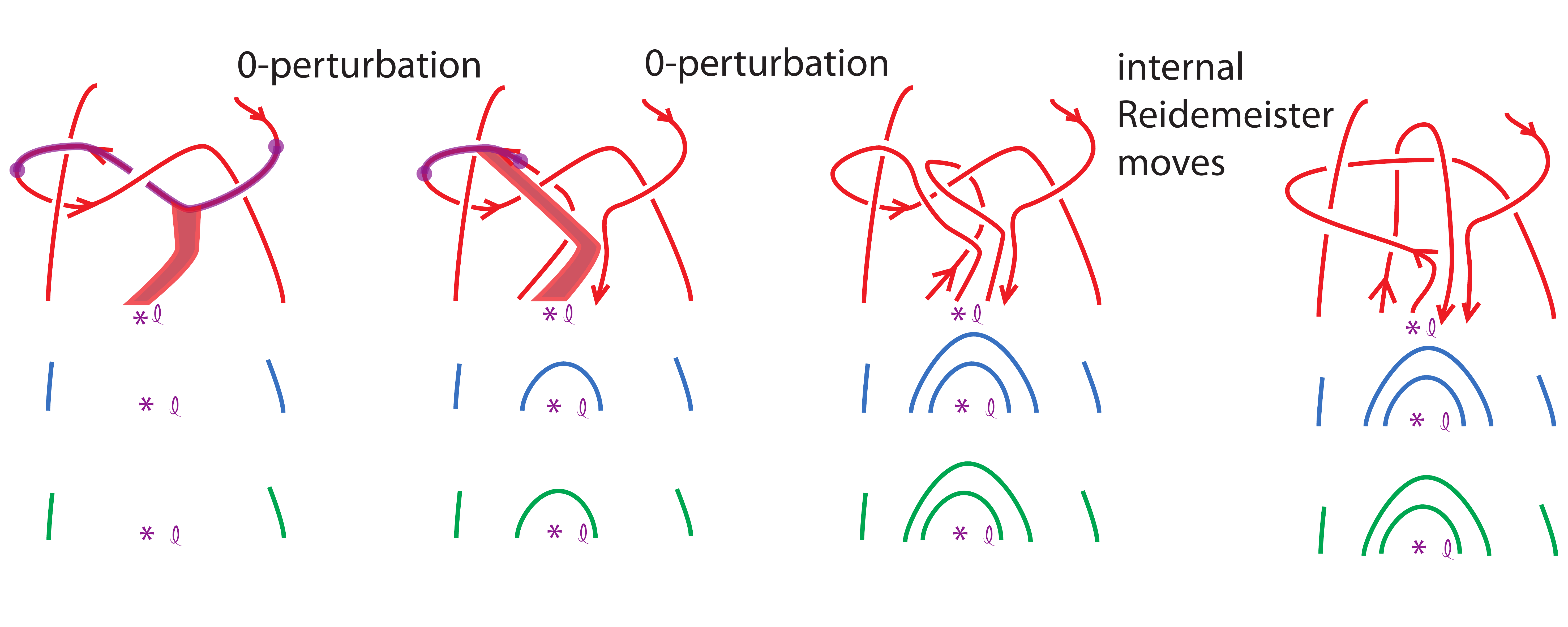}
\caption{How to use 0-sector perturbations to push a bad segment through the origin $\star$. Notice that, in this case, we need two perturbations to achieve our goal.
}
\label{fig:Alex2}
\end{figure}

\textbf{Step 3.} Perform step 2 until the triplane diagram has no bad segments. 
\end{procedure}

\begin{example}[Alexander's Method for $9_1$]
In Figures~\ref{fig:AlexMethod_9_1_part1} and~\ref{fig:AlexMethod_9_1_part2}, we run Procedure~\ref{alg:Alex_Method} to draw a weak-rainbow diagram for Yoshikawa's $9_1$ 2-knot~\cite{Yoshikawa}.
The input, depicted in the left column of Figure~\ref{fig:AlexMethod_9_1_part1}, is a triplane diagram for $9_1$ from \cite{ZupanREU2021}. Steps 0 and 1 are shown in Figure~\ref{fig:AlexMethod_9_1_part1}. In Figure~\ref{fig:AlexMethod_9_1_part2} we perform two 0-sector perturbations (in columns 4-7) to remove the bad segments. 
\end{example}

\begin{figure}[ht]
\centering
\includegraphics[width=.45\textwidth]{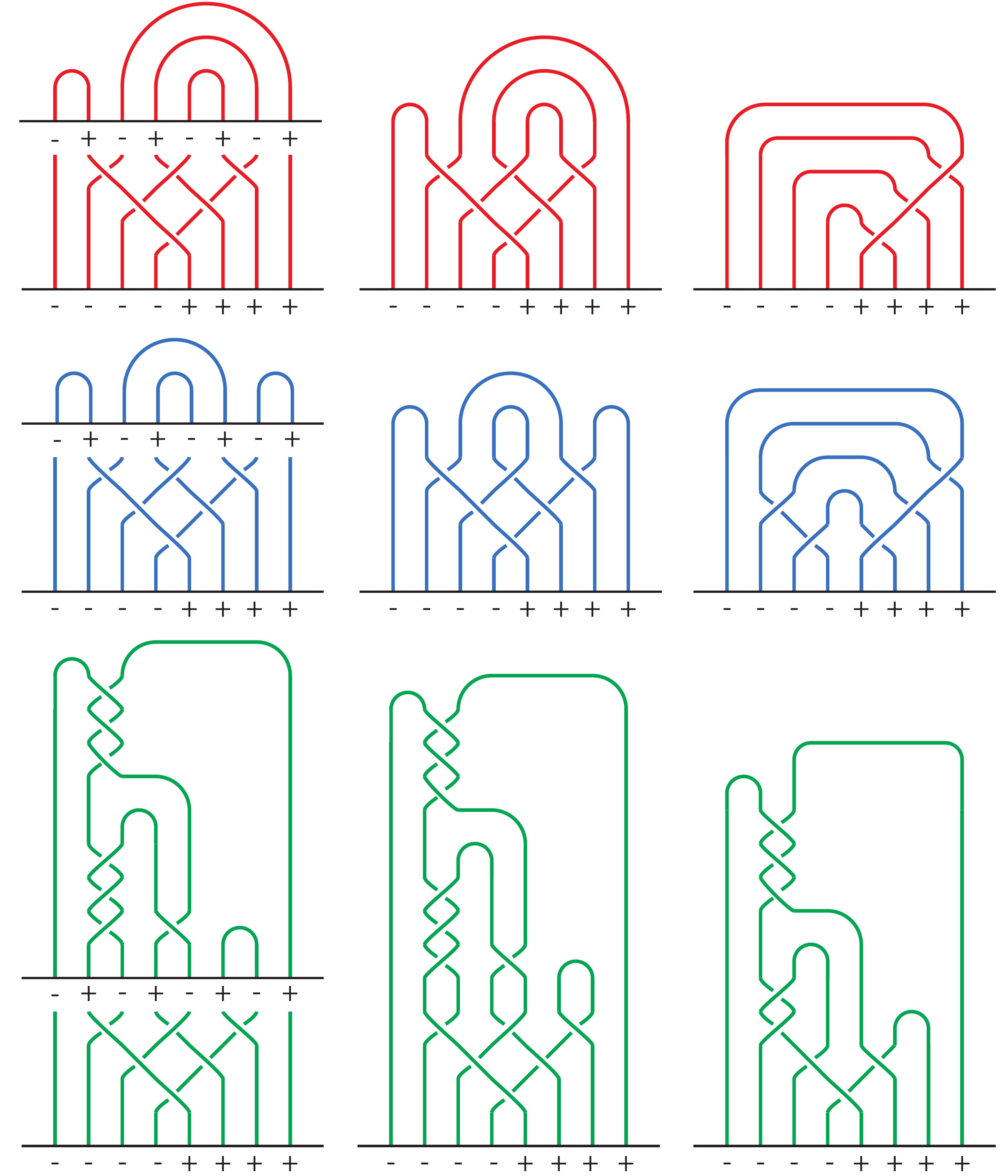}
\caption{Alexander's Method for $9_1$ (part 1/2): From left to right, we performed braid moves to cluster the negative punctures to the left of the positive punctures. The resulting red and blue tangles (right column) are rainbows but the green tangle is not.}
\label{fig:AlexMethod_9_1_part1}
\end{figure}

\begin{figure}[ht]
\centering
\includegraphics[width=.9\textwidth]{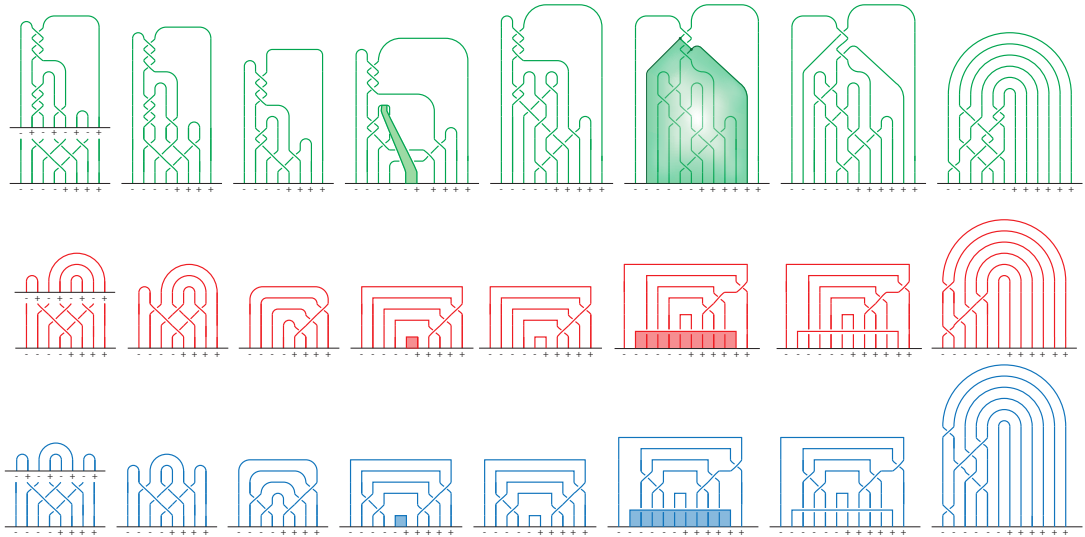}
\caption{Alexander's Method for $9_1$ (part 2/2): two 0-sector perturbations to remove the bad segments in the green tangle. The result (right column) is a weak-rainbow diagram.}
\label{fig:AlexMethod_9_1_part2}
\end{figure}

\begin{proposition}\label{prop:rainbowing_a_triplane}
The result of performing Alexander's method on a triplane diagram of an orientable surface $F$ is a weak-rainbow diagram for $F$.
\end{proposition}
\begin{proof}
Ambient isotopies of $\Sigma$ correspond to mutual braid moves which, together with 0-sector perturbations, do not change the isotopy class of the bridge trisected surface.
\end{proof}

\subsection{Threading shadows}\label{sec:threading_shadows}
In~\cite{Morton_1986}, Morton gave a procedure to braid a knot $K$ in $S^3$ by `threading' a suitable unknotted curve $L$ through the strings of $K$ so that $K$ is braided relative to a braid axis equal to $L$. We can run a similar process for triplane diagrams by pressing the strands of a triplane onto the same bridge sphere. We now describe such a procedure; our presentation resembles that of \cite[Sec 5.4]{knotbook}.

Fix a triplane diagram $\Tcal=(T_1,T_2,T_3)$ for an orientable surface $F$ in $S^4$. A shadow diagram for $\Tcal$ is a tuple $\Scal=\Scal=(s_1,s_2,s_3)$ where each $s_i$ is an embedded collection of pairwise disjoint arcs in a $2b$-punctured sphere $\Sigma$ resulting from pushing the arcs of $T_i$ into the bridge sphere. At the level of shadow diagrams, a 0-sector perturbation of $\Tcal$ breaks one arc of $s_i$ into two new arcs connected by parallel arcs from $s_j$ and $s_k$ as in Figure~\ref{fig:0perturbation_shadows}. 

\begin{figure}[h]
\centering
\includegraphics[width=.4\textwidth]{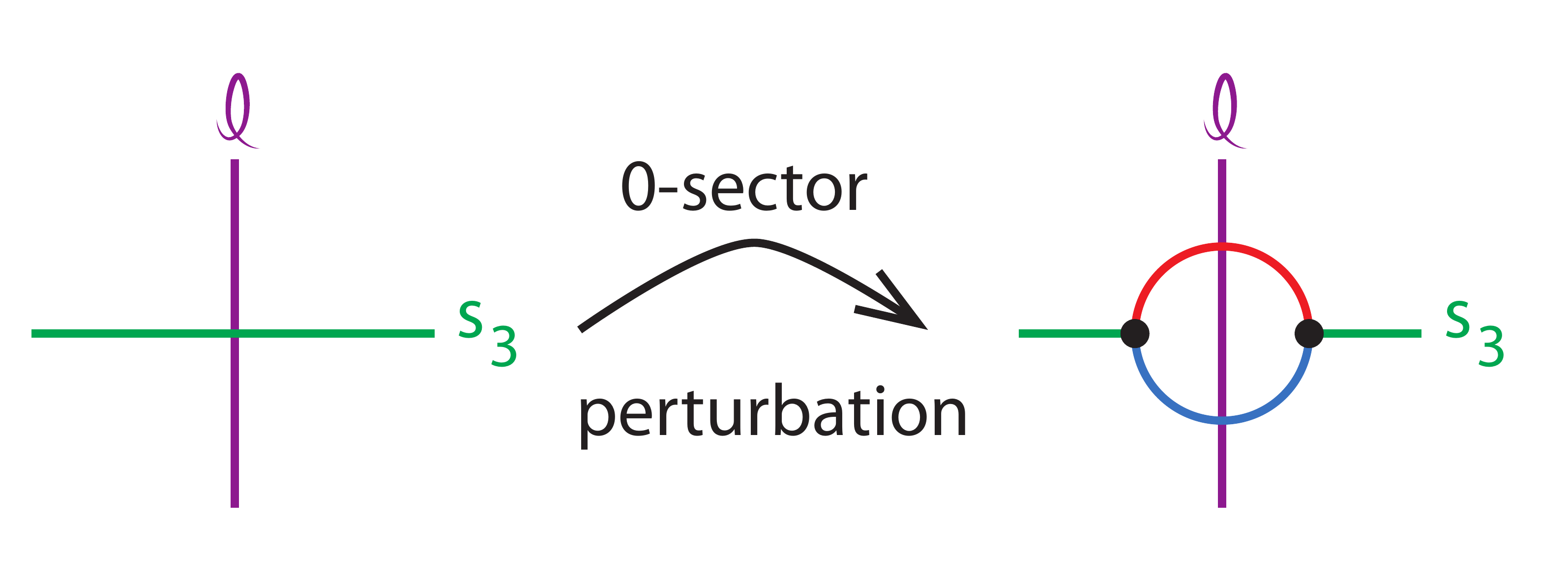}
\caption{The result of 0-sector perturbation on a shadow diagram.
}
\label{fig:0perturbation_shadows}
\end{figure}

\begin{procedure}[Morton's method for triplane diagrams]\label{alg:Morton_Method}
$\quad$

\textbf{Step 0.} Label the punctures of $\Sigma$ with $\{+,-\}$ so that each arc of $\Tcal$ contains both $\pm$. 

\textbf{Step 1.} Find a shadow diagram $\Scal$ for $\Tcal$ and loop $\ell$ in $\Sigma$ disjoint from the punctures such that: (a) each arc of $s_1$ or $s_2$ intersects $\ell$ exactly once, and (b) all punctures labeled with a $+$ lie on the same side of $\ell$. \\
To achieve this step without adding more punctures, one can consider shadows $s_1$ and $s_2$ with disjoint interiors such that $s_1\cup s_2$ form polygonal curves bounding pairwise disjoint polygons in $\Sigma$. Such shadows exist as $T_1\cup \T_2$ is an unlink \cite[Prop 2.3]{MZ17Trans}. In Figures~\ref{fig:shadows1} and~\ref{fig:shadows2}, we provide two possible choices for $\ell$ given the hexagon $s_1\cup s_2$.

\begin{figure}[h]
\centering
\includegraphics[width=.6\textwidth]{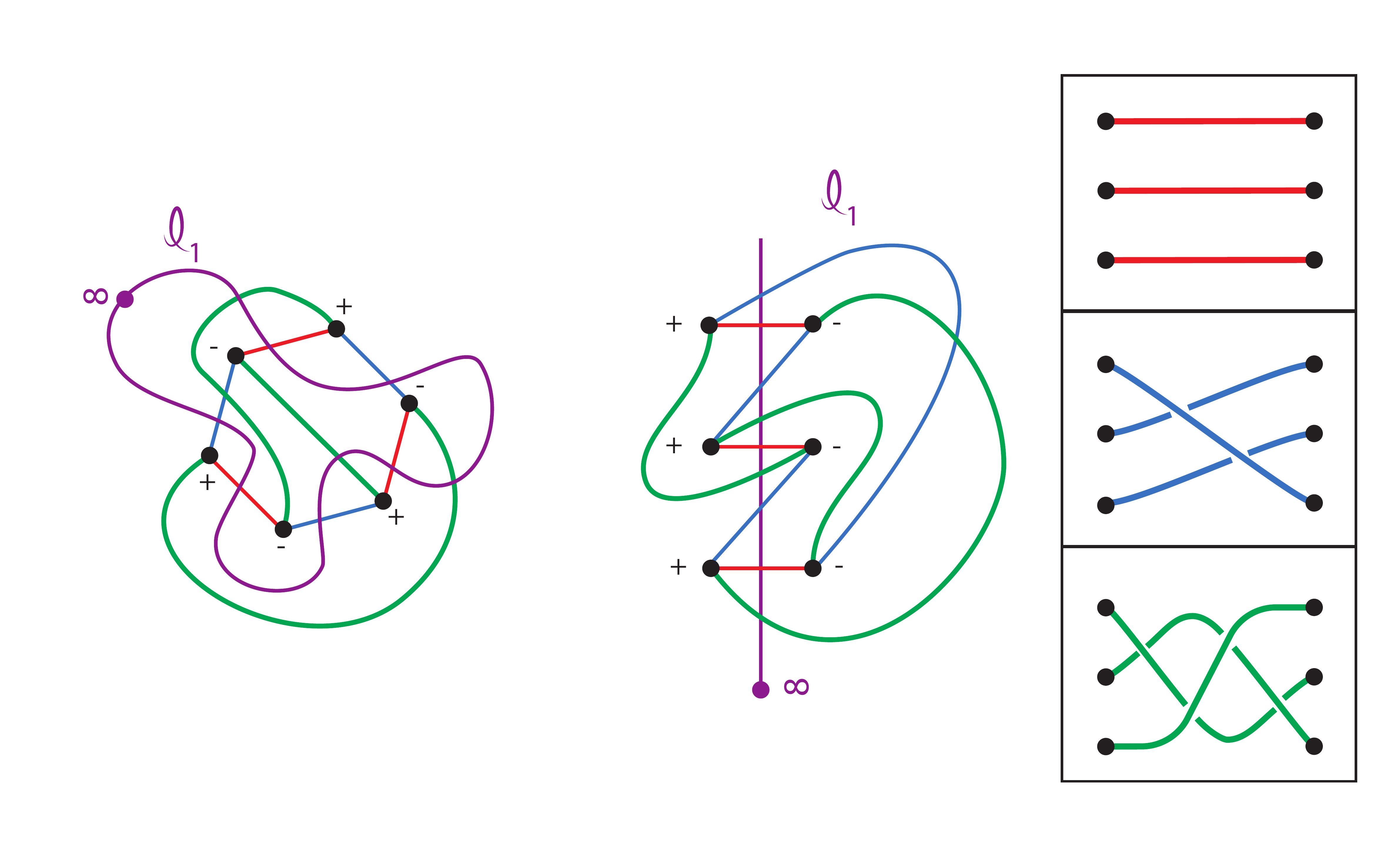}
\caption{Threading a shadow diagram for the unknotted torus using a loop $\ell_1$.  
}
\label{fig:shadows1}
\end{figure}

\textbf{Step 2.} 
Take $x\in s_3$ and enumerate the intersections of $x$ with $\ell$ in order of appearance; notice that $|x\cap l|$ must be a positive odd number. At each of the intersections with even numbering, perform a 0-sector perturbation as in Figure~\ref{fig:0perturbation_shadows}. 
The result is a (possibly new) shadow diagram where each arc of each $s_i$ intersects $\ell$ exactly once.

\begin{figure}[h]
\centering
\includegraphics[width=.8\textwidth]{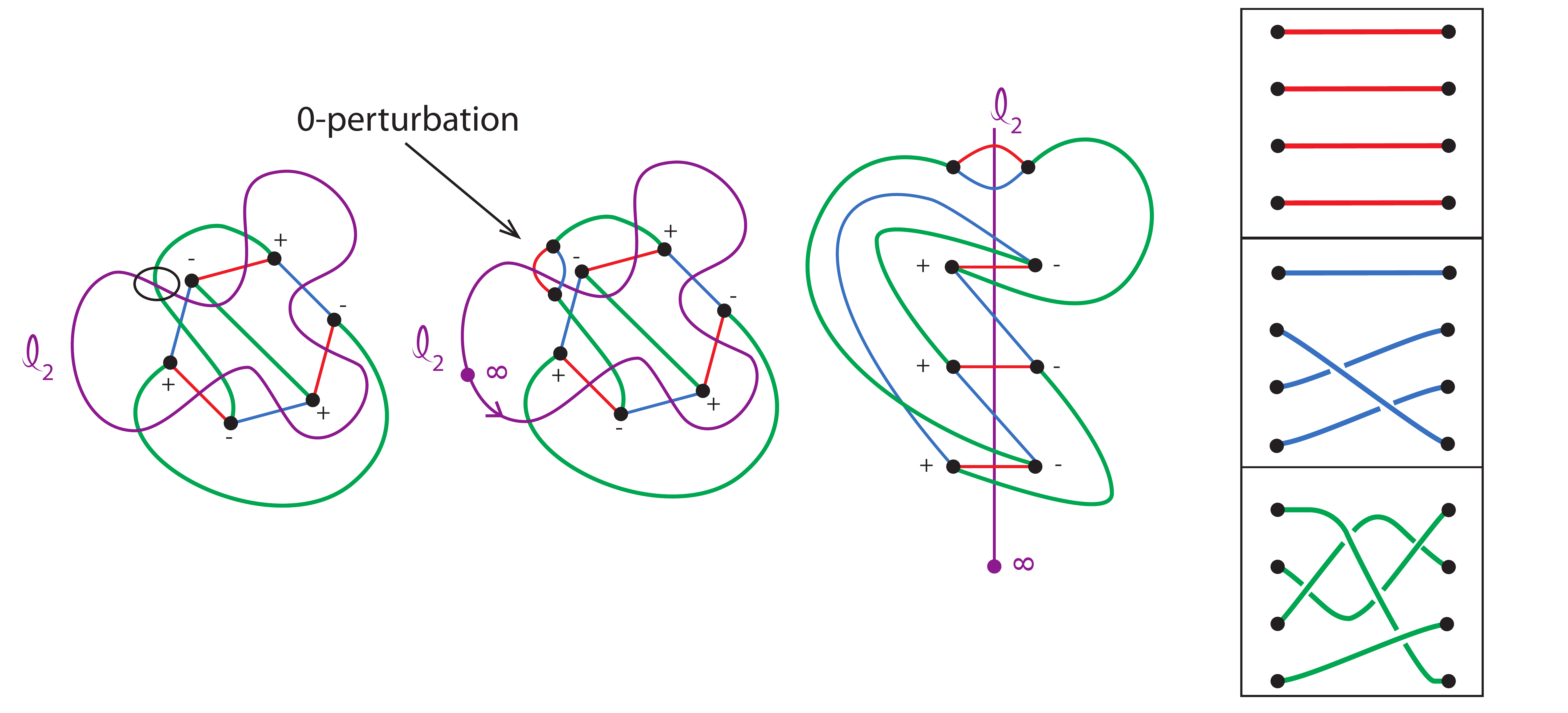}
\caption{Threading a shadow diagram for the unknotted torus using a loop $\ell_2$.
}
\label{fig:shadows2}
\end{figure}

\textbf{Step 3.} Perform an ambient isotopy of $\Sigma$ so that $\ell$ is a straight line passing through $\infty$, where all positive punctures lie on the same side of $\ell$. The resulting shadow diagram can be raised to form a tuple of tangles $\Tcal'=(T'_1,T'_2,T'_3)$ in the 3-ball which are braided with respect to the axis $\ell$.
\end{procedure}

\begin{proposition}
The result of performing Morton's method on a triplane diagram of an orientable surface $F$ is a weak-rainbow diagram for $F$.
\end{proposition}
\begin{proof}
The proof is the same as that of Proposition~\ref{prop:rainbowing_a_triplane}.
\end{proof}

\begin{remark}
In practice, `raising the shadows to rainbows' in Step 3 may be a complicated task. In Figure~\ref{fig_shadows3}, procedures for doing so are outlined. The shadow diagram for the unknotted torus that uses the loop $\ell_2$ is replicated (and reflected vertically) on the left of the figure. In the top illustration, the blue arcs are removed.Then the points labeled $-1, \ldots, -n, +(n+1), \ldots , +(2n)$ are rearranged vertically. The loops formed by the red and green arcs are isotoped during the point rearrangement. The red arcs form an un-braided nested sequence of arcs. Remove the red arcs, and then arrange the green arcs by the braid action on a punctured disk while recording the braiding. The green rainbow is thereby created. 

In the bottom of the illustration, the green arcs are removed, and the process follows an analogous procedure. 
On the far right of the illustration, the diagrams from the right side of Figure~\ref{fig:shadows2} are drawn for comparison.
\end{remark}

\begin{figure}[h]
\centering
\includegraphics[width=1\textwidth]{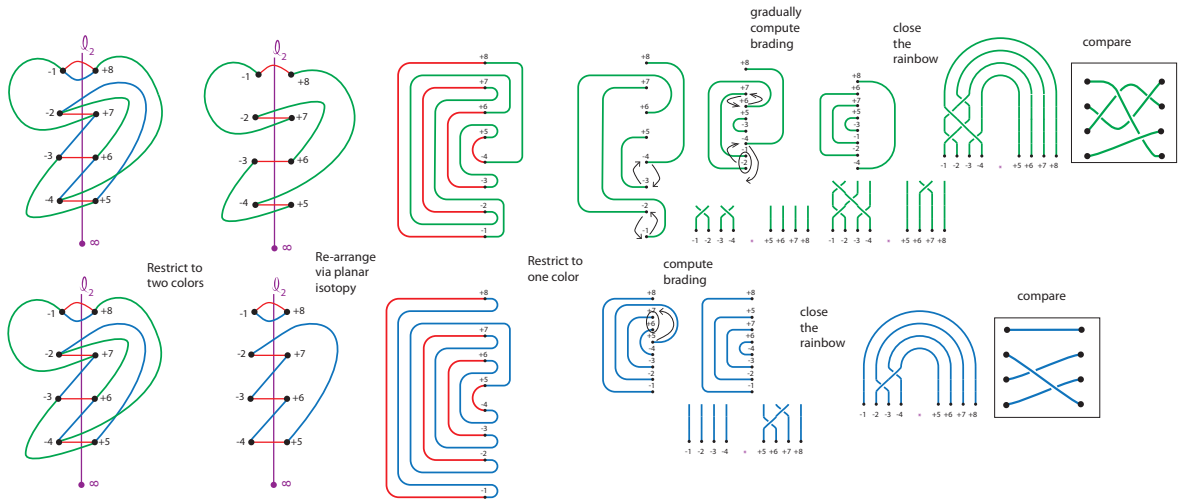}
\caption{How to draw the braid corresponding to the green and blue shadows in Figure~\ref{fig:shadows2}.
The red curves form a nested set of arches. Remove the ``other color" in the shadow. Rearrange the vertices by an isotopy of the graph. Remove the red curve, and apply braid action on the punctured disk until the green, or blue, are nested arches.
}
\label{fig_shadows3}
\end{figure}

\begin{example}[Morton's Method for $9^{0,1}_1$]
In Figures~\ref{fig:Threading_9_01_1_part2} amd ~\ref{fig:Threading_9_01_1_part1} we run Procedure~\ref{alg:Morton_Method} to draw a weak-rainbow diagram for Yoshikawa's $9^{0,1}_1$ surface link~\cite{Yoshikawa}.
The left column of Figure~\ref{fig:Threading_9_01_1_part1}, is a triplane diagram for $9^{0,1}_1$ from \cite{ZupanREU2021}. Columns 2 and 3 of Figure~\ref{fig:Threading_9_01_1_part1} are shadows of the triplane diagram on the left. To aid the reader, Figure~\ref{fig:Threading_9_01_1_part2} shows the intermediate steps to find  the shadows for the green tangle  depicted in item (A) of the figure. Two over-arcs are pushed into the plane and wrap around the endpoints in item (B). Then two more over-arcs are pushed forward while an under-arc is pushed backwards to get to item (C). The last over-arc is also pushed forward to get to item (D). Item (E) indicates the same curves, having lifted the horizontal picture into the plane of the page. Similar, but simpler, processes are applied to the red and blue tangles, and this union appears in item (F).

The vertical line in Figure~\ref{fig:Threading_9_01_1_part1} is the loop $\ell$ through infinity. Notice that each arc of the red/blue shadows intersects $\ell$ once as in Step 1. The black circles around the green shadows mark the even intersections between a green shadow and $\ell$. Columns 3 and 4 depict the result of 0-sector perturbations; these are the shadows for a weak-rainbow diagram. We leave it as an exercise to the reader to draw the corresponding rainbow triplane diagram of $9^{0,1}_1$.
\end{example}

\begin{figure}[ht]
\centering
\includegraphics[width=.8\textwidth]{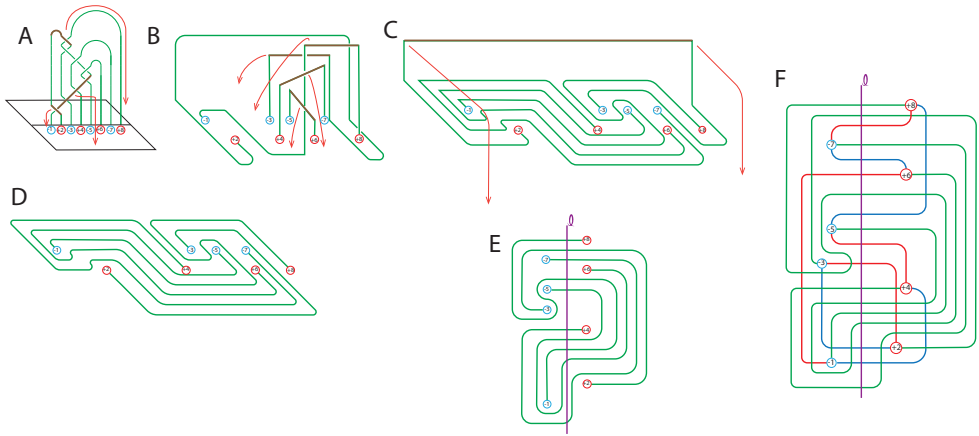}
\caption{Morton's Method for $9^{0,1}_1$ (extra): Drawing the shadows for the green tangle from left column of Figure~\ref{fig:Threading_9_01_1_part1}: we think of the tangle as lying on the vertical plane and slide the over-crossings arcs onto the horizontal plane until the arc surrounds a puncture.}
\label{fig:Threading_9_01_1_part2}
\end{figure}

\begin{figure}[ht]
\centering
\includegraphics[width=.8\textwidth]{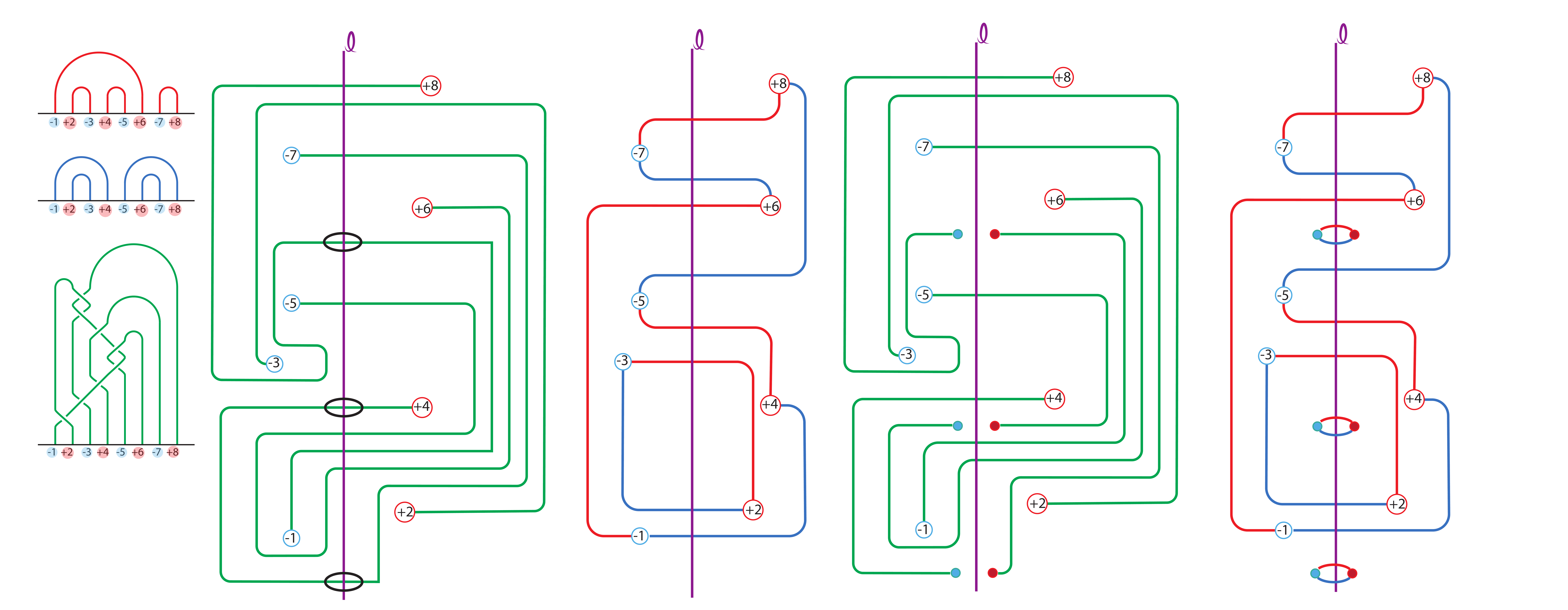}
\caption{Morton's Method for $9^{0,1}_1$: (Columns 2-3) Shadows of the triplane diagram in the left. (Columns 3-4) Shadows after 0-sector perturbations.}
\label{fig:Threading_9_01_1_part1}
\end{figure}

\subsection{Braiding Seifert circles}\label{sec:Yamada}

Let $\Tcal=(T_1,T_2,T_3)$ be an oriented triplane diagram. In this subsection, each $T_i$ will be considered as an oriented tangle diagram in a disk $D_i$ with crossing information and boundary on $\partial D_i$. Each crossing of the tangle $T_i$ can be resolved in an orientation-preserving way to get a collection of embedded oriented 1-manifolds in the disk. There are $b$ properly embedded arcs and some number $s_i$ of Seifert circles; see Figure~\ref{fig:Yamadadiag}. 

In Seifert's algorithm for knots, we remember the location of the crossings of the original link by drawing arcs where the crossings used to be, as in Figure~\ref{fig:Yamadadiag}. For example, in the center of Figure~\ref{fig:Yamadadiag}, the braid $1222\overline{1}$ is replaced by five horizontal line segments. Yamada expanded \cite{YamadaSeifert} on this idea by allowing the arcs to represent any braiding between the strands they touch. This way, the braid $1222\overline{1}$ can also be replaced with one horizontal segment crossing three vertical strands. On the right of Figure~\ref{fig:Yamadadiag}, a schematic appears in which a generic braid is replaced by a thicker segment in the same braid box. 

In this section, we will work with tuples $(\Scal,\Acal)$ where $\Scal$ is a collection of properly embedded oriented 1-manifolds in a disk $D$ with boundary equal to a fixed collection of $2b$ points, and $\Acal$ is a set of embedded arcs in the interior of $D$ such that $\Acal\cap \Scal$ is a finite set of points all the intersections of $a\in \Acal$ with $\Scal$ have the same sign. Imprecisely, if we walk along an arc $a\in \Acal$, every oriented component of $\Scal$ crossing $a$ passes from left to right, or vice versa. See Figure~\ref{fig:move2}(A) for an example of such a tuple $(\Scal,\Acal)$.  
\begin{figure}[ht]
\centering
\includegraphics[width=.7\textwidth]{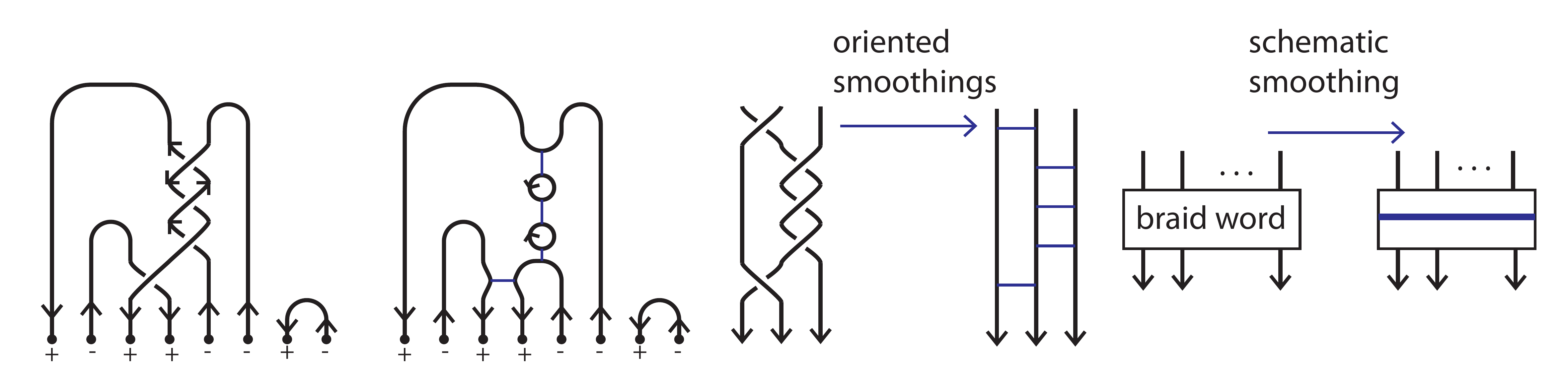}
\caption{Yamada diagram of a tangle.}
\label{fig:Yamadadiag}
\end{figure}

Given $(S,A)$ as above, we can draw many properly embedded tangles in the 3-ball by replacing each $a\in \Acal$ with a braid box between the strands of $\Scal$ it touches, as in Figure~\ref{fig:Yamadadiag}. We will refer to the circle components of $\Scal$ as \emph{Seifert circles}. If a tangle $T$ can be represented with a tuple $(\Scal, \Acal)$, we say that $(\Scal,\Acal)$ is a \emph{Yamada diagram for $T$}. 

We now discuss two ways in which we can trade a Seifert circle with an arc. Let $(\Scal_i, \Acal_i)$ be the Yamada diagram of the $i$-th tangle of $\Tcal$ and let $c\in \Scal_1$ be a Seifert circle. Recall that $\Scal_1\cup \Acal_1$ are subsets of a disk $D_1$. 

\textbf{Move 1.} Suppose that $\rho$ is an arc in the $D_1$ with interior disjoint from $\Scal_1\cup \Acal_1$ that connects $c$ with the boundary of the disk. Then 0-sector perturbation of $\Tcal$ along $\rho$ produces a $(b+1)$-triplane diagram with Yamada diagram having one less Seifert circle; see Figure~\ref{fig:move1}. 

\begin{figure}[ht]
\centering
\includegraphics[width=.7\textwidth]{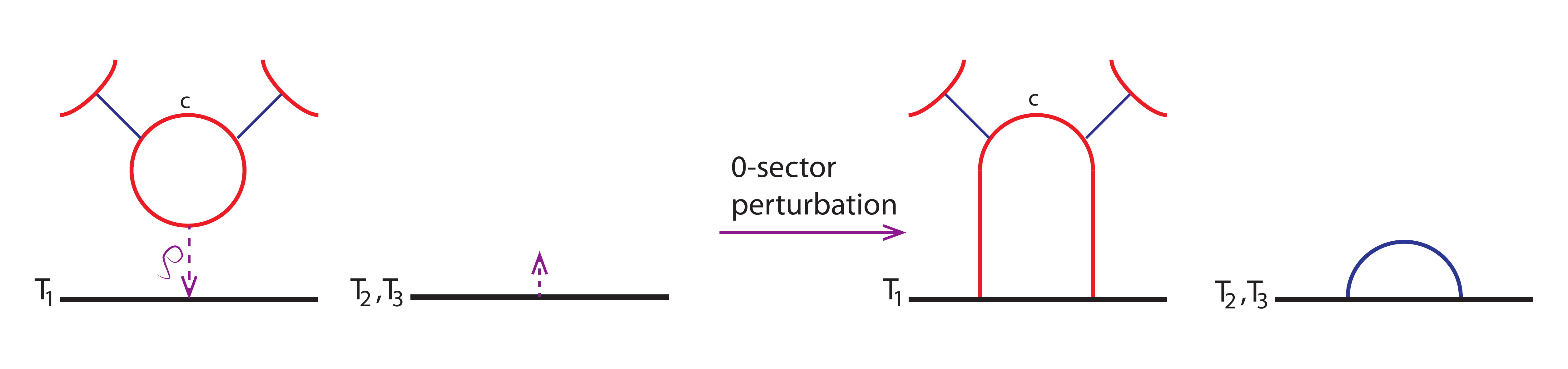}
\caption{Move 1: After a 0-sector perturbation, the Seifert circle $c$ turns into properly embedded arcs; thus, the new Yamada diagram has one less Seifert circle.}
\label{fig:move1}
\end{figure}

\textbf{Move 2.} Let $\rho$ be an arc in the disk with interior disjoint from $\Scal_1\cup \Acal_1$ that connects $c$ with an arc $a\in \Acal_1$. Suppose that, near $\rho$, the segments of $a$ and $c$ have the same orientation. Suppose further that all arcs of $\Acal$ in the bigon of $D_1-a$ not containing $\rho$ are parallel (as oriented arcs) to $a$. In short, assume that $\rho$ connects $c$ to a sub-rainbow of $\Scal_1$ as in Figure~\ref{fig:move2}. In this case, we can slide $\rho$ along the subrainbow containing $a$ to obtain a 0-sector perturbation arc as in the figure. Note that the subrainbow condition forces the orientation of the punctures of $T_1$ (and so $T_2$ and $T_3$). After the perturbation, we perform a braid move as in Figure~\ref{fig:move2}. After Reidemeister moves, we resolve the new crossings in $T_2$ and $T_3$, and we obtain new Yamada diagrams $\{(\Scal'_i, \Acal'_i)\}_{i=1}^3$ for $\Tcal'$ with one less Seifert circle than $\{(\Scal_i, \Acal_i)\}_{i=1}^3$. Notice that the arcs $x\in \Acal_1$ turn into arcs $x\in \Acal'_1$ that now cross the old strands of $c$ in the same way that $a$ does. This is true as the orientations of these strands agree with those already adjacent to the arc $x\in \Acal_1$. 

\begin{figure}[ht]
\centering
\includegraphics[width=1\textwidth]{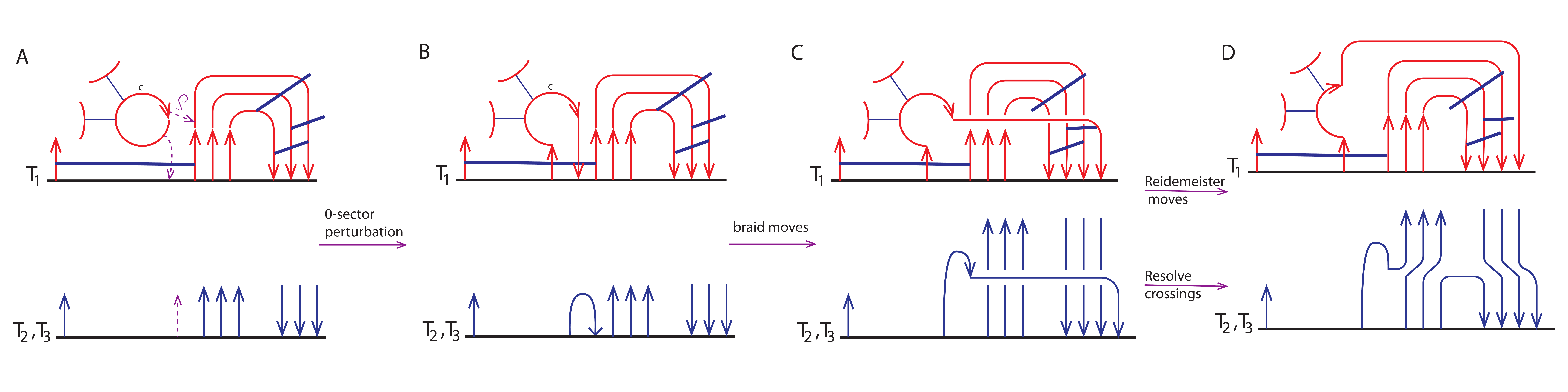}
\caption{Move 2: combination of triplane moves that trade a Seifert circle $c$ with a new bridge arc parallel (as an oriented arc) to $a$.}
\label{fig:move2}
\end{figure}

\begin{definition}
If $\{(\Scal_i,\Acal_i)\}_{i=1}^3$ admits arcs as in moves 1 or 2 above, we say that $\Tcal$ admits a simplification move of type 1 or 2, respectively. 
\end{definition}

We say that an oriented triplane diagram is \emph{clustered} if all the negative punctures are separated from the positive punctures. When the tangles are drawn in the upper-half plane, the signs of punctures are arranged like $\left(p,p,\dots, p, q, q, \dots, q, p, p, \dots,p\right)$ where $\{p,q\}=\{-,+\}$. 
Denote by $\Scal_i^A$ the arcs of $\Scal_i$. The clustered condition on the triplane diagram $\Tcal=(T_1,T_2,T_3)$ forces each Yamada diagram $(\Scal_i,\Acal_i)$ of $T_i$ to be \emph{rainbow-like}: arcs of $\Scal_i$ are pairwise parallel (as oriented arcs) and form a crossingless braid. Thus, the Seifert circles of $\Scal_i$ lie between two consecutive arcs of $\Scal_i^A$ or in an outermost bigon of $D-\Scal_i^A$. 

\begin{lemma}\label{lem:existance_simplification}
If a clustered triplane diagram has a Seifert circle, then it admits a simplification move.
\end{lemma}
\begin{proof}
Suppose that $\Scal_1$ has Seifert circles. As $\Tcal$ is clustered, the circles lie inside one of the bigons of $D_1-\Scal_1^A$ or in one of the 4-gons. In the former case, we can find a simplification of type 1. 

Assume that there are Seifert circles between consecutive arcs $a_1$ and $a_2$ of $\Scal_1^A$. The region of $D$ cut off by $a_1\cup a_2$ is a 4-gon with two opposite sides $a_1$ and $a_2$ oriented the same way. In particular, any circle in such a 4-gon must have the same orientation as exactly one of $a_1$ or $a_2$. Now, consider a circle in this 4-gon that is closest to the unlabeled boundaries. This circle has paths $\rho_1$ and $\rho_2$ to each $a_1$ and $a_2$ such that $int(\rho_i)\cap (\Scal_1\cup \Acal_1)=\emptyset$. One of these arcs provides the simplification move of type 2. 
\end{proof}

\begin{procedure}[Yamada's method for triplane diagrams]\label{alg:Yamada_algorithm}
$\quad$ 

\textbf{Step 0.} Choose an orientation of $\Tcal$.

\textbf{Step 1.} Apply mutual braid moves to cluster all the negative punctures, away from the positive ones. This way, the triplane diagram is clustered. By Lemma~\ref{lem:existance_simplification}, if $\Tcal$ has Seifert circles, then it admits a simplification move. In fact, the move is of type 2 if a Seifert circle lies in a 4-gon of some $D_i-\Scal_i^A$, and of type 1 if there are Seifert circles inside one of the bigons. 

\textbf{Step 2.} Apply a simplification move to $\Tcal$. If the resulting triplane is still clustered, go to the next step. If the new triplane is not clustered, then the move was of type 1 with the circle and rainbow oriented as in Figure~\ref{fig:Seifert4}(A). As depicted in the parts (B)-(D) of the figure, after one braid move, the triplane is now clustered with a Yamada diagram having one less Seifert circle than in the previous step. 

\textbf{Step 3.} Repeat Step 2 until there are no Seifert circles. 
\end{procedure}

\begin{figure}[ht]
\centering
\includegraphics[width=.7\textwidth]{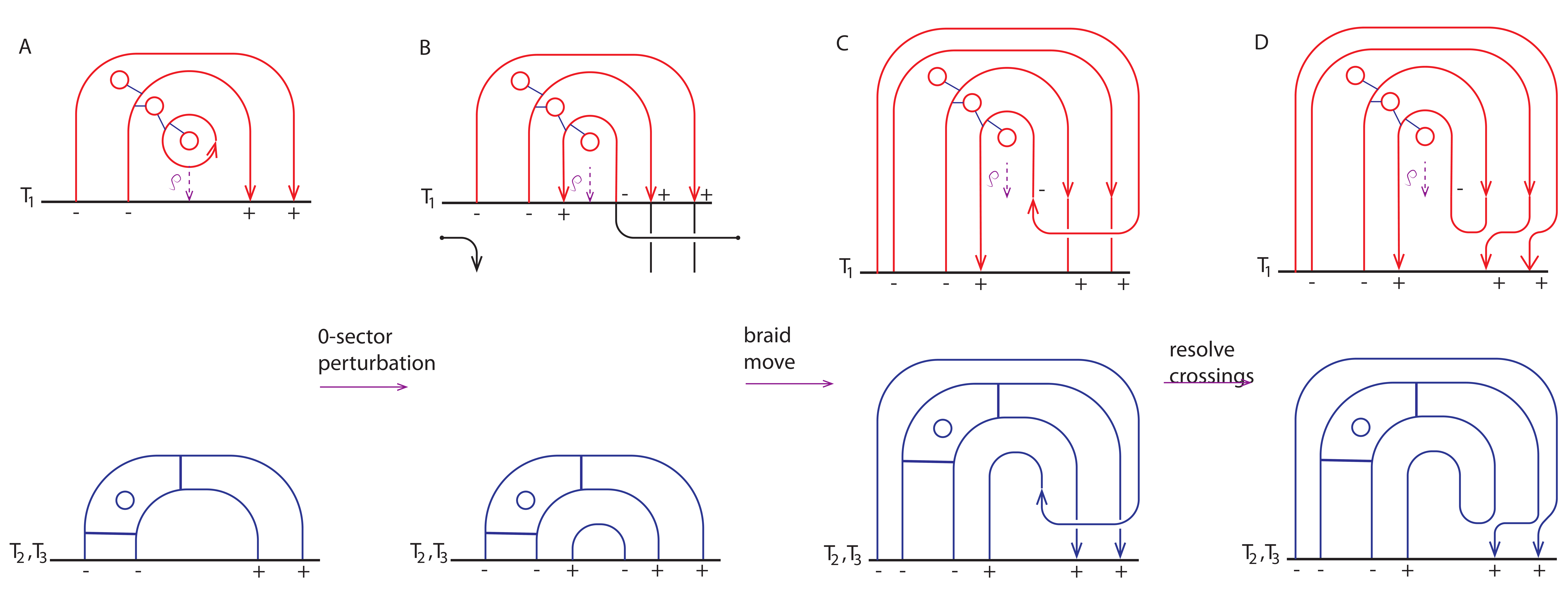}
\caption{(A) Type 1 simplification that does not preserve the property of being clustered. To fix this, we perform more triplane moves (B)-(D).}
\label{fig:Seifert4}
\end{figure}

\begin{example}[Yamada's Method for $9^{0,1}_1$]\label{example:yamada_9_01_1}
In Figures~\ref{fig:Yamada_9_01_1_part1}, \ref{fig:Yamada_9_01_1_part2}, and \ref{fig:Yamada_9_01_1_part3} we run Procedure~\ref{alg:Yamada_algorithm} to draw a weak-rainbow diagram for Yoshikawa's $9^{0,1}_1$ surface link~\cite{Yoshikawa}.
The input, depicted in the left column of Figure~\ref{fig:Yamada_9_01_1_part1}, is a triplane diagram for $9^{0,1}_1$ from \cite{ZupanREU2021}. In columns 2 and 3 of Figure~\ref{fig:Yamada_9_01_1_part1} we perform braid moves to cluster the negative endpoints to the left of the positive endpoints. The resulting red and blue tangles are rainbows but the green tangle is not; the rightmost figure is a Yamada diagram for the green tangle. To remind ourselves of the sign of the crossings we resolved, we use the convention of black arcs for positive crossings and purple arcs for negative crossings.  In Figure~\ref{fig:Yamada_9_01_1_part2}, we found an arc $\rho$ corresponding to a simplification move of type 2, which we perform in the following columns of the figure. The resulting Yamada diagram has one remaining Seifert circle. In Figure~\ref{fig:Yamada_9_01_1_part3} we found an arc $\rho$ corresponding to a simplification move of type 2, which we perform in the remaining columns of this figure. We recover the green tangle using the convention for the signs of the crossings. The resulting triplet of tangles (red, blue, green) is a weak-rainbow diagram for $9^{0,1}_1$.
\end{example}

\begin{figure}[ht]
\centering
\includegraphics[width=.7\textwidth]{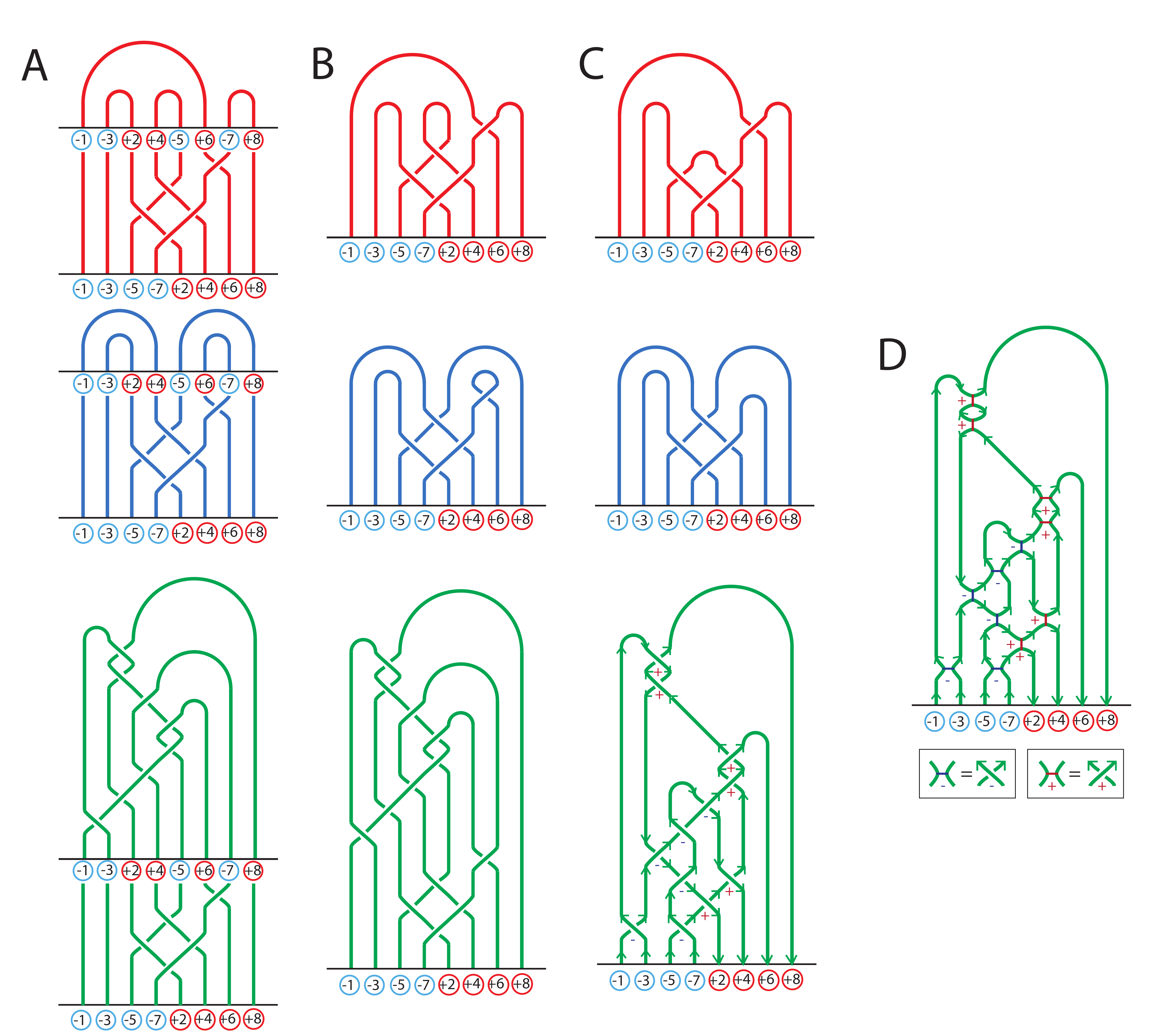}
\caption{Yamada's Method for $9^{0,1}_1$ (part 1/3): After braid moves to cluster the negative endpoints to the left of the positive endpoints, the resulting red and blue tangles are rainbows but the green tangle is not. (Right) A Yamada diagram for the green tangle.}
\label{fig:Yamada_9_01_1_part1}
\end{figure}

\begin{figure}[ht]
\centering
\includegraphics[width=.8\textwidth]{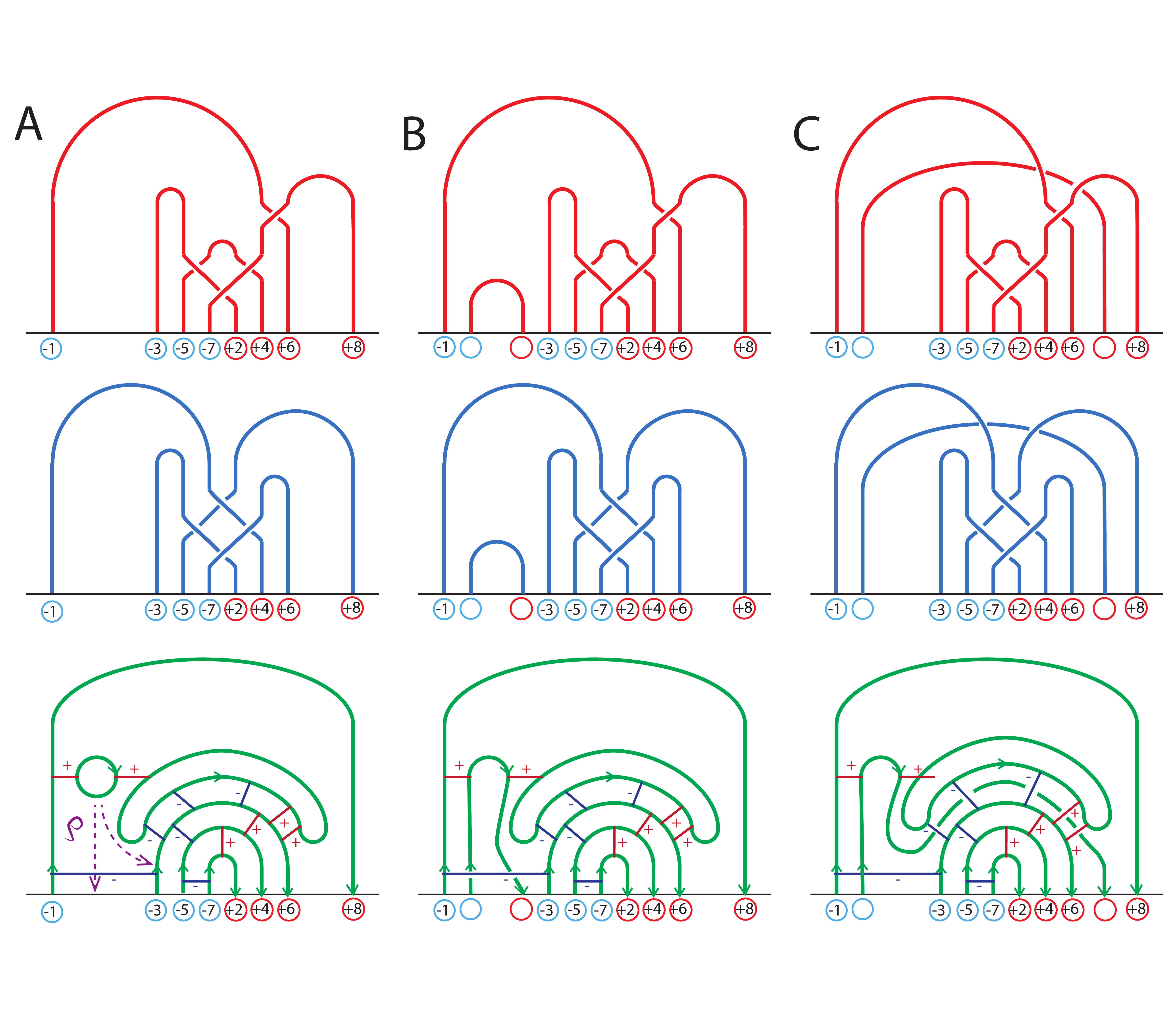}
\caption{Yamada's Method for $9^{0,1}_1$ (part 2/3): we found an arc $\rho$ corresponding to a simplification move of type 2, which we perform in the following columns.}
\label{fig:Yamada_9_01_1_part2}
\end{figure}

\begin{figure}[ht]
\centering
\includegraphics[width=.8\textwidth]{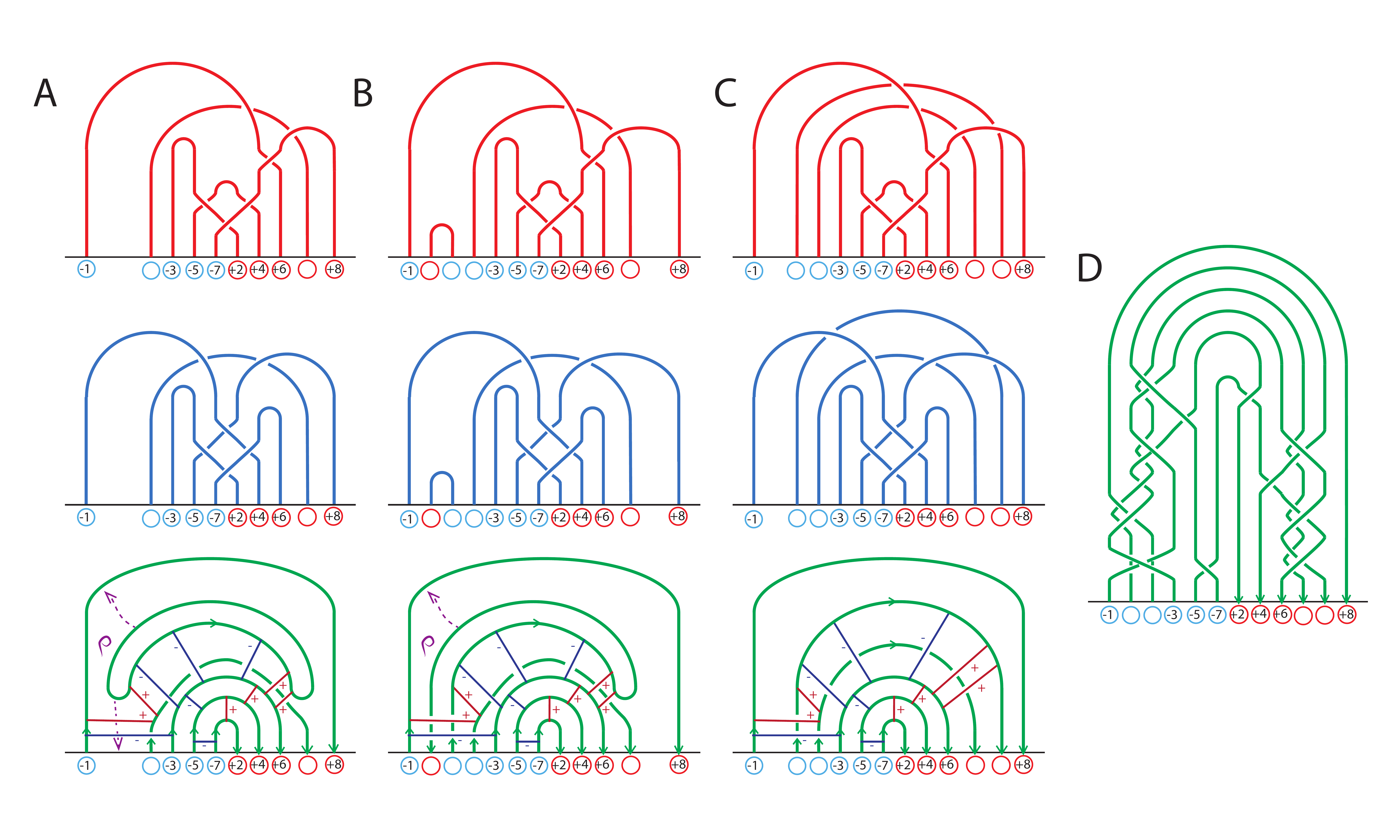}
\caption{Yamada's Method for $9^{0,1}_1$ (part 3/3): we found an arc $\rho$ corresponding to a simplification move of type 2, which we perform in the following columns. The result, rightmost (red, blue, green) tangles form a weak-rainbow diagram.}
\label{fig:Yamada_9_01_1_part3}
\end{figure}
\clearpage

\begin{proposition}\label{prop:Yamada_alg}
The result of performing Yamada's method on a $b$-bridge triplane diagram of an orientable surface is a weak-rainbow diagram for $F$ with $\left(b+\sum_i s_i\right)$ bridges.
\end{proposition}
\begin{proof}
Each simplification move is a combination of triplane moves that increase the bridge index by one. 
\end{proof}

The following result follows directly from Proposition~\ref{prop:Yamada_alg} and the observation that only rainbows have Yamada diagrams with no Seifert circles.  

\begin{theorem}\label{thm:seifertcircles}
Let $F$ be an orientable surface in $S^4$. Then, 
\[ \Wrain(F) = \min\left\{b(\Tcal)+\sum_i s_i(\Tcal): \Tcal \text{ is clustered triplane for } F\right\}, \]
where $s_i\left(\Tcal\right)$ is the number of Seifert circles of the $i$-th tangle of $\Tcal$.
\end{theorem}

\section{Rainbows and braid movies} \label{sec:rains_and_movies}
\emph{Braid movies}, as their name suggests, are surfaces in $\R^4$ traced by movies of braids containing braid isotopies and saddle cobordisms between $\sigma_i^{\pm1}$ and the identity braid. See Figure~\ref{brokenAndChart} for broken  surface diagram descriptions of local braid movies or see Section 14 of \cite{KamBook} for more details. Kamada introduced a \emph{normal braided form} for braid movies \cite[Sec 21]{KamBook}. In short, these are banded unlink diagrams where the unlink is braided and the band is `short'; so that band surgery corresponds to adding one standard braid generator $\sigma_i^{\pm1}$ to the braided unlink. In this section, we will work with an alternative/weaker normal braided form to allow our bands to be conjugations of $\sigma_i^{\pm 1}$. These longer bands, referred to as \emph{pre-simple} in \cite[Sec. 21]{KamBook}, will enable us to move our bands around while preserving the original braided unlink. In a sense, this subtle difference is that between positive and quasi-positive knots. 

\begin{definition}\label{def:braided_bands}
Let $L$ be a braid in $S^3$ with respect to the braid axis $\ell$. The complement of $\ell$, $S^3-\ell$ is foliated by open disks with boundary equal to $\ell$; we call them \emph{disk pages}. A \emph{braided band} or \emph{pre-simple band} is an embedded strip $v:[-1,1]\times [-1,1]\to S^3$ such that,
\begin{enumerate}
    \item $Im(v)\cap L = v\left( [-1,1]\times \{-1,1\}\right)$, 
    \item \emph{the core of $v$}, equal to $v\left(\{0\}\times [-1,1]\right)$, lies in a disk page of $\ell$, and 
    \item \emph{the band surgered link} $L[v]:=\left[L-v\left( [-1,1]\times \{-1,1\}\right)\right] \cup v\left( \{-1,1\}\times [-1,1]\right)$ is braided with braid axis equal to $\ell$. 
\end{enumerate}
\end{definition}
It follows from the definition that $v$ is equal to a half-twisted band along its core. Thus, one can describe a braided band with (1) an arc in a disk-page of $\ell$ that connects different braid points, and (2) a sign, $+$ or $-$, depending on the band twisting counter-clockwise or clockwise when traversing its core. For example, the sign in the higher band  on the right of Figure~\ref{fig:spuntref_rain_to_band} is positive while the lower band is negative. Observe that $L[v]$ is a new braid obtained from $L$ by inserting a sub-braid of the form $w\sigma_i^{\pm1}w^{-1}$. 

\begin{definition}
    A \emph{braided banded unlink diagram} is a pair $(L,\nu)$ satisfying 
    \begin{enumerate}
        \item $L$ is a fully destabilizable braided unlink, 
        \item $\nu$ is a collection of pairwise disjoint braided bands for $L$, and 
        \item $L[\nu]$ is a fully destabilizable braided unlink. 
    \end{enumerate}
\end{definition}

Braided banded unlink diagrams describe braid movies; in fact, they determine braided surfaces in four dimensions. Every orientable surface in $S^4$ admits a braided banded unlink diagram~\cite[Sec 21]{KamBook}. Denote by $F(L,\nu)$ the surface in $S^4$ determined by a braided banded unlink diagram $(L,\nu)$. In this section, we will describe how to go back and forth between braided banded unlink diagrams and rainbow diagrams. We will need the following technical lemma. 

\begin{lemma}\label{lem:existence_bands_for_unlink}
Let $L$ be a $c$-component fully destabilizable $b$-strand braid and let $U$ be the crossingless braid with $b$-components. There is a collection $\nu$ of $(b-c)$ braided bands such that $U[\nu]$ is braid isotopic to $L$. Moreover, the surface in $F(U,\nu)\subset B^4$ is a trivial disk system.
\end{lemma}
\begin{proof}
Recall that a Markov stabilization of a braid $B$ corresponds to adding one braided band to the link $B\cup U'$ where $U'$ is a small meridian of the braid axis. The result follows by induction on $(b-c)\geq 0$.
\end{proof}

\subsection{Braid movie from a rainbow diagram}
In particular, a rainbow can be turned into a banded unlink diagram.

\begin{lemma}\label{lem:existence_bands_for_tangle}
Let $\Tcal$ be a $(b;c_1,c_2,c_3)$-rainbow diagram for an oriented surface $F\subset S^4$. After mutual braid moves that preserve the braid axis of $\Tcal$, there exists a collection $\nu$ of $(b-c_2)$ braided bands for the tangle $T_2$ such that $T_2[\nu]$ is isotopic to $T_3$ rel boundary. 
\end{lemma}
\begin{proof}
We describe the process of finding such a set of bands. We will denote by $x,y,z\in B_b$ the elements of the braid group describing $T_1$, $T_2$, and $T_3$, respectively. First perform mutual braid moves on $\Tcal$ so that $T_2$ is a crossingless braid; that is, $y=id$. As $T_3\cup \T_2$ is fully destabilizable and by Lemma~\ref{lem:existence_bands_for_unlink}, there exist a collection $\nu$ of $(b-c_2)$ braided bands such that $U[\nu]$ and $T_3\cup \T_2$ are braid isotopic. We can represent the bands in $\nu$ by the group element $v\in B_b$; $v$ is the product of $(b-c_2)$ conjugations  of distinct standard generators of $B_b$. Recall that two braids are conjugate if and only if their closures are isotopic away from the braid axis \cite{morton1978infinitely}. Thus, since $T_3\cup \T_2$ and $U[\nu]$ are the braid closures of $z$ and $v$, respectively, there is a braid element $w\in B_b$ such that $\overline{w}zw=v$. 

We use $w$ from above to perform a mutual braid move to $\Tcal$ as follows: we concatenate two copies of $w$, one to the $b$ left-most punctures and the other to the $b$ right-most punctures. See Figure~\ref{fig:existance_bands} for reference. The new tangles $T'_1$, $T'_2$, and $T'_3$ are described by the braid elements $\overline{w}xw$, $\overline{w}w=id$, and $\overline{w}zw=v$, respectively. The last equation implies that $T'_3$ is equal to $T'_2[\nu]$. 

\begin{figure}[ht]
\centering
\includegraphics[width=.7\textwidth]{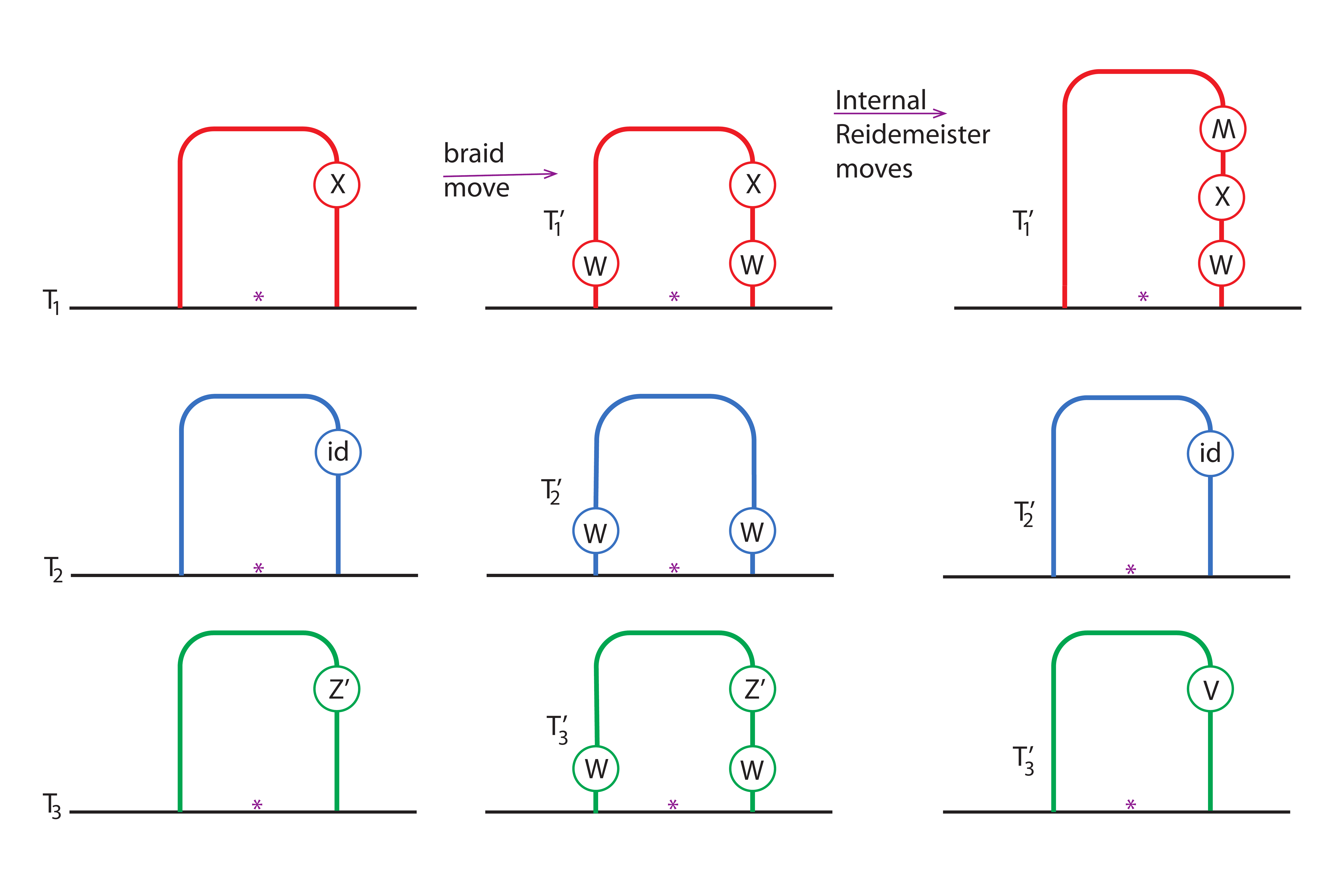}
\caption{We can think of a braided tangle as described by an element of the braid group $B_b$. Then, mutual braid moves of the form $(w,w)$ correspond to conjugation.}
\label{fig:existance_bands}
\end{figure}
\end{proof}

\begin{procedure}[Braid movie from rainbow]\label{alg:movie_from_rainbow}
Let $\Tcal$ be a weak-rainbow diagram for a surface $F\subset S^4$. 

\textbf{Step 1.} Perform enough Markov stabilizations to $\Tcal$ to obtain a rainbow diagram; Lemma~\ref{lem:strong_rainbow}.

\textbf{Step 2.} Find a collection of braided bands $\nu$ for the tangle $T_2$ satisfying that $T_2[\nu]=T_3$, as braided tangles. This is possible by Lemma~\ref{lem:existence_bands_for_tangle}.

\textbf{Step 3.} Consider the $b$-stranded braid $L=T_1\cup \T_2$ and let $\omega=\overline{\nu}$ be braided bands for $L$. 
\end{procedure}

\begin{example}[Spun trefoil]
Figure~\ref{fig:spuntref_rain_to_band} depicts the output of Procedure~\ref{alg:movie_from_rainbow} on a diagram of the spun trefoil. The input is the Meier-Zupan triplane diagram for the spun trefoil~\cite[Fig 20]{MZ17Trans}. After mutual braid moves (shown in the second panel of the figure) and interior Reidemeister moves (not shown), we obtain a rainbow diagram for the spun trefoil $\Tcal=(T_1,T_2,T_3)$. We leave it to the reader to check that each unlink $T_i\cup \T_{i+1}$ is fully destabilizable. Notice that the crossings in $T_3$ can be turned into two bands $\nu$ satisfying $T_3=T_2[\nu]$. As in Step 3 above, the braided banded unlink is given by $L=T_1\cup \T_2$ and the bands are the mirors of $\nu$ as in the right panel of Figure~\ref{fig:spuntref_rain_to_band}.
\end{example}

\begin{figure}[ht]
\centering
\includegraphics[width=.65\textwidth]{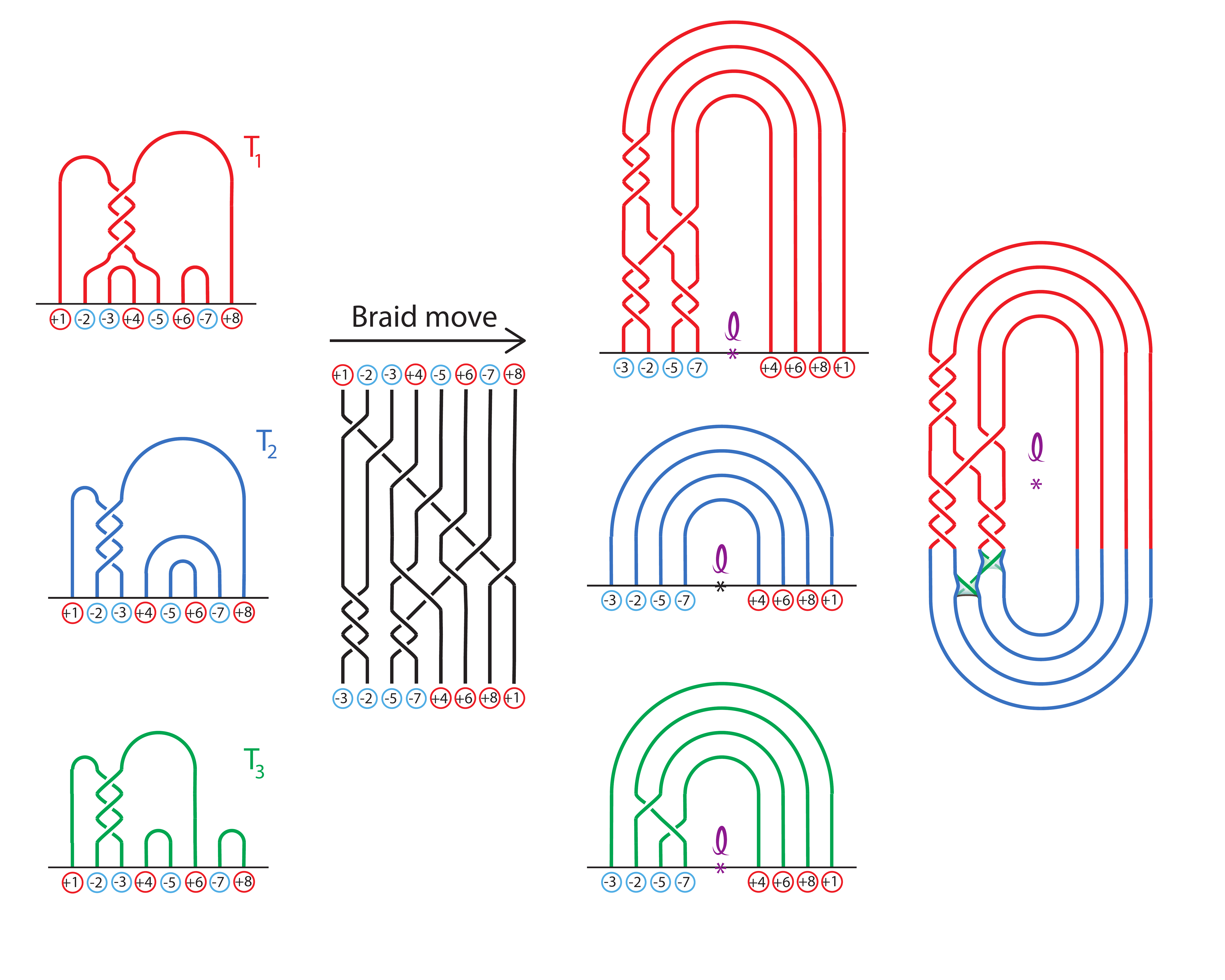}
\caption{(Left) A triplane diagram of the spun trefoil. (Middle) Rainbow diagram of spun trefoil together with bands $\nu$ such that $T_2[\nu]=T_3$. (Right) A banded unlink describing a braid movie of the same surface.}
\label{fig:spuntref_rain_to_band}
\end{figure}

\begin{proposition}
The output of Procedure~\ref{alg:movie_from_rainbow} is a banded unlink $(L,\omega)$ representing a movie for $F$.
\end{proposition}
\begin{proof}
By construction, $L[\omega]=T_1\cup \T_3$. As $\Tcal$ is a strong rainbow diagram, both $L$ and $L[\omega]$ are fully destabilizable braids. So, by Lemma~\ref{lem:existence_bands_for_unlink}, there exist braid movies that trace trivial disk systems bounded by $L$ and $L[\omega]$. Hence, $(L,\omega)$ is a braided banded unlink diagram.

It remains to see that the surface described by $(L,\omega)$ is, in fact, $F$. As $F$ is bridge trisected with spine $\Tcal$, $F$ is decomposed as the union of three trivial disk systems $\Dcal_1\cup \Dcal_2\cup \Dcal_3$. The movie described by this banded presentation can be divided into three cobordisms: (1) from the empty set to $L$, (2) from $L$ to $L[\omega]$, and (3) from $L[\omega]$ to the empty set. The first and third cobordisms correspond to the trivial disk systems $\Dcal_1$ and $\Dcal_3$ as they are bounded by $L=T_1\cup \T_2$ and $T_1\cup \T_3$. The second cobordism is, in turn, cut along a copy of $S^2\times I\subset S^3\times I$ into two parts: (a) one product region $T_1\times I$, and (b) the union of saddle cobordisms from $\T_2$ to $\T_2[\overline\nu]=\T_3$. The first piece can be regarded as a subset of $\Dcal_1$. By Lemma~\ref{lem:existence_bands_for_unlink}, the second piece is a trivial disk system with boundary $T_2\cup\T_3$. By uniqueness of trivial disk systems \cite{Livingston}, this piece is the trivial disk system $\Dcal_2$. Hence, $F$ is described by $(L,\omega)$.
\end{proof}

\subsection{Rainbow diagram from a braid movie}
In particular, this can used to turn a banded unlink diagram into a rainbow diagram using an algorithm due to Kamada that braids banded unlink diagrams \cite{kamada1994characterization}. By keeping track of the braid index of our braid movies, we obtain an inequality between the rainbow number and the braid index of a surface link (Theorem~\ref{thm:main_ineq}). 

\begin{procedure}[Rainbows from braided banded unlinks]\label{alg:rain_from_movies}
Let $(L,\nu)$ be a braided banded unlink for an oriented surface $F\subset S^4$; $L$ is a braided unlink and $\nu$ are braided bands so that $L[\nu]$ is also a braided unlink. 

\textbf{Step 1.} Divide the braid into two braided tangles $L=T_1\cup \T_2$ satisfying the following properties: 
\begin{enumerate}
    \item $T_1$ contains all the crossings of $L$, and 
    \item $T_2$ contains all the bands in $\nu$. 
\end{enumerate}
One may need to conjugate the braid $L$ and to slide the endpoints of the bands of $\nu$ along $L$. Denote by $\Sigma$ the sphere dividing $T_1\cup \T_2$. The braid axis $\ell$ is a subset of $\Sigma$ and divides it into two disks $\Sigma_-$ and $\Sigma_+$. We think of $\Sigma_\pm$ as disks punctured by the braid $L$ and tangles $T_1$ and $T_2$.  

\textbf{Step 2.} As $\nu$ is formed by braided bands, the cores of $\nu$ are embedded arcs in different disk pages (see Definition~\ref{def:braided_bands}). Project these cores onto $\Sigma_+$ to get an immersed graph $\Gamma$ with vertices forming a subset of $L\cap \Sigma_+$. See Figure~\ref{fig:spuntref_movie_to_rain1} for an example. 

\begin{definition}\label{def:banded_bridge_position}
We say that $(L,\nu)$ as above is in \emph{banded braided-bridge position} if the graph $\Gamma\subset \Sigma_+$ 
(a) is embedded (one edge per core of $\nu$), and 
(b) acyclic (a disjoint union of trees). In this case, we say that $\nu$ is \emph{dual to $L$}. 
\end{definition}
Figure~\ref{fig:spuntref_movie_to_rain2}(B) shows an example of a dual set of bands. It can be seen that $(L,\nu)$ is in a banded braided-bridge position if $|\nu|=1$. 
The goal of the following steps is to find a sequence of Markov stabilizations to put $(L,\nu)$ into banded braided-bridge position when $|\nu|>1$. The process will be iterative: Starting with one band, which is dual to $L$, we will modify $L$ so that two bands are dual to $L$, then modify $L$ again so that three bands are dual, and so on. 

\textbf{Step 3 (Iterations).} Order the bands in $\nu$ in terms of height, with $v_1$ being the lowest. Set $\omega=\{v_1\}$; at the moment $(L,\omega)$ is in banded-braided bridge position. 

\textbf{Step 3(a).} Let $x\in \nu$ be the lowest band in $
\nu$ above all the bands in $\omega$. If no such $x$ exists, we have then exhausted $\nu$ and move to Step 4. We will put a tilde on top of a band to denote the projection of its core onto $\Sigma_+$. In particular, the arcs of $\wt \omega$ have pairwise disjoint interiors and may intersect $\wt x$. 

\textbf{Step 3(b).} Find an arc $\alpha$ in $\Sigma_+$ with (i) one endpoint equal to an endpoint of $\wt x$ and the other in $\ell$ such that (ii) the interior of $\alpha$ is disjoint from $\wt \omega$. See bottom of Figure~\ref{fig:spuntref_movie_to_rain1}(A) for an example of such an $\alpha$. If $\alpha$ exists, move to Step 3(c). 
If such an $\alpha$ does not exist, we do the following: Pick any arc $\beta$ satisfying (i) that may cross $\wt \omega$; by construction, the union $\beta\cup \wt x$ has one puncture of $L$ in its interior. Push $\beta\cup \wt x$ off this puncture to get an arc $\beta'$, and use it to guide a Markov stabilization of $L$. As you can see in Figure~\ref{fig:movie_to_rain1}, after stabilization and pushing the new crossings of the braid above $\Sigma$, we will be able to find $\alpha$. 

It is important to note the following caveat. There were two possible ways to obtain $\beta'$ from $\beta\cup \wt x$, and the choice we made in Figure~\ref{fig:movie_to_rain1} requires a negative stabilization for the sliding process to be successful. In particular, a different push off $\beta'$ requires a positive Markov stabilization.  

\begin{figure}[ht]
\centering
\includegraphics[width=.9\textwidth]{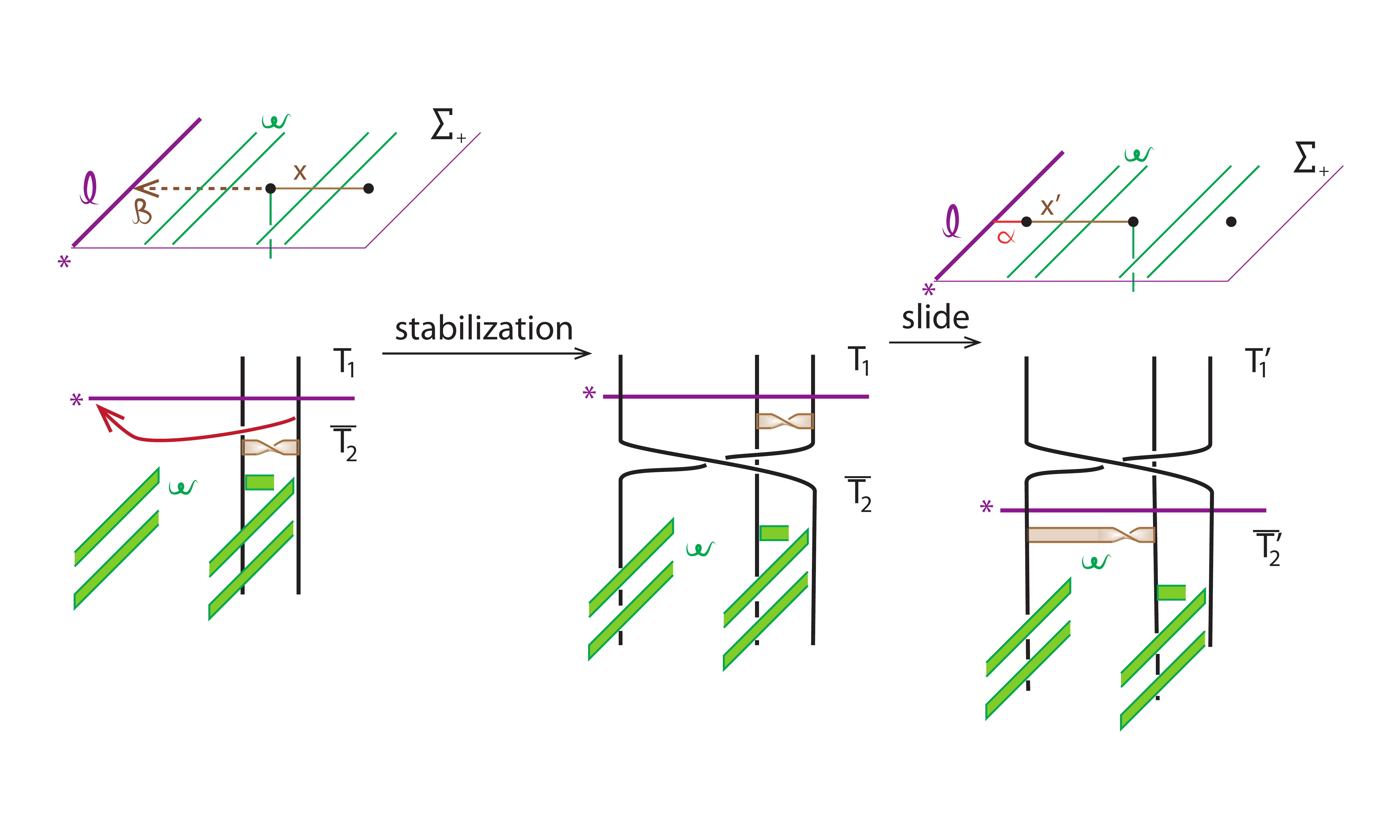}
\caption{Step 3(b): how to alter $L$ so that we can find $\alpha$. We use any arc $\beta$ disjoint from $\wt x$ to guide a Markov stabilization (left). Then, we push the new crossing of the braid above $\Sigma$ while keeping the bands below it. The resulting (braid, bands) tuple will contain the desired arc (top right). The top figures show a 3-dimensional view of the disk $\Sigma_+$, while the top pictures show the respective braid and bands near $\Sigma_+$.}
\label{fig:movie_to_rain1}
\end{figure}

\textbf{Step 3(c).} Consider $\alpha$ from Step 3(b). By construction, the endpoint union $\alpha\cup \wt x$ intersects one puncture of $L$ in its interior: $\alpha \cap \wt x$. Push $\alpha\cup \wt x$ off this puncture to get an arc $\alpha'$, and use it to guide a Markov stabilization of $L$. Push the new crossings above $\Sigma_+$ while keeping the bands in $\nu$ below it. As you can see in Figure~\ref{fig:movie_to_rain2}, the resulting tuple of (braid, bands) has the property that the arc corresponding to $\wt x$ is now disjoint from $\wt \omega$. In fact, the graph $\Gamma_\nu$ given by the edges of $\wt \nu=\wt \omega\cup \wt x$ is acyclic. 

The same caveat from Step 3(b) applies in this step. There were two possible ways to obtain $\alpha'$ from $\alpha\cup \wt x$, and the choice we made in Figure~\ref{fig:movie_to_rain2} requires a negative stabilization for the sliding process to be successful. In particular, a different push off $\alpha'$ requires a positive Markov stabilization. 

\textbf{Step 3(d) Repeat.} Add $x$ to the set $\omega$ and go to Step 3(a). 

\begin{figure}[ht]
\centering
\includegraphics[width=.8\textwidth]{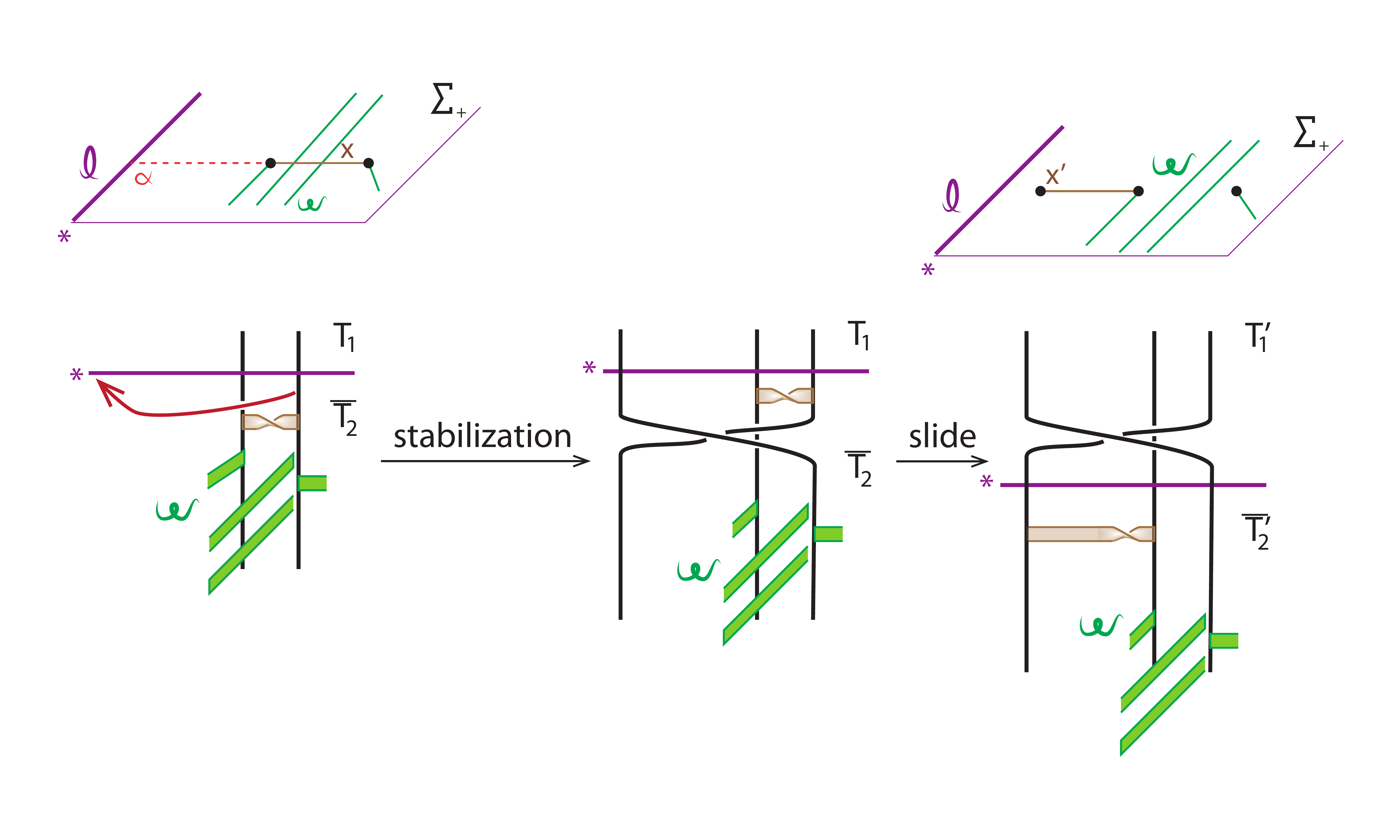}
\caption{Step 3(c): how to use $\alpha$ to make $(L,\omega\cup \{x\})$ in banded braided bridge position. The bottom figures show a 3-dimensional view of the disk $\Sigma_+$, while the top pictures show the respective braid and bands near $\Sigma_+$.}
\label{fig:movie_to_rain2}
\end{figure}

\textbf{Step 4.} At this point, we have a braided banded unlink $(L,\nu)$ is in braided banded bridge position (Def~\ref{def:banded_bridge_position}). In particular, $L=T_1\cup \T_2$, $T_2$ is crossingless, and the bands $\nu$ have endpoints on the tangle $\T_2$. Let $\T_3=T_2[\nu]$; so $T_3$ is the mirror image of the band surgery of the identity braid along the bands $\nu$. We output the tuple $\Tcal=(T_1,T_2,T_3)$. 
\end{procedure}

\begin{example}[Spun trefoil]\label{ex:trefoil_movie_to_rain}
Figures~\ref{fig:spuntref_movie_to_rain1}, \ref{fig:spuntref_movie_to_rain3}, and~\ref{fig:spuntref_movie_to_rain2} depict Procedure~\ref{alg:rain_from_movies} on a braided banded unlink for a diagram of the spun trefoil. The input is the braided banded unlink in Figure~\ref{fig:spuntref_movie_to_rain1}(A), which is the output of Figure~\ref{fig:spuntref_chart_to_movie}. We delay the explanation of why this is a braided banded unlink diagram for the spun trefoil until Section~\ref{sec:charts_from_rain} and Figure~\ref{fig:spuntref_chart_to_movie}. The signed arcs in Figure~\ref{fig:spuntref_movie_to_rain1} represent braided bands: conjugations of positive or negative powers of the generators of the braid group. Figure~\ref{fig:spuntref_movie_to_rain1}(A) is already the output of Step 2, but the bands $\{v_1,v_2\}$ are not dual to $L$; recall we label the bands in $\nu$ from lowest to highest. In this case, there is an arc $\alpha$ as in Step 3(b). In Figure~\ref{fig:spuntref_movie_to_rain1}(B)-(C), we use $\alpha$ to perform a Markov stabilization and slide the bands down while pushing the new crossings of $L$ above $\Sigma_+$. Figure~\ref{fig:spuntref_movie_to_rain3} is meant to help the reader see how the bands change when sliding down. At this point, notice that $\{v_1,v_2,v_3\}$ is dual to $L$ so we only need to run Steps 3(a)-(d) one more time, which we do in Figure~\ref{fig:spuntref_movie_to_rain2}(A)-(B). The rainbow diagram $(T_1,T_2,T_3)$ is then extracted in Figure~\ref{fig:spuntref_movie_to_rain2}(B)-(C).
\end{example}

\begin{figure}[ht]
\centering
\includegraphics[width=.55\textwidth]{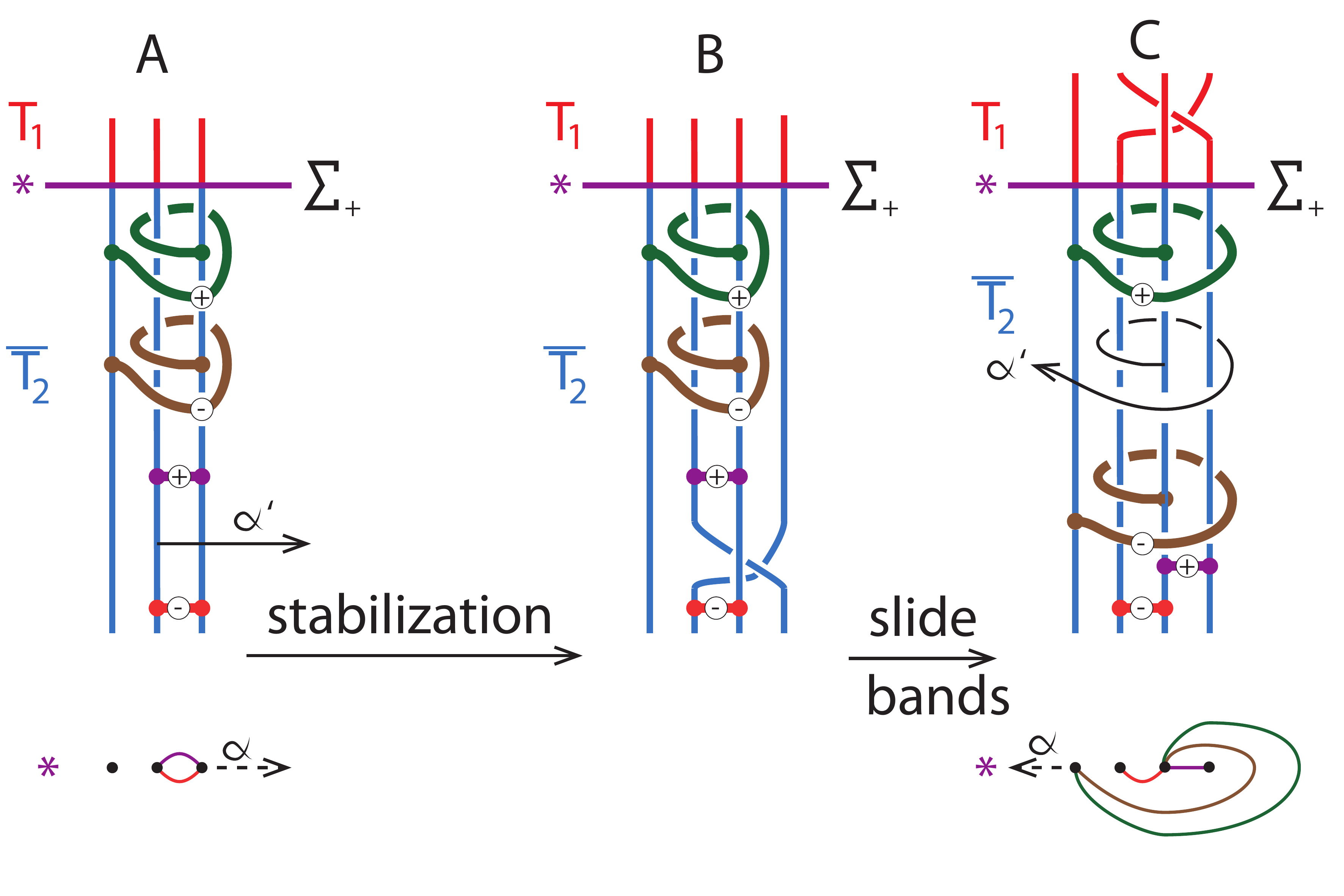}
\caption{Rainbow of spun trefoil from movie (1/3): We start with the braided banded unlink diagram from Figure~\ref{fig:spuntref_chart_to_movie}. After one stabilization, the resulting graph $\Gamma$ is already embedded, but it is not acyclic (right). So we perform Step 3(a) again and find an arc $\alpha$ connecting an endpoint of the green band with the binding. }
\label{fig:spuntref_movie_to_rain1}
\end{figure}
\begin{figure}[ht]
\centering
\includegraphics[width=.45\textwidth]{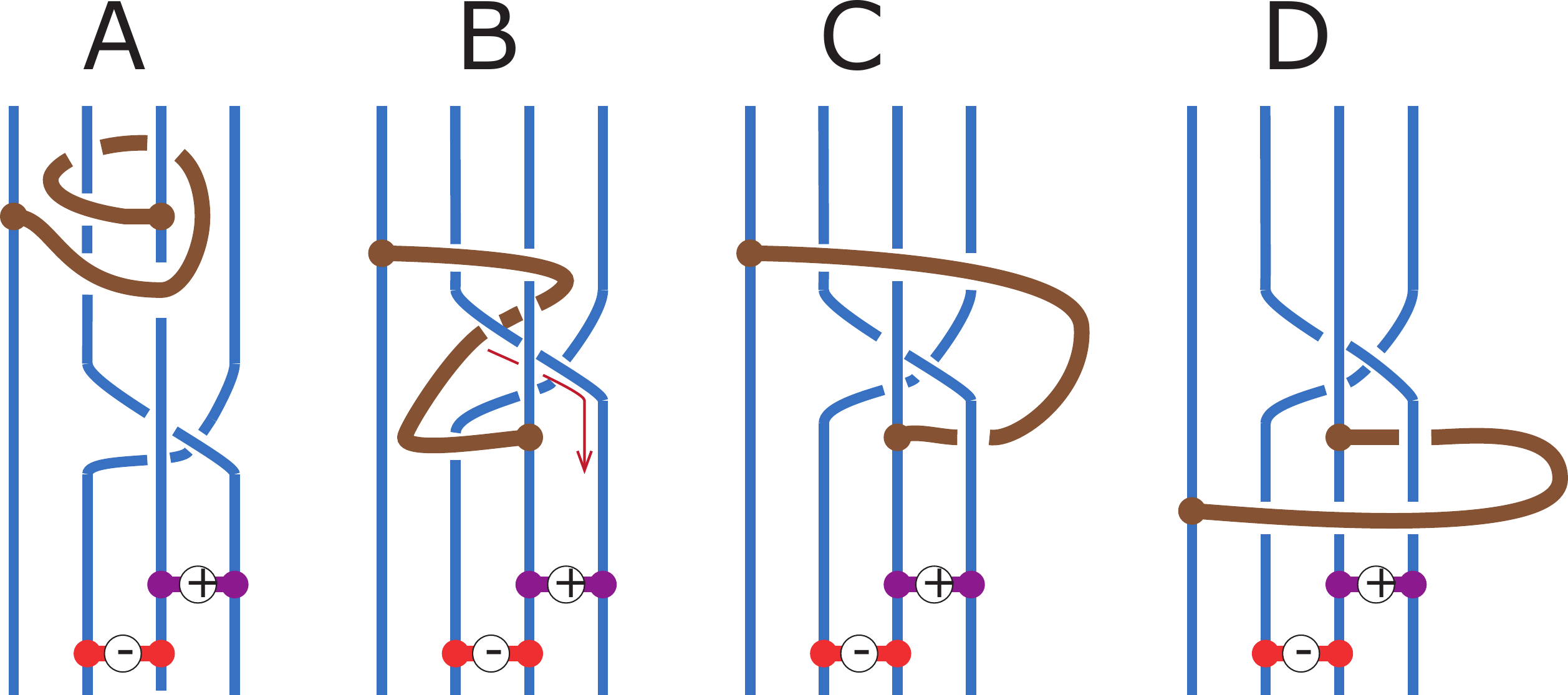}
\caption{Rainbow of spun trefoil from movie (2/3): Isotopy that slides down the red band from Figure~\ref{fig:spuntref_movie_to_rain1}. }
\label{fig:spuntref_movie_to_rain3}
\end{figure}
\begin{figure}[ht]
\centering
\includegraphics[width=.5\textwidth]{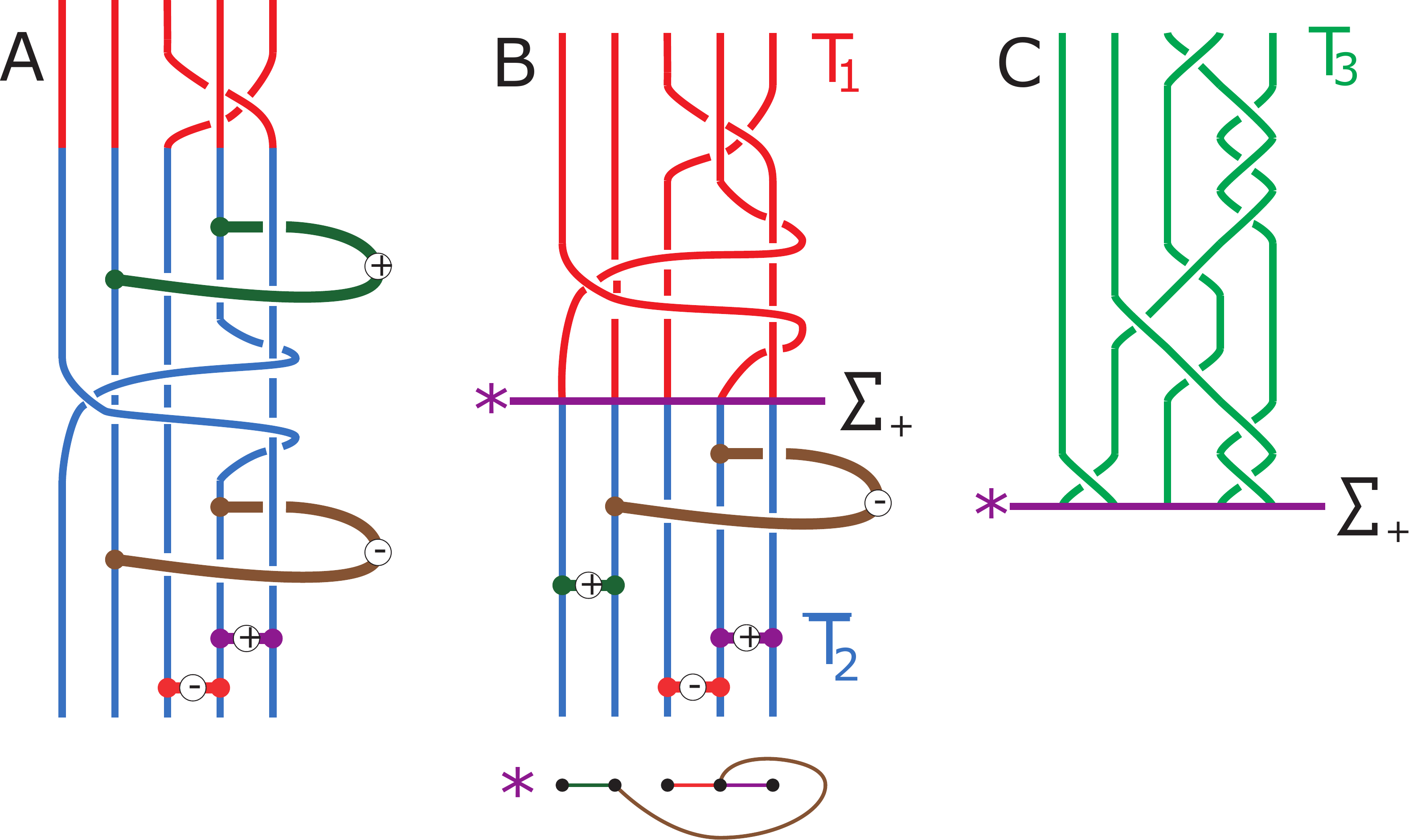}
\caption{Rainbow of spun trefoil from movie (3/3): (left) result of a negative Markov stabilization to the rightmost braid in Figure~\ref{fig:spuntref_movie_to_rain1}. After sliding the bands down $\Sigma_+$ and the crossings above it, we obtain a rainbow diagram $(T_1,T_2,T_3)$; $T_2$ is the crossingless braid and $T_3$ is the mirror image of surgery of $T_2$ along the four bands in $\nu$.}
\label{fig:spuntref_movie_to_rain2}
\end{figure}

\begin{proposition}\label{prop:rain_from_movies}
The output of Procedure~\ref{alg:rain_from_movies} is a rainbow diagram $(T_1,T_2,T_3)$ for the input surface.
\end{proposition}
\begin{proof}
Procedure~\ref{alg:rain_from_movies} performs isotopies and braid stabilizations to turn a braided banded unlink for $F$ into one in banded braided bridge position (Definition~\ref{def:banded_bridge_position}). In particular, the final braided banded presentation still represents the original surface $F$. In what follows, we will show that a banded unlink diagram in banded braided-bridge position determines a braided bridge splitting for its underlying surface. 

Let $(L,\nu)$ be in banded braided bridge position, and let $T_1$, $T_2$, and $T_3$ be the tangles in Step 4 above. By construction, each $T_i$ is a braided tangle or rainbow. Thus, it is enough to check that $\Tcal=(T_1,T_2,T_3)$ is the spine of a bridge trisection of $F$. In other words, we need to describe trivial disk systems $\Dcal_i$ with $F=\Dcal_1\cup \Dcal_2\cup \Dcal_3$ and $\Dcal_i\cap \Dcal_{i-1}=T_i$. The first disk system corresponds to a neighborhood of the index-zero critical points of $F$. These are cobordisms from the empty set to $L$ tracing an isotopy from $|L|$ small unknots to $L$. Similarly, $\Dcal_3$ corresponds to a neighborhood of the index-two critical points. We can take $\Dcal_1$ and $\Dcal_3$ so that $\Dcal_1\cap \Dcal_3=T_1$. 

It remains to check that the leftover surface $\Dcal_2=\overline{F-(\Dcal_1\cup \Dcal_3)}$ is a trivial disk system. By construction, $\Dcal_2$ can be thought of as a cobordism from $T_2$ to $T_3$ with saddles induced by the bands in $\nu$. We will show that the bands in $\nu$ can be used to fully destabilize $T_2\cup \T_3$. This way, the cobordism $\Dcal_2$ ought to be a trivial disk system. 

By construction, $T_2$ is crossingless and $T_3=\overline{T_2[\nu]}$. So $L_{23}=T_2\cup \T_3$ is equal to the band surgery on the crossingless braid along the bands in $\nu$. We also know that the braided bands $\nu$ lie near $\Sigma_+$ and the projections of their cores form an acyclic graph $\Gamma$. Let $x\in \nu$ be a band corresponding to a leaf of $\Gamma$ and let $p\in \Sigma_+$ be an endpoint of $x$ of degree one. As $\Gamma_+$ is an acyclic graph in a disk $\Sigma_+$, there is an arc $\gamma$ with interior disjoint from $\Gamma$ that connects $p$ with a point in the braid axis $\ell=\partial \Sigma_+$. As the braid $L_{23}$ is crossingless, sliding the arc $\gamma$ along $L_{23}$ traces an embedded disk $D_\gamma$ in $S^3$ intersecting $\ell$ once and with interior disjoint from $L_{23}$. In the band surgered braid $L_{23}[\nu]$, we can use $D_\gamma$ and the band $x$ to perform a Markov destabilization as in Figure~\ref{fig:movie_to_rain3}. 
The resulting braid is equal to $L'_{23}[\nu']$ where $\nu'=\nu-\{x\}$ and $L'_{23}$ is the crossingless braid resulted from $L_{23}$ by removing the boundary of $D_\gamma$ (that is the unknot traced by the puncture $p$). As the new tuple $(L'_{23},\nu')$ is also in banded braided-bridge position, we can continue destabilizing the link $L'_{23}[\nu']$ until there are no bands left. In conclusion, $L_{23}[\nu]$ fully destabilizes as desired. 
\end{proof}
\begin{figure}[ht]
\centering
\includegraphics[width=.5\textwidth]{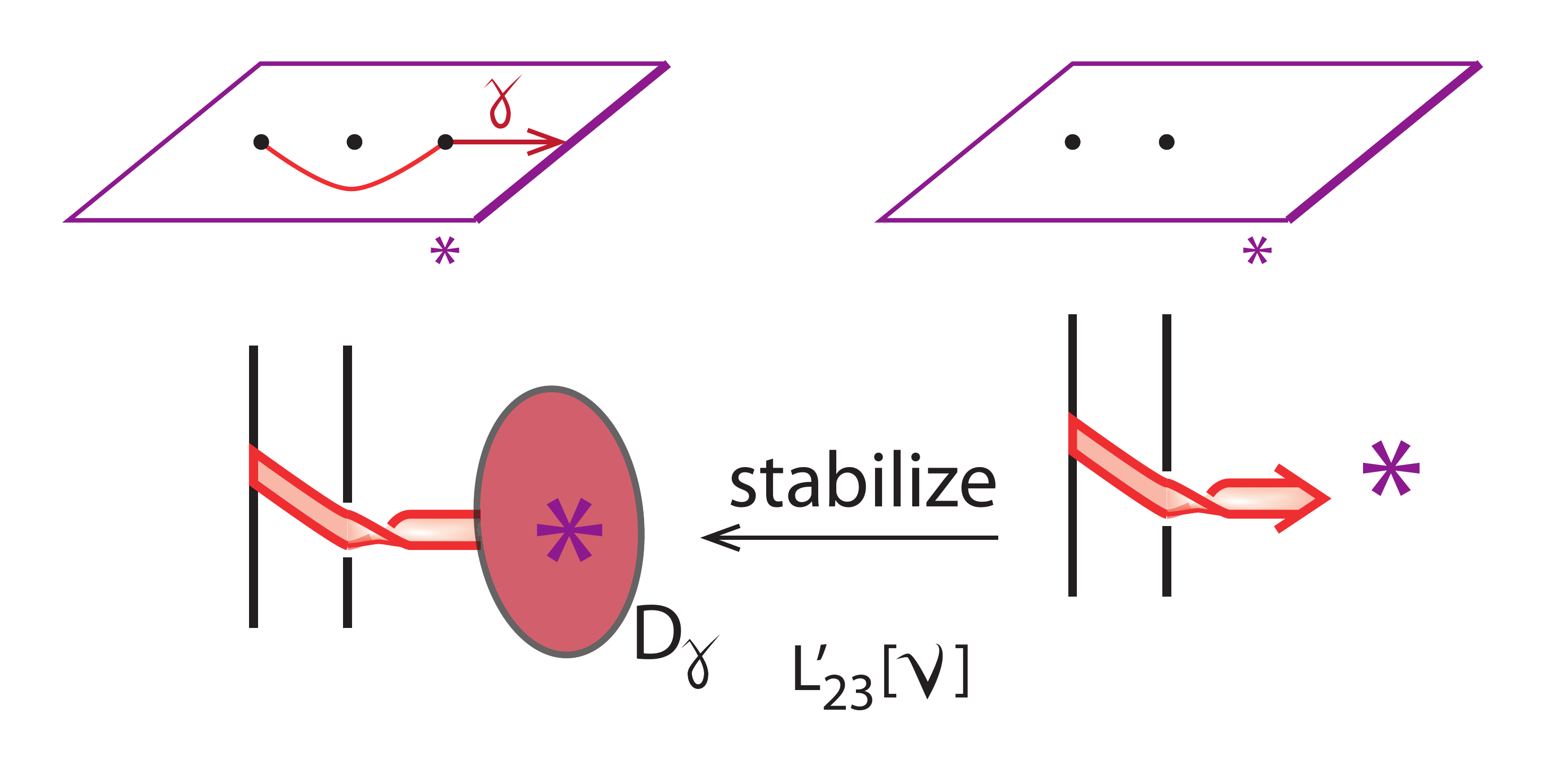}
\caption{How $D_\gamma$ and the band $x$ can be used to find a destabilization for $L_{23}[\nu]$ (left). The top figures show a ``bird's-eye view" of the disk $\Sigma_+$ while the bottom pictures show the respective braid and bands near $\Sigma_+$.}
\label{fig:movie_to_rain3}
\end{figure}

\section{Rainbows and braid charts}\label{sec:rainbows_and_charts}

\subsection{Review on braid charts}

A \emph{braid chart of degree $n$} is an edge-oriented, labeled graph that is embedded in a disk and that has three types of vertices. First, the edge labels are taken from the set \mbox{$\{\pm 1, \pm 2, \ldots , \pm (n-1) \}.$} The integer $n$ corresponds to the {\emph{braid index of the chart}}.  Edges usually run horizontally. Edges that run left-to-right have positive integer labels. Edges that run right-to-left have negative integer labels. The label upon an edge coincides with the braid generator or its inverse that corresponds to the edge. \emph{Black vertices} are monovalent, and they can be classified as sources or sinks.  
\emph{White vertices} have three incoming and three outgoing edges. The labels upon the incident edges at a white vertex alternate as indicated in Figure~\ref{BraidChart}. We no longer draw small white circles around white vertices, but in the tradition of red herrings\footnote{According to Lewis Carroll, a red herring neither is red nor is it a fish}, we keep calling them white vertices. As indicated in Figure~\ref{brokenAndChart}, white vertices correspond to Reidemeister type-III moves. and they project to triple points when a surface in $4$-space is projected into $3$-space.
A {\it crossing} is a $4$-valent vertex at which the incident edges alternate cyclically $\pm i, \pm j, \pm i, \pm j$ and $||i|-|j||>2.$ Crossings are drawn in a fashion so that the edge whose label is larger (in absolute value) is depicted as an over-crossing. 

In most of our braid charts, we add color and thickness redundancies: Edges whose labels are smaller in absolute value are drawn with thinner edges. Different colors correspond to different labels. Consequently, black vertices are seldom colored black. The sign of the label on an edge is a redundant indicator. For graphical and notational ease, $-i$ indicates the braid generator $\sigma_i^{-1}$. 

 \begin{figure}[ht]
 \centering
\includegraphics[width=.85\textwidth]{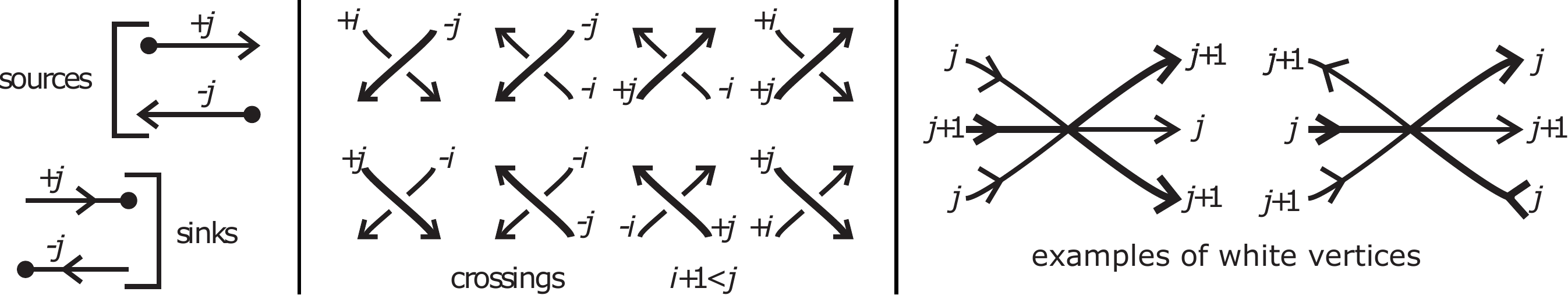}
\caption{Types of vertices in a braid chart: black vertices, crossings, and white vertices.}
\label{BraidChart}
\end{figure}

Figure~\ref{brokenAndChart} gives a correspondence between braid charts, local movies, and broken surface diagrams \cite{KamBook}. When thinking of the left-to-right direction of the disk in which a chart appears, edges have optimal points (minima and Maxima). At such, a positive and a negative edge converge. Such a critical event corresponds to the cancellation of a braid generator and its inverse. We refer the reader to \cite{KamBook} for more background on braid charts.
 
\begin{figure}[ht]
 \centering
\includegraphics[width=.55\textwidth]{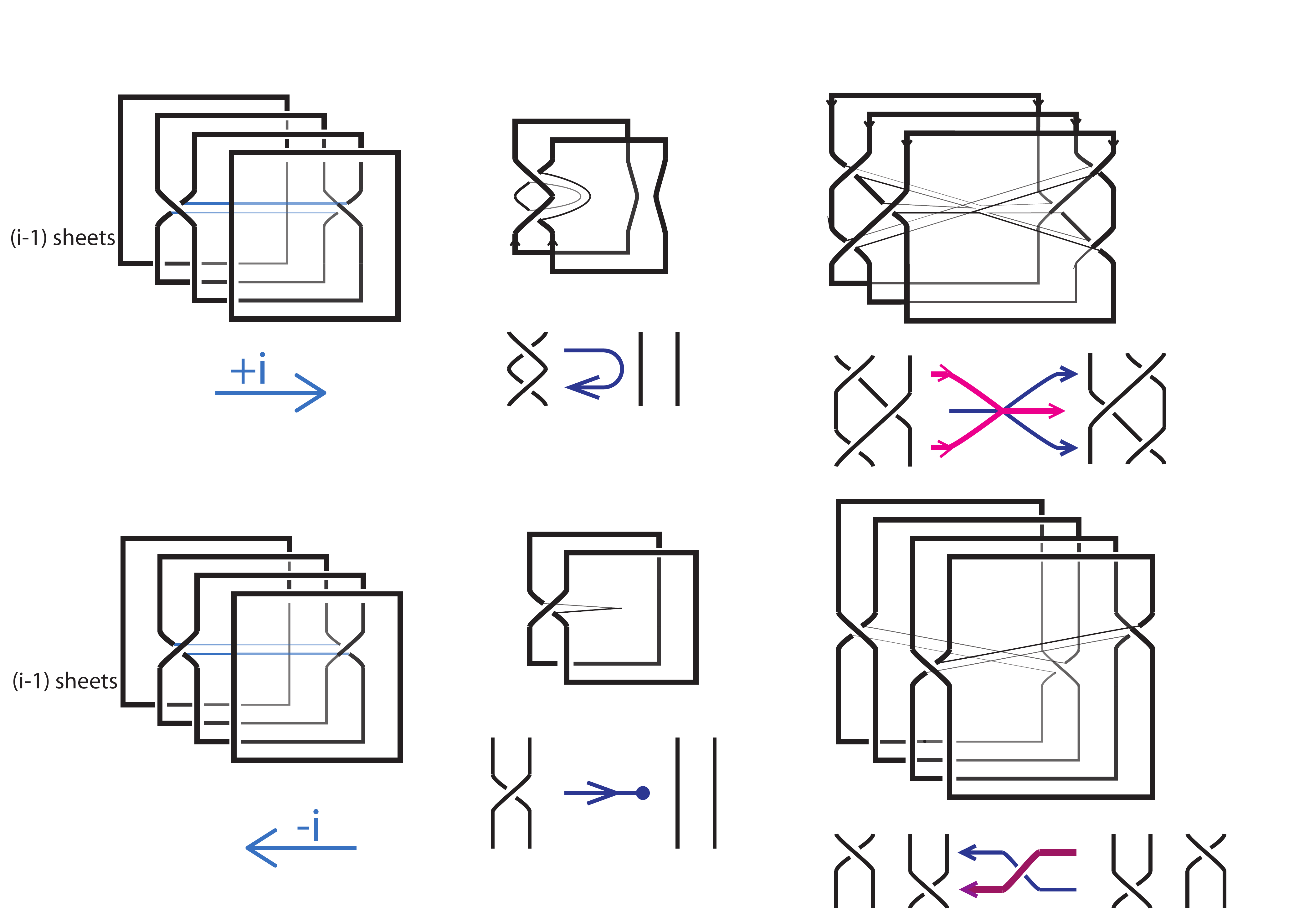}
\caption{Broken surfaces, local movies,  and chart vertices}
\label{brokenAndChart}
\end{figure}

\subsection{Braid chart from a rainbow diagram}
Joseph, Meier, Miller, and Zupan described a procedure to obtain a broken surface diagram from a triplane \cite[Thm 3.2]{Joseph_Classical_knot_theory}. Briefly explained, one draws the spine of the trisection (respecting the crossings) inside $\R^3$. As each pair forms an unlink, a sequence of Reidemeister moves turning it into a crossingless diagram traces one third of a broken surface diagram. In the context of rainbows, a strong rainbow diagram can be encoded with three braid words. As each pair of words forms a fully destabilizable unlink (by multiplying one word by the other's inverse), a sequence of braid moves and destabilizations turning this to a trivial word will give us a braid movie to a crossingless braid. Thus, tracing a third of a braid chart. 

\begin{procedure}[Braid chart from rainbow]\label{alg:chart_from_rainbow}
Let $\Tcal$ be a weak-rainbow diagram for a surface $F\subset S^4$. 

\textbf{Step 1.} Perform enough Markov stabilizations to $\Tcal$ to obtain a rainbow diagram; Lemma~\ref{lem:strong_rainbow}. Let $x_1,x_2,x_3\in B_b$ be elements in the braid group that represent the tangles $T_1$, $T_2$, and $T_3$, respectively. Recall that, in this paper, we read braids from top to bottom so that $T_i\cup \T_j$ is the braid closure of $x_i\overline{x_j}$.

\textbf{Step 2.} For each $j=1,2,3$, let $\Gamma_j=\{re^{k\pi i}:r\geq 0, k=1/2+2(j-1)/3\}$ be a ray in the plane. Following the convention in Figure~\ref{fig:spuntref_rain_to_chart1} (third column), draw lines in each $\Gamma_j$ so that traversing $\Gamma_j$ from the origin reads the braid word $\overline {x_j}$ once.  Note that, if we travel $\Gamma_i\cup \Gamma_j$, passing through the origin, we will then read the braid word for $T_i\cup \T_j$. Thus, the added marks on $\Gamma_1\cup \Gamma_2\cup \Gamma_3$ describe the spine of $\Tcal$ embedded in $\R^4$. 

\textbf{Step 3.} As we have a rainbow diagram, the braids $T_i\cup \T_{i+1}$ are fully destabilizable. So there is a sequence of braid and saddle moves that turn $x_i\overline{x}_{i+1}$ into a trivial braid that traces a trivial disk system bounded by $T_i\cup\T_{i+1}$; Lemma~\ref{lem:existence_bands_for_unlink}. Using Figure~\ref{brokenAndChart} as a reference, translate these moves to a braid chart. See Figure~\ref{fig:spuntref_rain_to_chart1} for an example. The union of these three charts is the output braid chart.
\end{procedure}

\begin{proposition}\label{prop:chart_from_rainbow}
The output of Procedure~\ref{alg:chart_from_rainbow} is a braid chart for the input surface. 
\end{proposition}
\begin{proof}
Recall that $F$ comes equipped with a braided bridge splitting; $F=\Dcal_1\cup \Dcal_2\cup \Dcal_3$. As each braid $T_i\cup \T_{i+1}$ is completely destabilizable, Lemma~\ref{lem:existence_bands_for_unlink} ensures that the movie of braids in Step 3 of the procedure is one for a disk system. As trivial disk systems are unique rel boundary, such disks are isotopic to $\Dcal_i$ \cite{Livingston}. Thus, the braid chart represents $F$.
\end{proof}

\begin{remark}
For the experts, one could think of braid charts as a braided surface version of the Morse 2-functions on 4-manifolds \cite{GK}. Thus, the output of Procedure~\ref{alg:chart_from_rainbow} can be regarded as a \emph{trisected braid chart}. The question then becomes: Can every braid chart be modified into a trisected braid chart? One possible answer to this question is given in Section~\ref{sec:charts_from_rain}.
\end{remark}

\begin{example}[Spun trefoil Version 1]
Figure~\ref{fig:spuntref_rain_to_chart1} depicts the output of Procedure~\ref{alg:chart_from_rainbow}; the input diagram is from Figure~\ref{fig:spuntref_rain_to_band}. Figure~\ref{fig:spuntref_rain_to_chart2} shows the braid movies that make each braid $x_i\cup \overline{x}_{i+1}$ into a trivial braid; this is needed in Step 3 of the procedure. For the readers familiar to braid charts, Figure~\ref{fig:spuntref_simplifying_chart} shows a sequence of braid chart moves that turn our chart into the standard chart for the spun trefoil. 
\end{example}
\begin{figure}[ht]
\centering
\includegraphics[width=.9\textwidth]{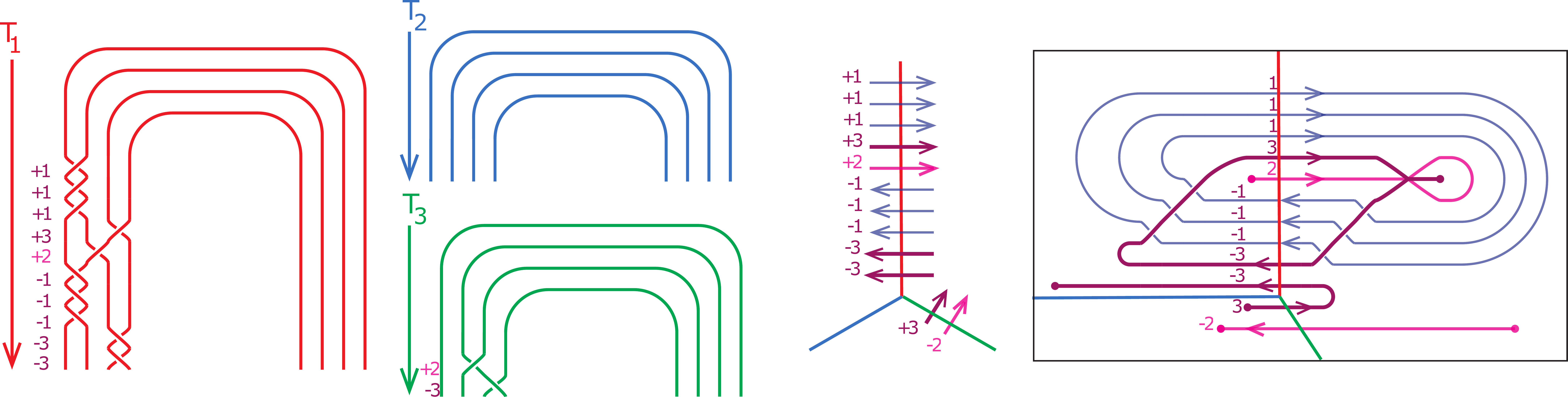}
\caption{Procedure~\ref{alg:chart_from_rainbow} applied to a triplane diagram of the spun trefoil from Figure~\ref{fig:spuntref_rain_to_band}.}
\label{fig:spuntref_rain_to_chart1}
\end{figure}
\begin{figure}
    \centering
    \includegraphics[width=0.75\linewidth]{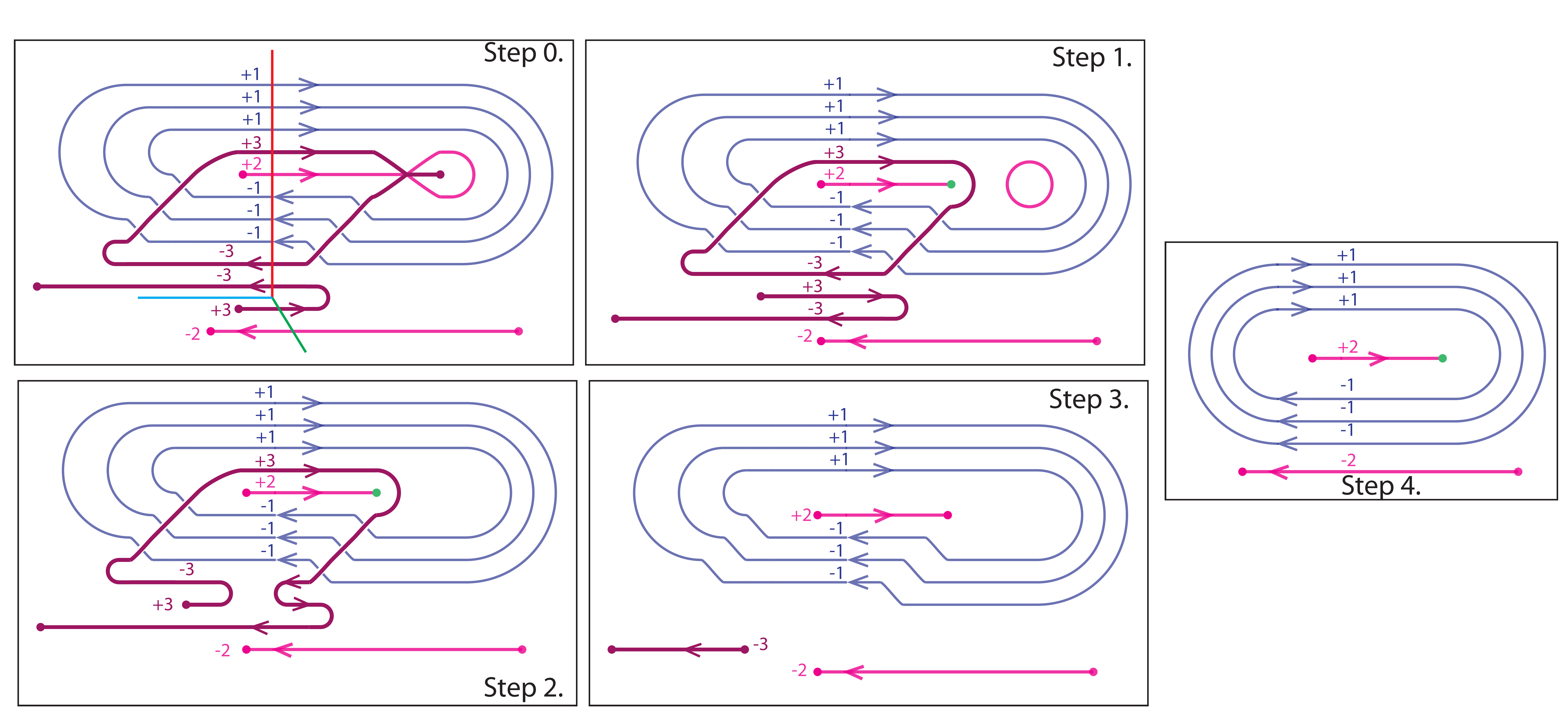}
    \caption{The output of Figure~\ref{fig:spuntref_rain_to_chart1} can be simplified to get a Kamada's \cite{KamBook} diagram for the spun trefoil.}
    \label{fig:spuntref_simplifying_chart}
\end{figure}
\begin{figure}
    \includegraphics[width=0.85\linewidth]{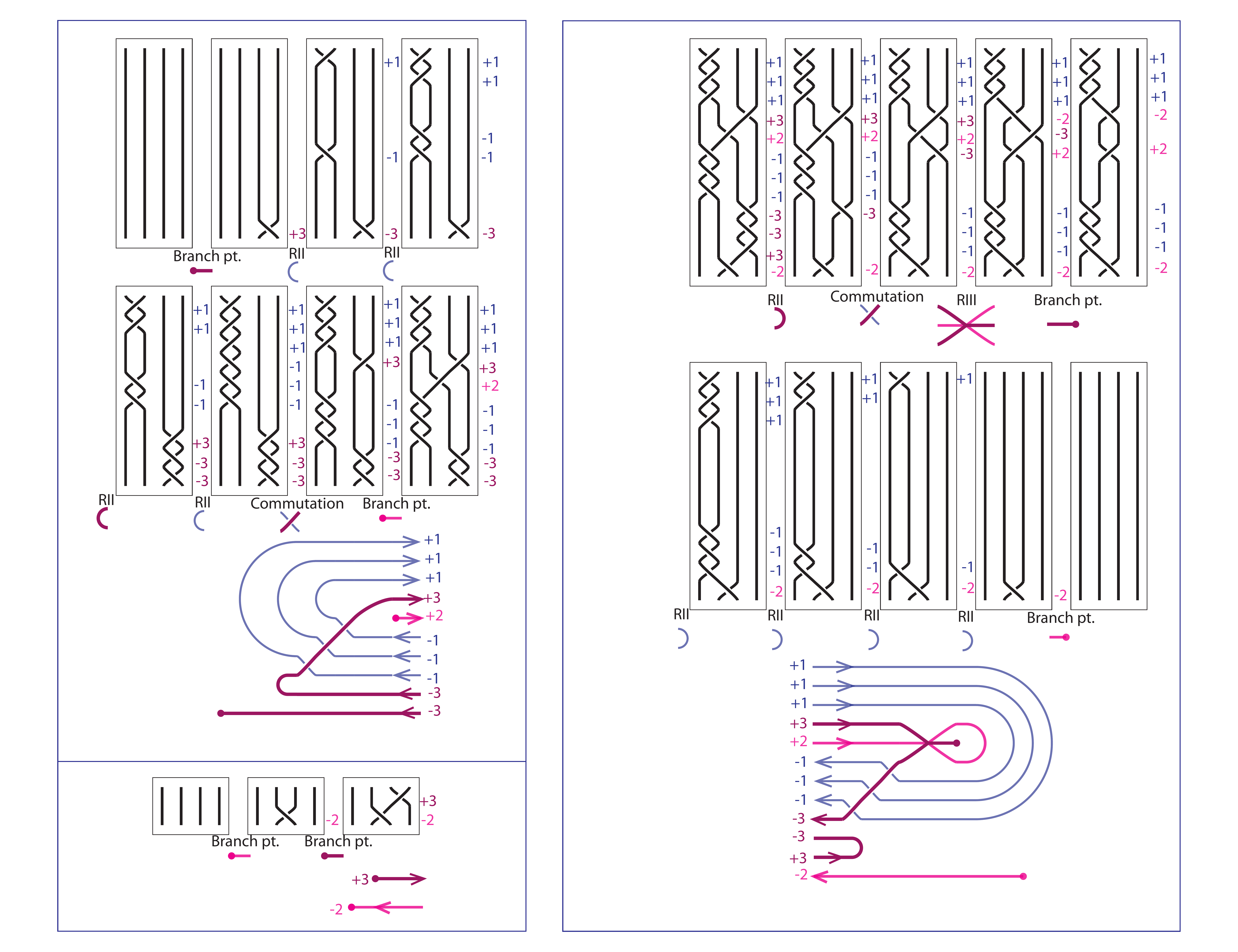}
    \caption{Each frame contains (1) a sequence of braid isotopies and saddle moves that untie $T_i\cup \T_{i+1}$ while tracing a trivial disk system bounded by it, and (2) a braid chart corresponding to such a movie. The tangles we used are those in Figure~\ref{fig:spuntref_rain_to_chart1}.}
    \label{fig:spuntref_rain_to_chart2}
\end{figure}

\subsection{Rainbows from braid charts}\label{sec:charts_from_rain}
The first half of Procedure~\ref{alg:rain_from_chart} is to turn a braid chart into a braided banded unlink diagram. Figure~\ref{fig:spuntref_chart_to_movie} shows how to do this for the spun trefoil. 

\begin{example}[Journey of the spun trefoil]
A rainbow diagram for the spun trefoil was obtained following Procedure~\ref{alg:rain_from_movies} as outlined in Example~\ref{ex:trefoil_movie_to_rain} and Figures~\ref{fig:spuntref_movie_to_rain1}--\ref{fig:spuntref_movie_to_rain2}. Then Procedure~\ref{alg:rain_from_chart} was applied to obtain a chart. The result is  the chart depicted in Figure~\ref{Fig_result_4_spunTref}. It is quite a bit more baroque than the chart in Figure~\ref{fig:spuntref_chart_to_movie} from whence it came. We leave it as an exercise to the reader to find chart moves that reduce it to the chart in Figure~\ref{fig:spuntref_chart_to_movie}.
\end{example}

\begin{figure}[ht]
\centering
\includegraphics[width=.45\textwidth]{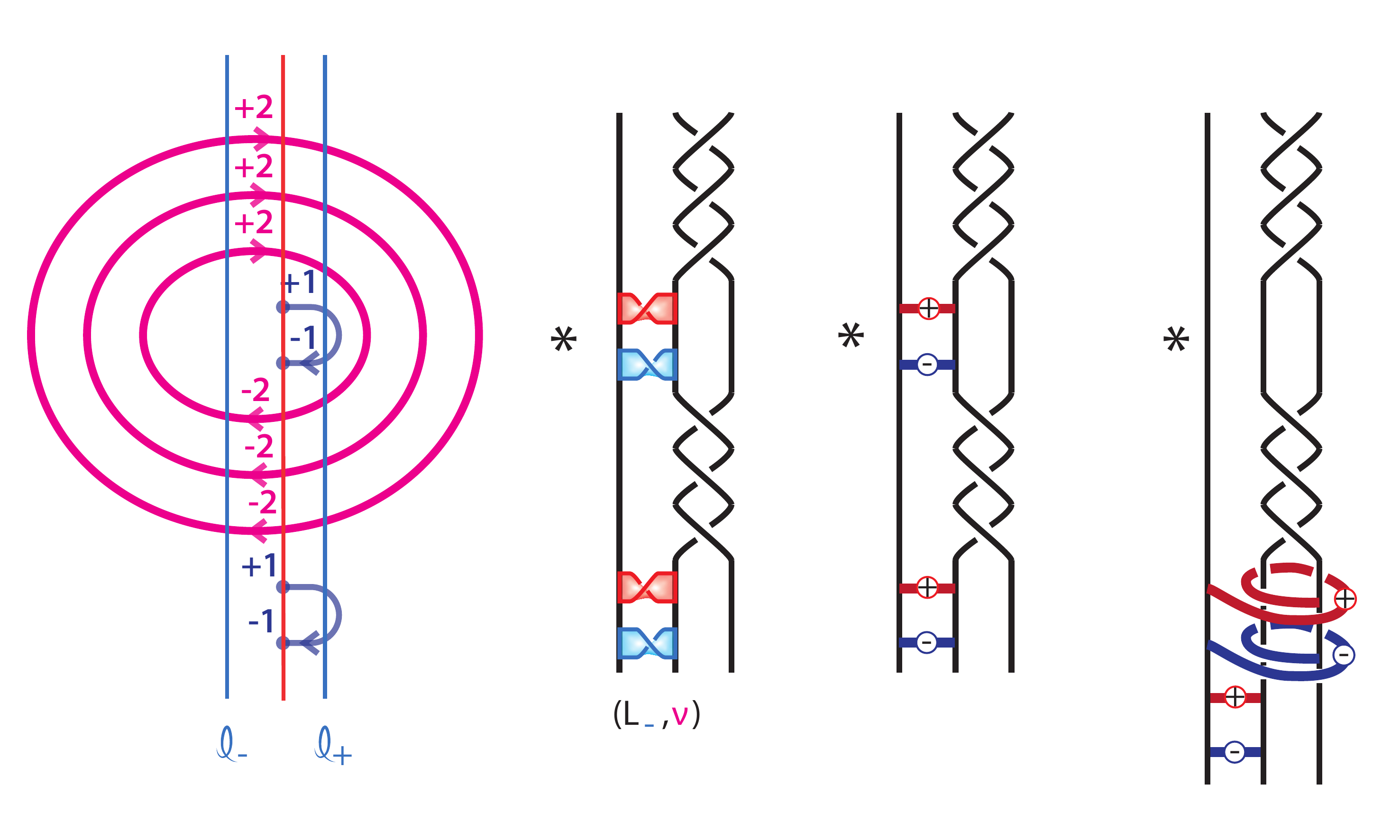}
\caption{Spun trefoil: braided banded unlink (right) from a braid chart (left). The symbol $\star$ denotes the binding of the braids, and the arcs in the rightmost figure correspond to the cores of the braided bands of $\nu$. Also compare to Figure~\ref{fig:spuntref_movie_to_rain1}.}
\label{fig:spuntref_chart_to_movie}
\end{figure}

\begin{figure}[ht]
\centering
\includegraphics[width=.3\textwidth, angle = 90]{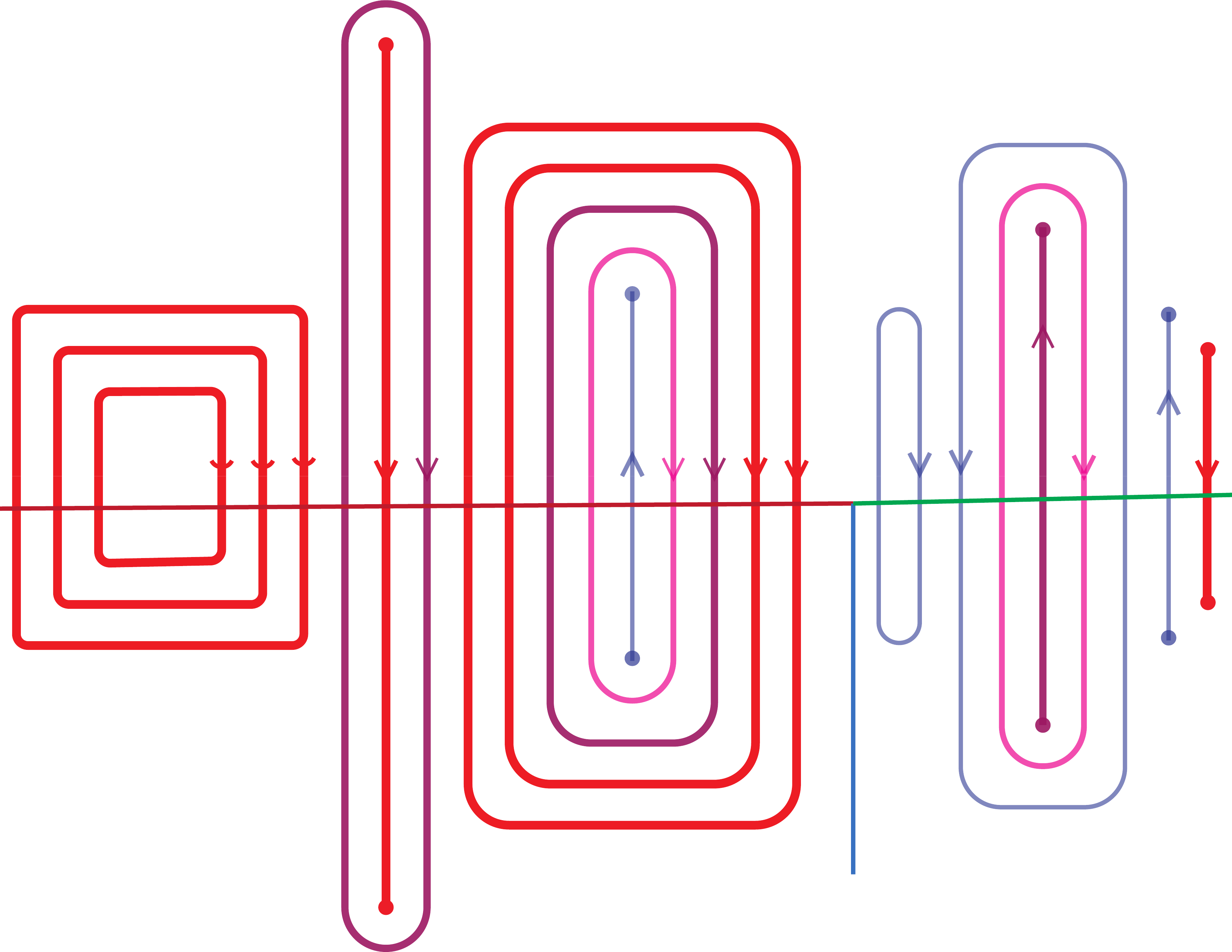}
\caption{Spun trefoil: Applying  Figure~\ref{fig:spuntref_movie_to_rain1}, to Figure~\ref{fig:spuntref_chart_to_movie}, this chart is obtained.}
\label{Fig_result_4_spunTref}
\end{figure}

\begin{procedure}[Rainbow from braid chart]\label{alg:rain_from_chart}
Let $\Ccal$ be a braid chart for a surface $F$. As above, this is a compact, labeled, oriented planar graph with vertices as in Figure~\ref{BraidChart}. 

\textbf{Step 1.} Find an embedded line $\ell$ in the plane, dividing it into two connected components, that satisfies the following conditions: 
\begin{enumerate}
\item $\ell$ avoids all white vertices of $\Ccal$, 
\item away from the saddle vertices, $\ell$ intersects $\Ccal$ transversely, 
\item for each saddle vertex $v$ of $\Ccal$, there is a small disk neighborhood $E_v$ for which $\ell\cap E_v$ is a diameter and $\Ccal\cap E_v$ is a radius, and 
\item $\Ccal\cap E_v$ is always on the same side of $\ell$ for all saddle vertices; see left panel of Figure~\ref{fig:spuntref_chart_to_movie}. 
\end{enumerate}

\textbf{Step 2.} Using the disks $E_v$ from (4) above, we can push $\ell$ off the saddle vertices in two ways. One arc is denoted by $\ell_-$ and satisfies that $\ell_-\cap E_v \cap \Ccal$ is empty for all saddle vertices. The other arc is denoted by $\ell_+$ and intersects each $E_v\cap \Ccal$ in one point. 

\textbf{Step 3.} By construction, $\ell_\pm$ satisfies (1)-(3) above. So we can use them to read off braids denoted by $L_-$ and $L_+$. Notice that each saddle vertex of $\Ccal$ gives us braided bands, denoted by $\nu$, so that $L_+=L_-[\nu]$. The pair $(L_-,\nu)$ is a braided banded unlink for $F$.

\textbf{Step 3.} Use Procedure~\ref{alg:rain_from_movies} to turn $(L_-,\nu)$ into  a rainbow.
\end{procedure}

\begin{proposition}
The output of Procedure~\ref{alg:rain_from_chart} is a rainbow diagram for the input surface.
\end{proposition}
\begin{proof}
Condition (3) in Step 1 ensures that $\ell_{-} \cup l_{+}$ cuts two planes from the braid chart $\Ccal$ that contain no saddle vertices. Thus, there is a movie of braids consisting of braid moves only (no saddles) that turns each $L_\pm$ into a crossingless braid. As all the saddles of $F$ are contained in the cobordism from $L_-$ to $L_+$ given by $\nu$, $(L_-,\nu)$ is a braided banded unlink for $F$. To end, Proposition~\ref{prop:rain_from_movies}, the output is a rainbow diagram for $F$. 
\end{proof}

\begin{theorem}\label{thm:main_ineq}
Let $F$ be an orientable surface, then 
\[ 
\braid(F) \leq \rain(F)\leq 5\cdot\braid(F) - 2\cdot\chi(F) -2, 
\]
where $\chi(F)$ is the Euler characteristic of $F$. 
\end{theorem}
\begin{proof}
Procedure~\ref{alg:chart_from_rainbow} turns a rainbow diagram into a braid chart of the same braid index. Thus, $\braid(F) \leq \rain(F)$. We now discuss the second inequality. 
The first three steps of Procedure~\ref{alg:rain_from_chart} turn a braid chart into a braided banded unlink of the same braid index and same surface. Then, Procedure~\ref{alg:rain_from_movies} produces a rainbow diagram of the same surface from the braided banded unlink. When running this, for braided banded unlinks with $n$ bands, one may need to perform at most $2(n-1)$ Markov stabilizations; see Steps 3(b)-(c) of Procedure~\ref{alg:rain_from_movies}. Thus, 
$$ \rain(F)\leq \braid(F)+2n-2.$$ Now, as the braided banded presentation arose from a braid chart, we know that the number of index-zero and index-two critical points is $\braid(F)$ each. Thus, $\chi(F)=2\cdot\braid(F)-n$. The result follows. 
\end{proof}

\begin{proof}[Proof of Theorem~\ref{thm:bridge_and_braid_index}]
By Equation~\eqref{eq:inequality1} back in Section~\ref{sec:rainbow_diags}, the bridge index is bounded from above by the rainbow number. Thus, the result follows from Theorem~\ref{thm:main_ineq}.
\end{proof}

\section{Fox's Example 12}\label{sec:Ex12}

In this section, various algorithms for navigating among the descriptions of the knotted surface, specifically the $2$-twist spun trefoil, are presented. 
Historical and background comments are interspersed. At times the example will be called ``Example 12" or $t_2(3_1)$ 
--- the $2$-twist-spun trefoil.

\begin{figure}[ht]
\centering\includegraphics[width=.7\textwidth]{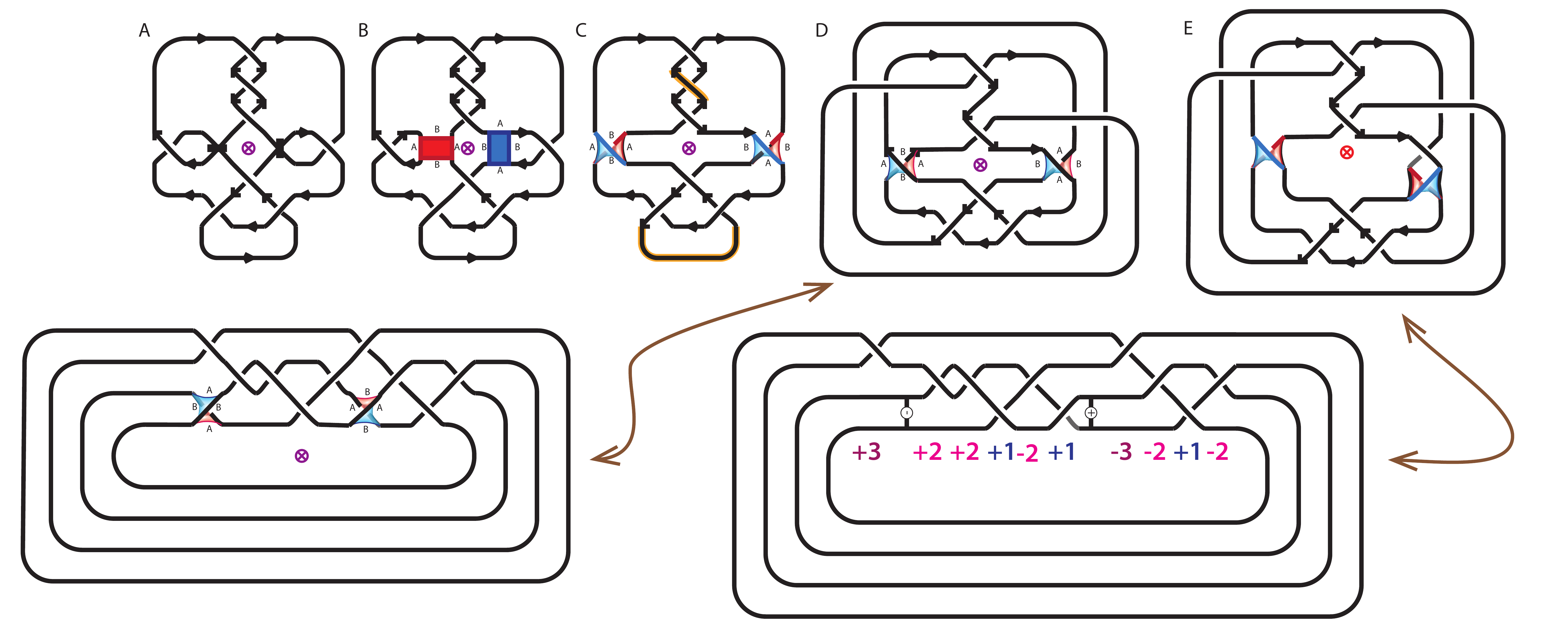}
\caption{Descriptions of Fox's Example 12: (A) marked vertex diagram, (B)-(D) banded unlink diagrams, and (E) braided banded unlink diagram.}
\label{fig:ex12_1}
\end{figure}

Figure~\ref{fig:ex12_1}(A) indicates the marked vertex diagram that is associated to Fox's Example 12 \cite{Fox61}. This diagram represents a $2$-knot ({\it knotted sphere in $\R^4$}) that is now known as the $2$-twist-spun trefoil \cite{Kanenobu83}. R. Litherland first noticed this for Example 12 in a letter to Cameron Gordon. Curiously, Fox's examples predate Zeeman's general description of twist-spinning \cite{Zeeman}. Both Fox and Zeeman were in attendance at the Georgia Topology Conference  whose proceedings contained ``A Quick Trip."  It is not uncommon to also indicate the example as a banded unlink diagram as in Fig.~\ref{fig:ex12_1}(B). 

We show how to follow Kamada's algorithm (\cite{KamBook}, Section 21.4 p. 164-167) in order to construct a closed surface braid of this example. Figure~\ref{fig:ex12_1}(C)-(E) serves our purpose by moving the bands into a twisted position, and moving the subsequent twisted band diagram into a braid position following Kamada's method of applying the Alexander isotopy. In this case, the Alexander isotopy is fairly easy to implement. The second-from-the-top overarc, which is highlighted in item C, runs anticlockwise around the purported braid axis  $\otimes$ since arcs run predominantly clockwise about it. It is lifted and wrapped around to run along the bottom of the diagram. Then the bottom highlighted arc is pushed behind the diagram and over it, and the resulting banded braid is rearranged into braid form. Notice that both the braid and the band surgered braid are fully destabilizable unlinks; thus Figure~\ref{fig:ex12_1}(E) is a braided banded unlink diagram for $t_2(3_1)$. It is redrawn on the bottom right of the figure as a braid with its two bands drawn as if they were wrist watches. Their signs indicate the twisting directions.

\begin{figure}[ht]
\centering
\includegraphics[width=.9\textwidth]{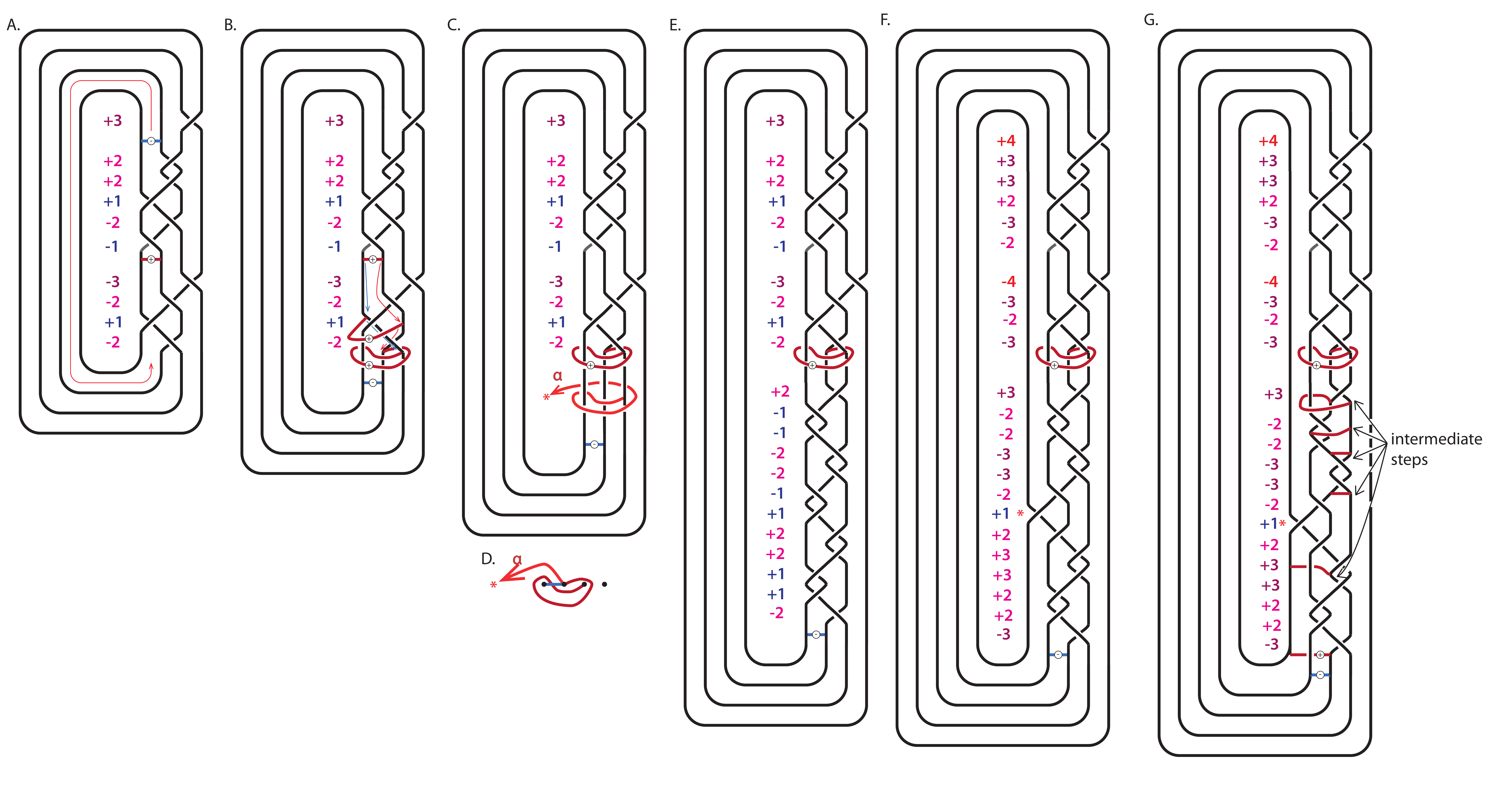}
\caption{Running Procedure~\ref{alg:rain_from_movies} on the braided banded unlink diagram of $t_2(3_1)$ from Figure~\ref{fig:ex12_1}(E). The output (G) is in braided banded bridge position.}
\label{fig:ex12_2}
\end{figure}

We now show how to follow Procedure~\ref{alg:rain_from_movies} to construct a rainbow diagram of this example. We keep track of the braided bands by drawing signed arcs instead. From Fig.~\ref{fig:ex12_2}A to B, the negative blue band has moved to the band on the bottom by rolling it over the braid closure. Meanwhile, B indicates how the red band can move down. When they are at the same disk level, they intersect (Fig.~\ref{fig:ex12_2}D). This intersection violates Condition~(a) in Definition~\ref{def:banded_bridge_position}.
Following Step 3(a) of the procedure, we find an arc $\alpha$ inside the disk-page that connects one common endpoint of the arcs with the boundary of the disk (also Fig.~\ref{fig:ex12_2}D). Then, as explained in Step 3(b), the arc $\alpha$ guides a Markov stabilization of the braid  (Fig.~\ref{fig:ex12_2}E-F). The red arc then slides downwards to lie near the blue arc (Fig.~\ref{fig:ex12_2}G). At this point, we have a diagram satisfying Definition~\ref{def:banded_bridge_position}.

In Figure~\ref{fig:ex12_4} the bridge splitting of the $2$-twist-spun trefoil is indicated via the diagram on the left. The resulting rainbow diagram for this $2$-knot is indicated upon the right. One can check that $\Tcal$ is, in fact, a strong rainbow diagram. Diagram $T_3$ indicates the result of removing the band, and diagram $T_2$ indicates the result of the surgery along the bands. The two closed braids, demonstrated below, represent trivial link diagrams. Indeed, a braid chart will be constructed from the simplifications thereof.

\begin{figure}[ht]
\centering
\includegraphics[width=.55\textwidth]{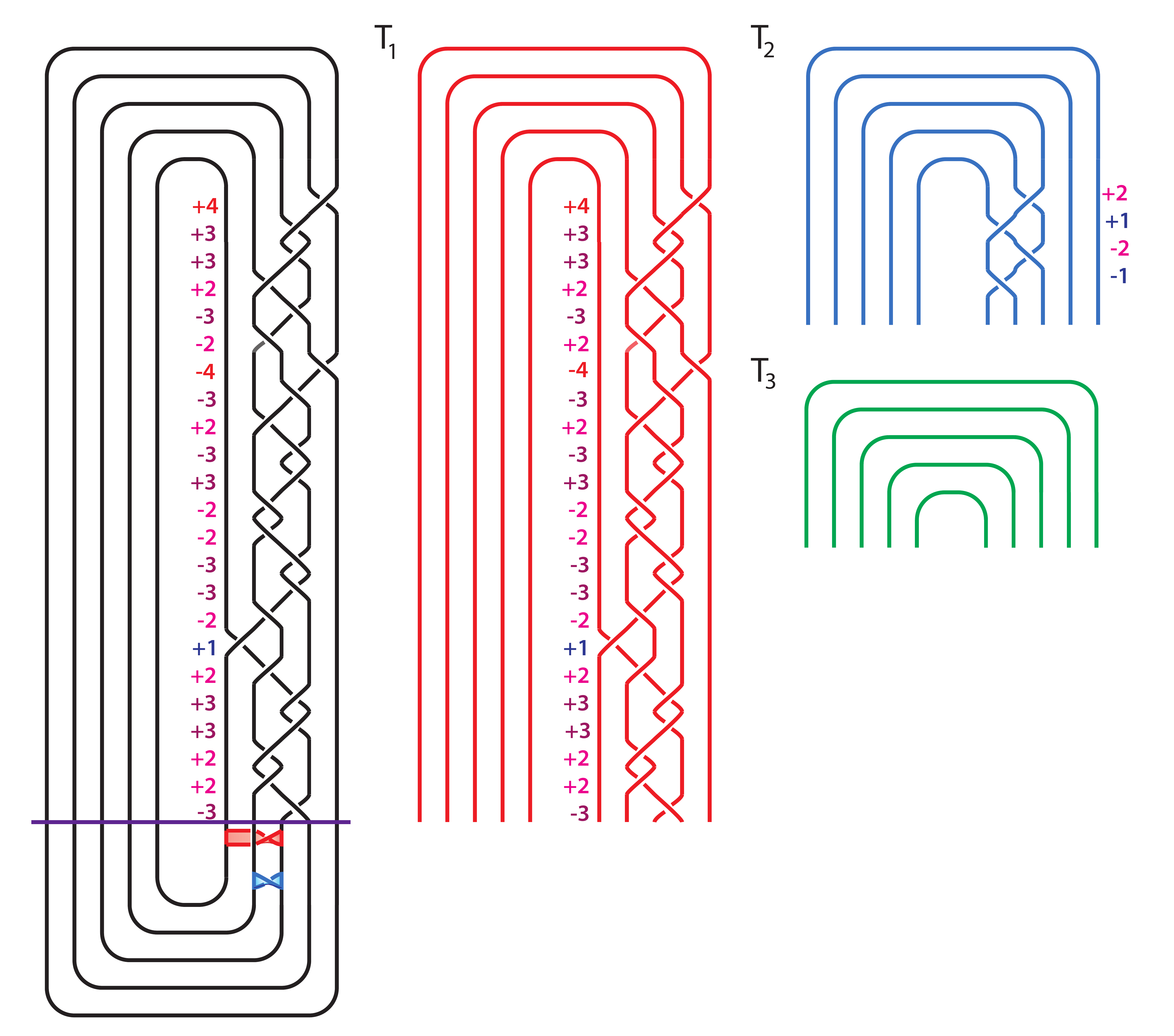}
\caption{(Left) Banded braided-bridge position for $t_2(3_1)$. (Right) Rainbow diagram for $t_2(3_1)$.}
\label{fig:ex12_4}
\end{figure}

A braid chart for $t_2(3_1)$ is obtained by using Procedure~\ref{alg:chart_from_rainbow} which, recall, explains how to go from a rainbow to a braid chart. For each pair of distinct indices $i,j\in\{1,2,3\}$, sequences of braid and saddle moves that turn $T_i\cup \T_{i+1}$ into a crossingless braid while tracing a trivial disk system bounded by $T_i\cup \T_{i+1}$ are shown in Figures~\ref{fig:ex12_T13_2} and~\ref{fig:ex12_T12_2}. In particular, Figure~\ref{fig:ex12_T13_2}  contains a movie and a chart for $T_2 \cup \T_3$ in the upper box while the main portion contains a movie $T_1\cup \T_3$; Figure~\ref{fig:ex12_T12_2} contains the movie for $T_1\cup \T_2$ and Figure~\ref{fig_chart_T12} contains the corresponding chart. The union of these three charts is the desired braid chart for $t_2(3_1)$ shown in Figure~\ref{DraftFullChart}.

As a sanity check, we can verify that our chart is indeed one for $t_2(3_1)$ by performing chart moves until we obtain one that matches an existing chart for $t_2(3_1)$, the one shown in Figure 21.15 of \cite{KamBook}. This is done in Figures~\ref{FullChart_Simplification_1}, \ref{FullChart_Simplification_2}, and \ref{FullChart_Simplification_3}. We refer the reader to Section 18.11 of \cite{KamBook} for details on chart moves. 
\begin{figure}[ht]
\centering
\includegraphics[width=.8\textwidth]{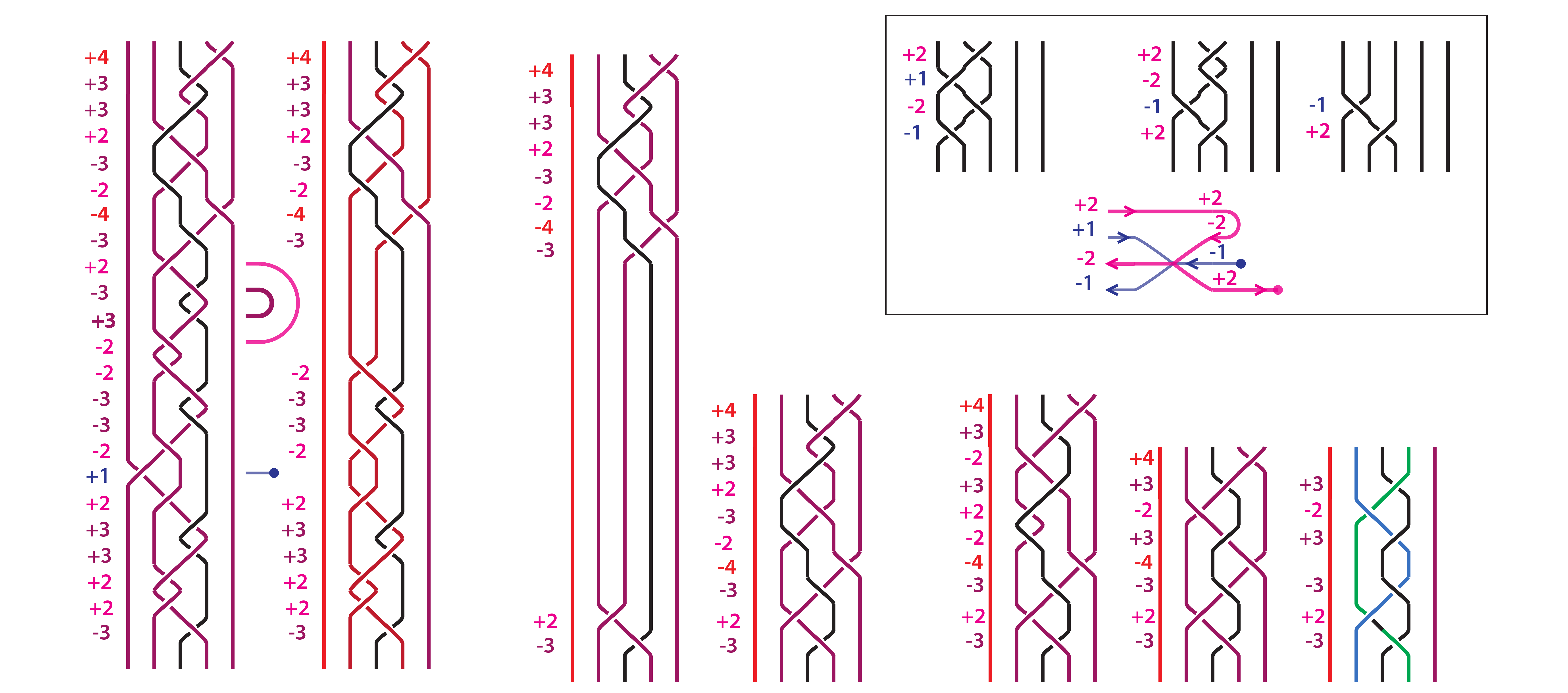}
\caption{Top box: Braid isotopy and resulting chart for a trivial disk system bounded by $T_2 \cup \T_3$. Main area: Sequence of braid isotopies and saddle moves that untie $T_1\cup \T_3$ while tracing a trivial disk system bounded by $T_1\cup \T_3$. Right area: the corresponding braid chart.}
\label{fig:ex12_T13_2}
\end{figure}
\begin{figure}[ht]
\centering
\includegraphics[width=.8\textwidth]{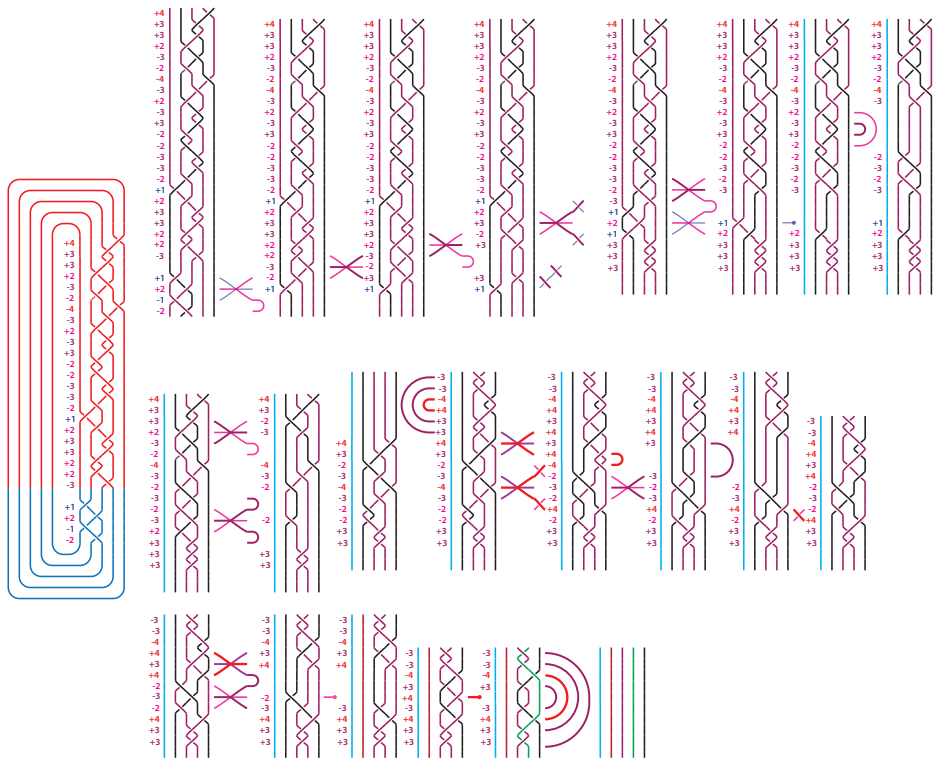}
\caption{Sequence of braid isotopies and saddle moves that untie $T_1\cup \T_2$ while tracing a trivial disk system bounded by $T_1\cup \T_2$.}
\label{fig:ex12_T12_2}
\end{figure}

\begin{figure}[ht]
\centering
\includegraphics[width=.8\textwidth]{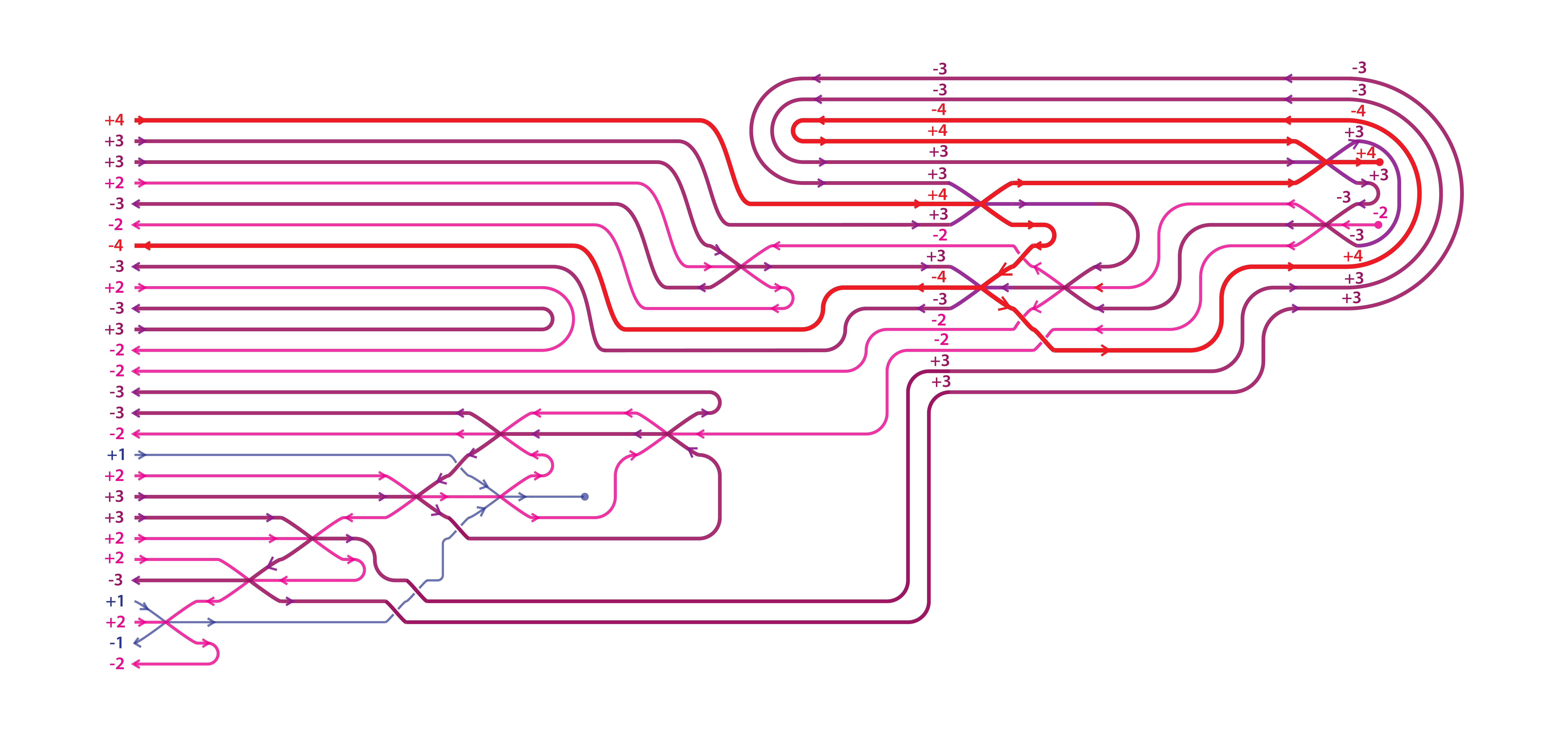}
 \caption{The chart that results from Fig.~\ref{fig:ex12_T12_2}.}
 \label{fig_chart_T12}
 \end{figure}

\begin{figure}[ht]
\centering
\includegraphics[width=.8\textwidth]{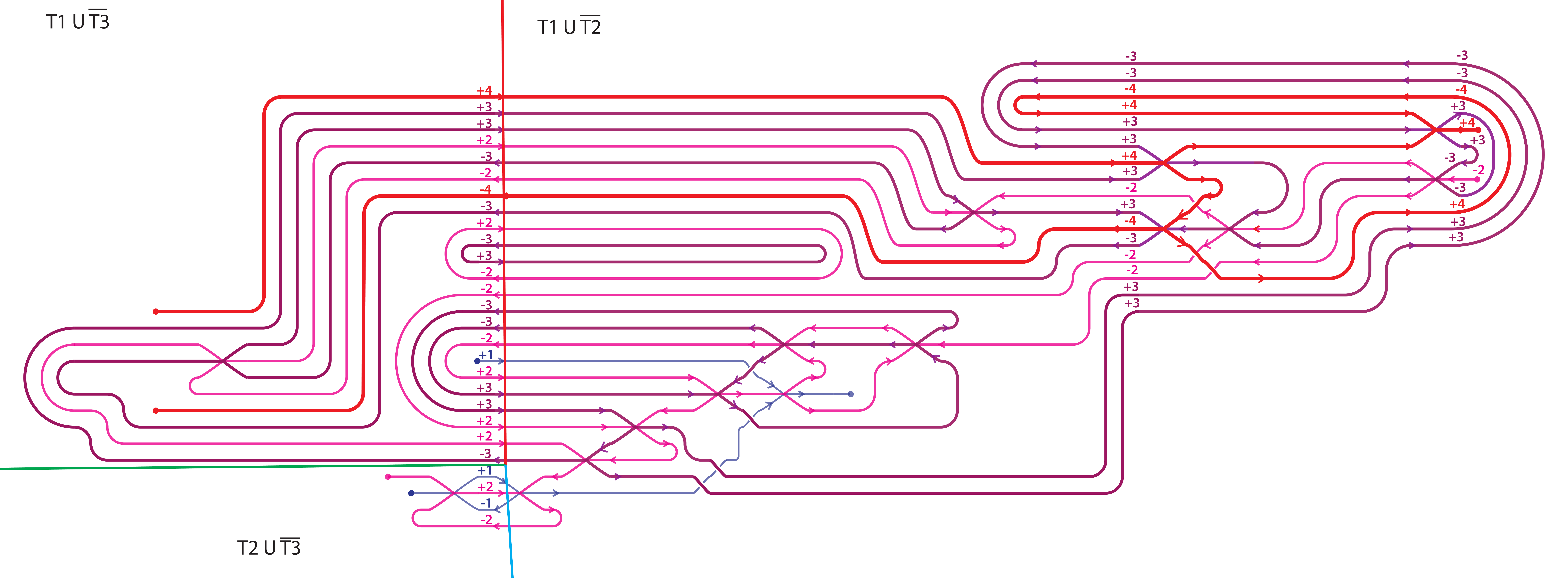}
 \caption{Braid chart for $t_2(3_1)$: The partial charts from Figs.~\ref{fig:ex12_T12_2} and~\ref{fig_chart_T12} are glued together.}
 \label{DraftFullChart}
 \end{figure}

\begin{figure}[ht]
\centering
\includegraphics[width=.8\textwidth]{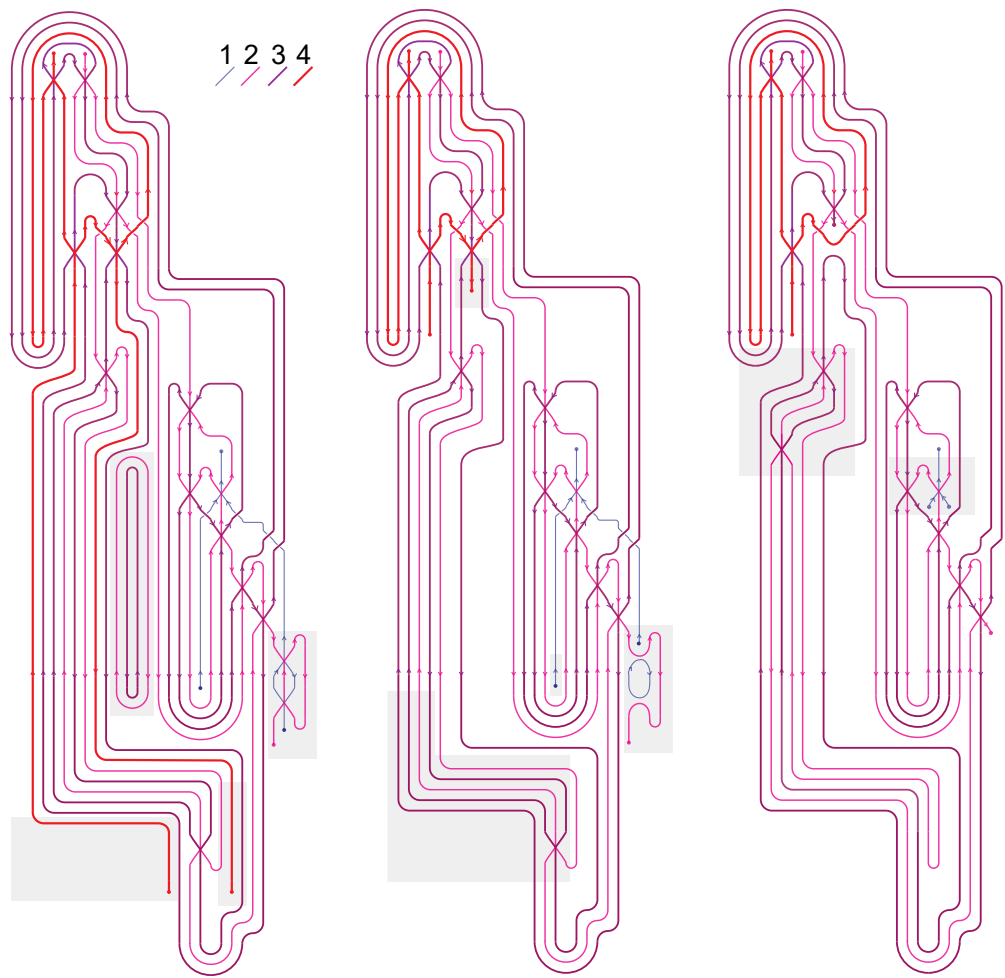}
\caption{Simplifying braid chart from Figure~\ref{DraftFullChart} (1/3): The shaded rectangles indicate the regions that will be modified in the following step.}
\label{FullChart_Simplification_1}
\end{figure}

\clearpage

\begin{figure}[ht]
\centering
\includegraphics[width=.8\textwidth]{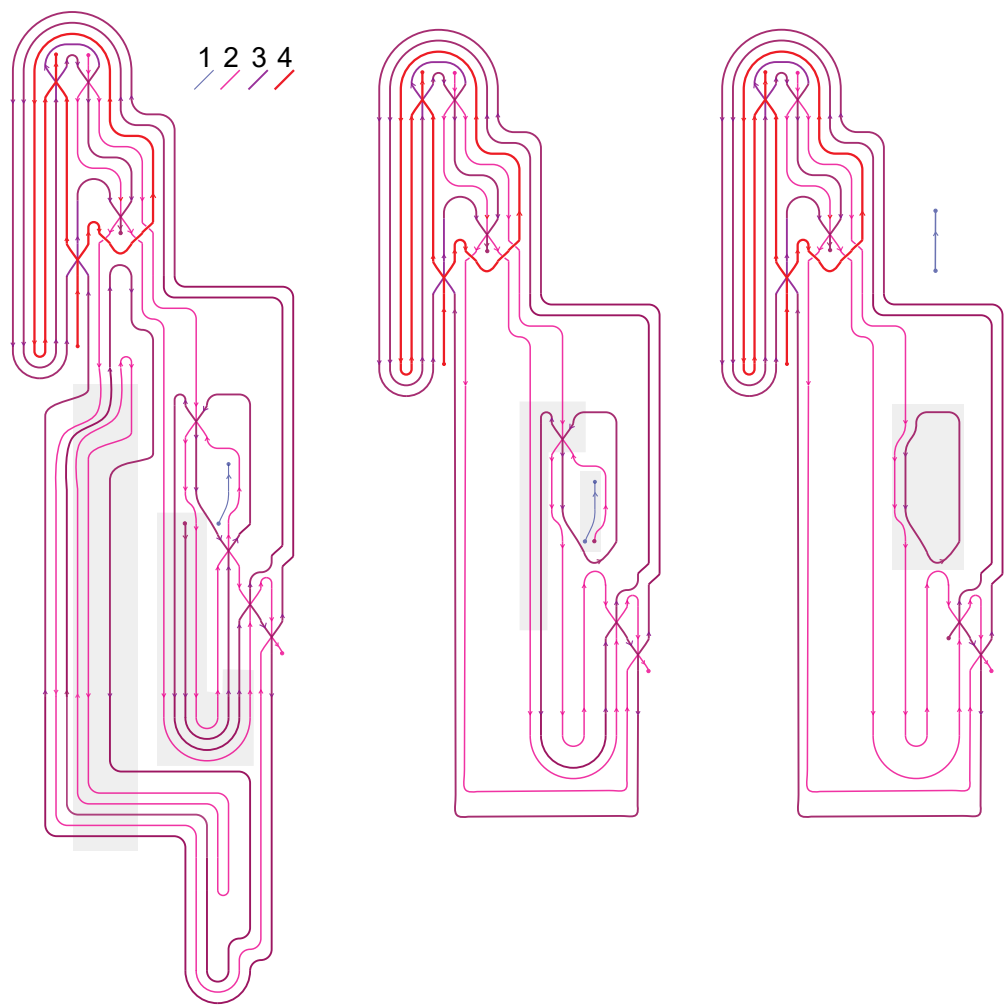}
\caption{Simplifying braid chart from Figure~\ref{DraftFullChart} (2/3): The shaded rectangles indicate the regions that will be modified in the following step.}
\label{FullChart_Simplification_2}
\end{figure}

\begin{figure}[ht]
\centering
\includegraphics[width=.6\textwidth]{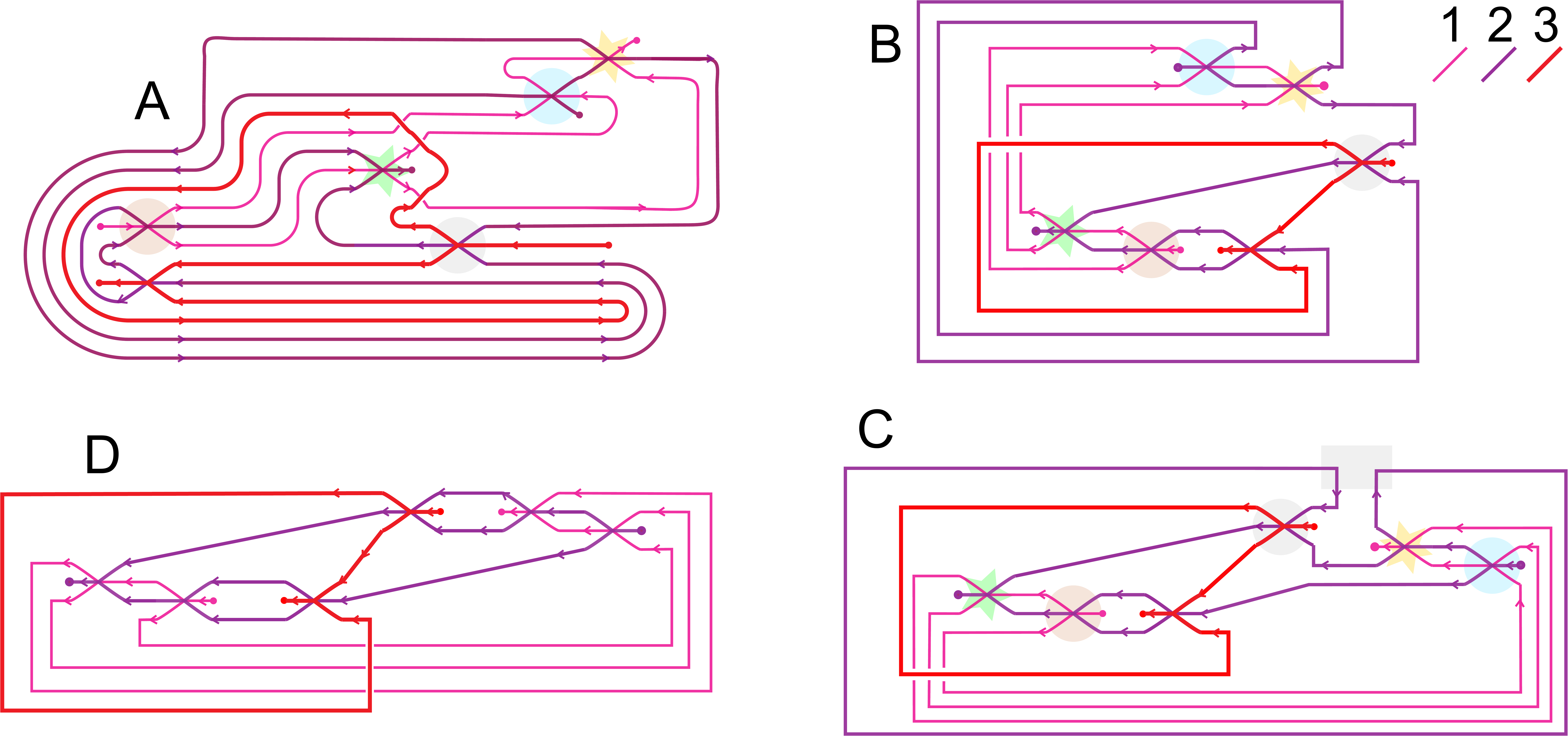}
\caption{Simplifying braid chart from Figure~\ref{DraftFullChart} (3/3): The shaded circles mark the same vertices throughout the different stages of the simplification process. Figure D is equal to Kamada's braid chart of $t_2(3_1)$ from \cite[Fig 21.15]{KamBook}.}
\label{FullChart_Simplification_3}
\end{figure}

\subsection{Rainbow number of Fox's Example 12}
In the previous section, we built an index-five rainbow diagram for $t_2(3_1)$. In this subsection, we compute the exact rainbow number. 

\begin{theorem}
    The rainbow number of the 2-twist spun trefoil is equal to four.
\end{theorem}
\begin{proof} 
It is known that the bridge index of $t_2(3_1)$ is equal to four; e.g. \cite[Ex 5.10]{BoundsKT2023}. Thus, by Equation~\refeq{eq:inequality1} in Section~\ref{sec:rainbow_diags}, we get that $4\leq \rain\left(t_2(3_1)\right)$. On the other hand, Figure~\ref{fig:Ex12_Meier}(A) shows a four-bridge triplane diagram of the 2-twist spun trefoil due to Meier~\cite{Meier_personal}. In the figure, we see how, after mutual braid moves and interior Reidemeister moves, we obtain an index-four weak-rainbow diagram $\Tcal$. Figures~\ref{fig:Ex12_Meier_chart1} and~\ref{fig:Ex12_Meier_chart2} show braid movies that make each braid $T_i\cup \T_{i+1}$ of $\Tcal$ into a trivial braid. Thus, $\Tcal$ is a rainbow diagram and $\rain\left(t_2(3_1)\right)=4$ as desired. 
\end{proof}
The careful reader may want an explanation of why Figure~\ref{fig:Ex12_Meier} describes Fox's Example 12. We verified this by drawing a braid chart from the rainbow diagram in Figure~\ref{fig:Ex12_Meier}(C) via Procedure~\ref{alg:chart_from_rainbow}. This is done in Figures~\ref{fig:Ex12_Meier_chart1} and~\ref{fig:Ex12_Meier_chart2}, and the final braid chart is shown in Figure~\ref{fig:Ex12_Meier_chart3}. At this point, one can check that our chart is equivalent, via a planar isotopy, to the braid chart for $t_2(3_1)$ in Figure~\ref{FullChart_Simplification_3}(A). To convince ourselves, we made the color schemes of both figures agree and marked corresponding vertices with shaded regions.  
\begin{figure}[h]
\centering
\includegraphics[width=.8\textwidth]{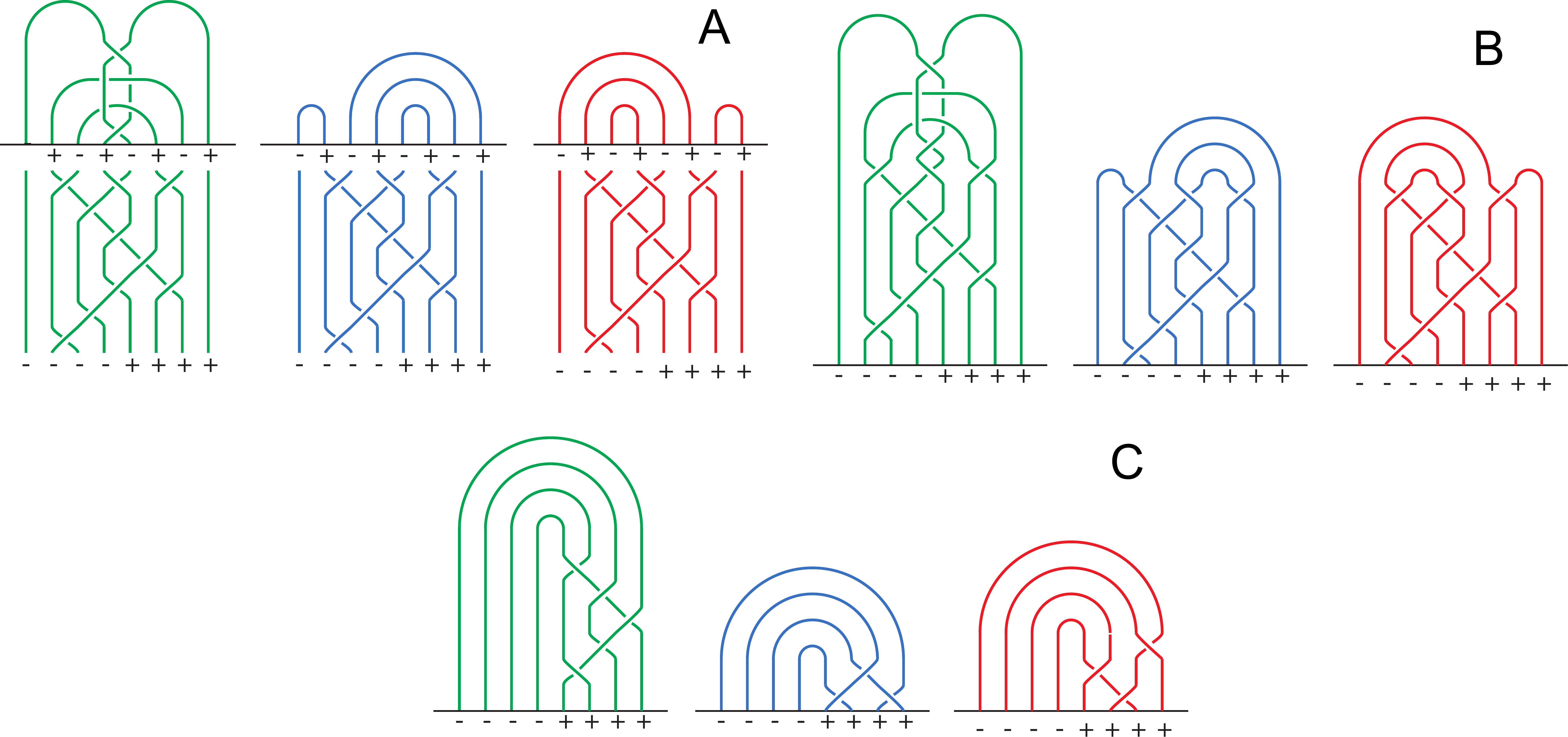}
\caption{2-twist spun trefoil. (A) Triplane diagram found by Meier \cite{Meier_personal}. (B) Mutual braid moves to cluster all the negative punctures to the left of the positive punctures. One can check that the result of interior Reidemeister moves (C) is a rainbow diagram.}
\label{fig:Ex12_Meier}
\end{figure}
\begin{figure}[h]
\centering
\includegraphics[width=.7\textwidth]{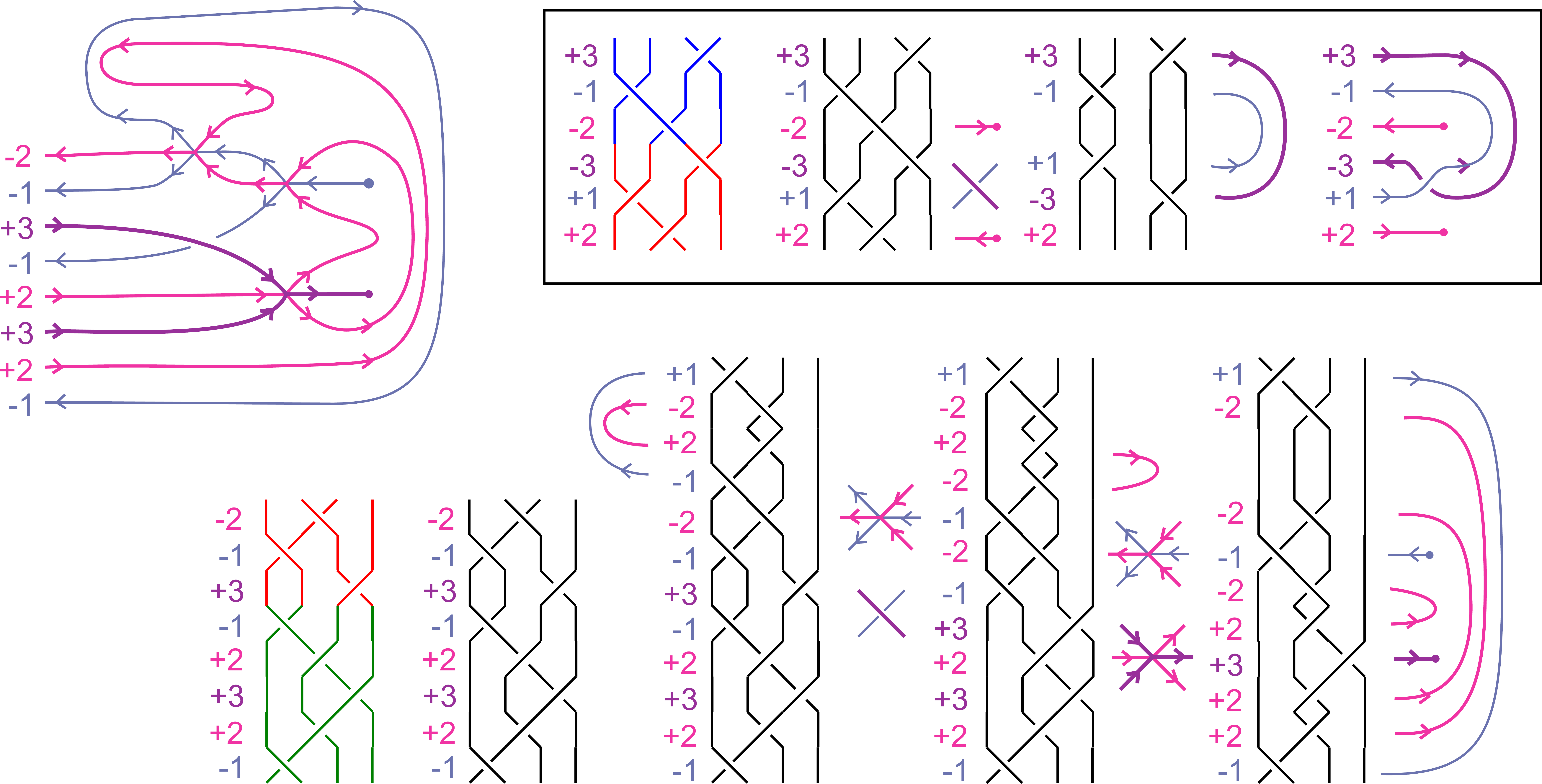}
\caption{Drawing a braid chart for rainbow in Figure~\ref{fig:Ex12_Meier}(C) (1/2). Braid isotopy and resulting chart for a trivial disk system bounded by the braid $\text{\color{red}red}\cup \overline{\text{\color{green}green}}$ and $\text{\color{blue}blue}\cup \overline{\text{\color{red}red}}$.}
\label{fig:Ex12_Meier_chart1}
\end{figure}

\begin{figure}[h]
\centering
\includegraphics[width=.85\textwidth]{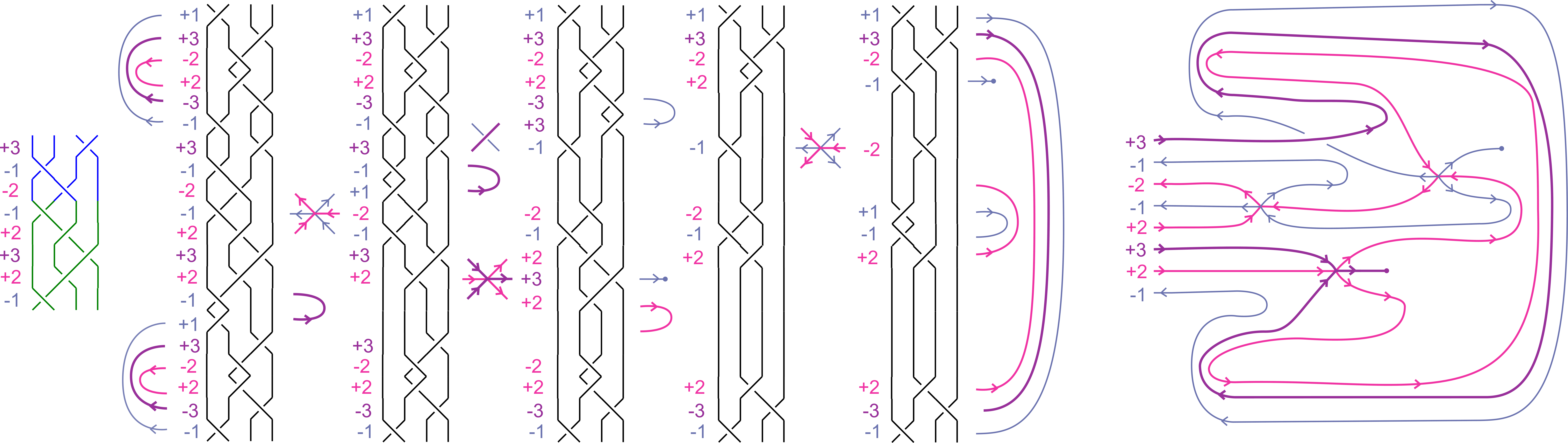}
\caption{Drawing a braid chart for rainbow in Figure~\ref{fig:Ex12_Meier}(C) (2/2). Braid isotopy and resulting chart for a trivial disk system bounded by the braid $\text{\color{blue}blue}\cup \overline{\text{\color{green}green}}$.}
\label{fig:Ex12_Meier_chart2}
\end{figure}

\begin{figure}[h]
\centering
\includegraphics[width=.85\textwidth]{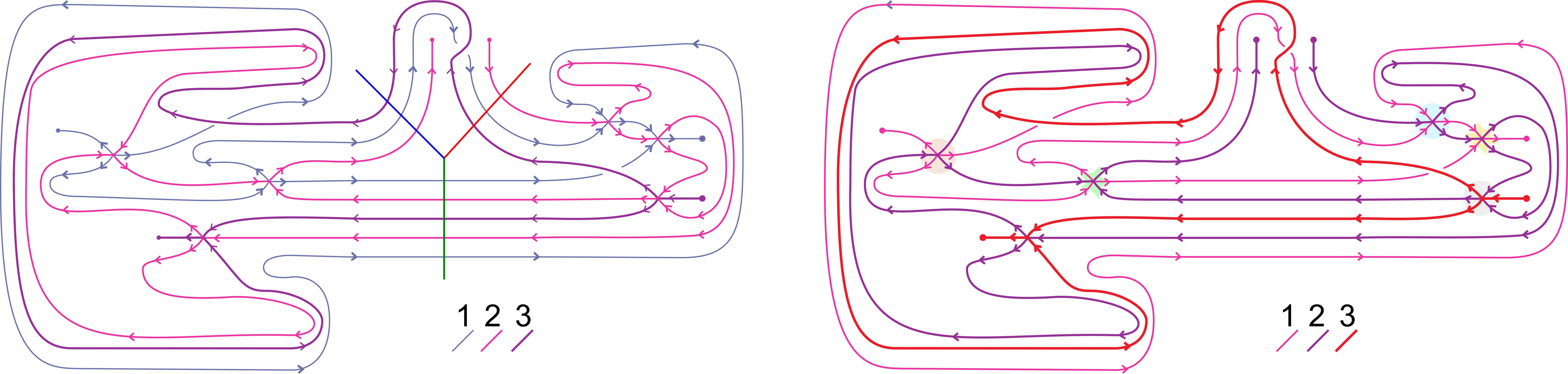}
\caption{Braid chart for rainbow in Figure~\ref{fig:Ex12_Meier}(C). In the right, we changed the color scheme to match that of Figure~\ref{FullChart_Simplification_3}. This way, our chart is equivalent (by a planar isotopy) to the chart in Figure~\ref{FullChart_Simplification_3}(A), we shaded the vertices of different colors to aid the reader.}
\label{fig:Ex12_Meier_chart3}
\end{figure}

\appendix

\section{Satellite surfaces}\label{appen:satellite}
Let $F$ be an orientable surface. Let $P$ be an embedded surface in $F\times D^2$ and let $C:F\rightarrow S^4$ be an embedding. Consider a diffeomorphism $\wt C:F\times D^2 \rightarrow N(C)$ such that the restriction to $F\times \{0\}$ is $C$. We call the image of $\wt C$ a \textit{satellite knot} with pattern $P$ and companion $C$. For more information on satellite surfaces, see \cite[Sec 2.4.2]{CKS_Surfaces_in_4space}. In this note, we present a procedure for drawing a triplane diagram for a satellite surface with a companion surface equal to a 2-sphere. 

The case when the companion has a positive genus is still open. The following appears to be a natural continuation of the work presented in this paper. 

\begin{problem}
Find a procedure to draw a triplane diagram for a cable surface of positive genus. 
\end{problem}

We briefly review what is known about cables of positive genus surfaces. The definition of cables of surfaces in the 4-sphere mimics that of knots in the 3-sphere: a cable surface for $F$ is a satellite surface with a companion surface equal to $F$ and a pattern that is a toroidal knotted surface. Following Hirose, a \emph{toroidal knotted surface surrounding $\Sigma_g$} is an incompressible horizontal surface inside $\Sigma_g\times S^1$, the boundary of a tubular neighborhood of a genus-$g$ unknotted surface in $S^4$ \cite{Hiro03}. It was shown that toroidal knotted surfaces surrounding tori are isotopic to $k$-twist spun torus knots with $k=0,1$\footnote{Although 1-twist spun knots are unknotted in $S^4$, they are not as embeddings in $T^2\times S^2$.}~\cite{Iwa88}. Hirose showed in \cite[Thm 1.2]{Hiro03} that toroidal surfaces surrounding higher genus surfaces are isotopic to a special connected sum of trivial $\Sigma_{g-1}$-knots and toroidal surfaces surrounding tori. 

\subsection{Triplane diagram of satellites with 2-knot companion}

A triplane diagram can be thought of as a triple of $b$-string tangle diagrams $\Pcal=(P_1,P_2,P_3)$ such that each $P_i$ is a collection of arcs properly embedded in the upper half plane $H$, where each crossing is decorated as an overcrossing or undercrossing. The boundaries of these arcs are $2b$ points on $\partial H$. An \textit{axis} of a triplane diagram is an embedding of the 0-sphere into $\partial H$ separating it into two arcs, each containing $b_i$ endpoints, where $2b=b_1+b_2$. An axis of a triplane diagram corresponds to a simple closed curve $\ell$ in the bridge sphere $\Sigma$. By embedding our triplane diagram into 4-space, we obtain the spine of a bridge trisection with bridge sphere $\Sigma\subset S^4$. In particular, an axis for $\Pcal$ determines a loop in $S^4$ that is disjoint from the bridge trisected surface. As $S^4-N(\ell)\approx D^2\times S^2$, we conclude the following result. As a consequence, we can use \emph{axed triplane diagrams} to represent patterns of satellite surfaces inside $S^2\times D^2$. 

\begin{lemma}\label{lem:axed_triplanes}
A triplane diagram, together with an axis, determines an embedded surface in $S^2\times D^2$. 
\end{lemma}

For a tangle $T$ inside a 3-ball $B^3$, a \emph{framing} $f$ is a collection of embedded strips $[-1,1]\times [-1,1]$ in $B^3$ with cores $\{0\}\times [-1,1]$ equal to the strands of $T$ and top and bottom sides $[-1,1]\times \{-1,1\}$ contained in $\partial B^3$. Equivalently, a framing $T$ is a choice of normal vector at each point of the tangle so the normal vector is tangent to $\partial B^3$ at the endpoints of $T$; the blackboard framing is an example. Two framed tangles with the same set of endpoints, and equal framings in $\partial B^3$, glue together to form a framed link. We say that a framing $f$ for the strands of a triplane diagram $\Ccal=(C_1,C_2,C_3)$ is \emph{a good framing for $\Ccal$} if for each $i\in \Z/3\Z$, each component of $C_i\cup \overline{C}_{i+1}$ is a zero-framed unknot. The reader can check that the blackboard framing is a good framing for Meier-Zupan triplane diagrams of spun knots; see \cite[Fig 20]{MZ17Trans}. Lemma~\ref{lem:frames_existence} states that every triplane diagram of an orientable surface admits good framings. 

\begin{figure}[ht]
\centering
\includegraphics[width=.75\textwidth]{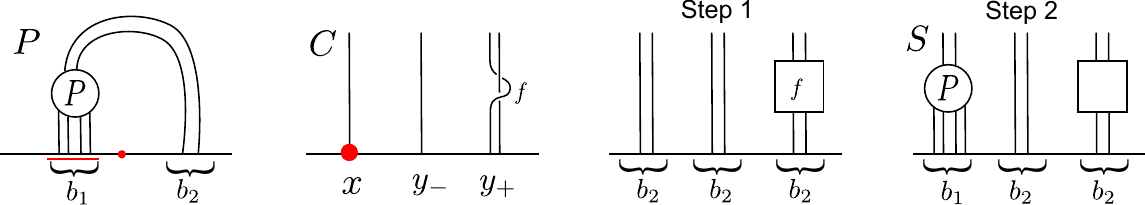}
\caption{How to draw a tangle $P$ inside a tangle $C$.}
\label{fig:satellite_1}
\end{figure}

We can embed axed tangles into other tangles using framings as follows. Let $(P, \ell)$ be a $b_P$-string tangle together with an axis, and let $C$ be an oriented $b_C$-string tangle. Pick a framing $f$ for the strands of $C$ and a positive endpoint $x$. Figure~\ref{fig:satellite_1} depicts a process to build a new tangle $S$ using the data $(P, \ell; C,f,x)$, which we explain now. The axis $\ell$ divides the $2b_P$ endpoints of $P$ into $b_1$ and $b_2$. First, you take $b_2$ parallel copies of $C$ framed with $f$ to get a $b_2b_C$-stranded tangle. We make the convention to add the $b_2$-stranded full twists (arising from the framing $f$) near the negative endpoints of $C$. Near the endpoint $x$, we now see $b_2$ vertical strands. Append in there the tangle $P$, we now see $b_1$ vertical endpoints near the boundary of the tangle. We obtain a $\left((b_2(b_C-1)+b_P\right)$-string tangle $S$ which we denote by $S(P, \ell; C,f,x)$. In Section~\ref{sec:examples} we give concrete examples of this procedure. 

\begin{definition} 
Let $(\Pcal,\ell)$ be a triplane with an axis. Let $\Ccal$ be an oriented triplane diagram with a framing $f$ and a fixed positive puncture $x$. Define $\Scal(\Pcal,\ell;\Ccal,f,x)$ to be the triplet of tangles $(S_1,S_2,S_3)$ where each $S_i$ is equal to $S(P_i,\ell;C_i,f,x)$. 
\end{definition}

We are ready to state the main theorem of this note. 

\begin{theorem}\label{thm:sat}
Let $\Ccal$ be a triplane diagram for a 2-sphere $C$ and let $\Pcal$ be a triplane diagram with axis $\ell$. Given a good framing $f$ of $\Ccal$, the tuple $\Scal(\Pcal,\ell;\Ccal,f,x)$ is a satellite knot with pattern $P$ and companion $C$, for any puncture $x$ of $\Ccal$.
\end{theorem}

\subsection{Examples}\label{sec:examples}
We draw three satellite surfaces for the same companion 2-knot and different patterns. The companion surface will be Yoshikawa's $10_2$ surface \cite{Yoshikawa} --- also known as the 2-twist spun trefoil. The oriented triplane diagram $\Ccal$ for $10_2$ we use is one from \cite[Fig 19]{ZupanREU2021}. In Figure~\ref{fig:102_framing}, we note that the blackboard framing on $\Ccal$ is not good as one component of $C_1\cup \overline{C}_2$ (resp. $C_2\cup \overline{C}_3$) has writhe equal to $+1$ (resp. $-1$). This can be fixed by setting the framing of one arc of $C_2$ to be $+1$ instead of blackboard; we drew this special arc in a thicker red font. 
With this framing $f$ in mind, Figure~\ref{fig:fig_rp2_satell} draws the triplane $\Scal(\Pcal, \ell;\Ccal,f)$ where $\Pcal$ represents an unknotted projective plane. 

\begin{figure}[ht]
\centering
\includegraphics[width=.6\textwidth]{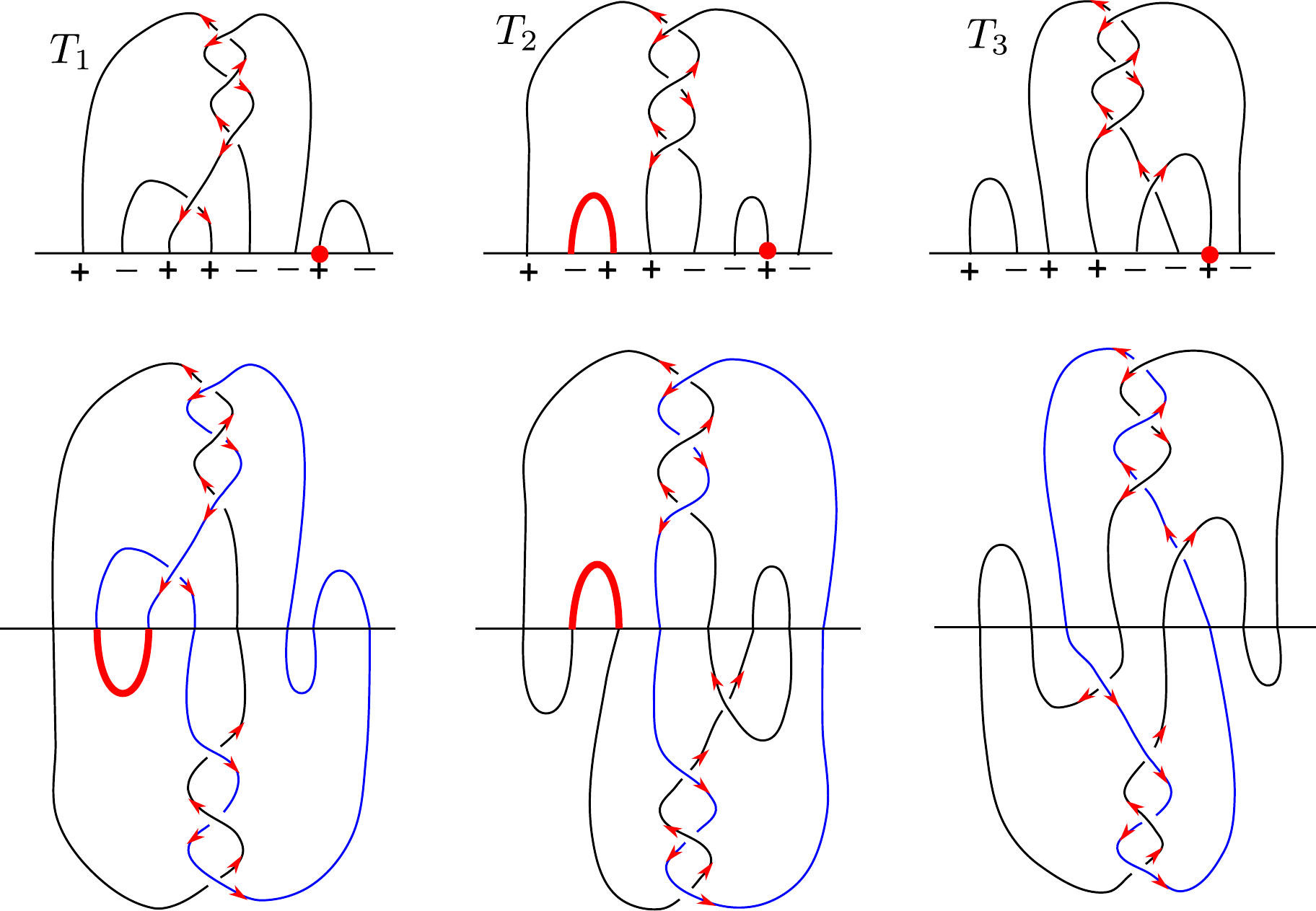}
\caption{(top) Triplane diagram $\Ccal$ for $10_2$. (bottom) Unlink diagrams of $C_i\cup \overline{C}_{i+1}$; for the blackboard framing to be good, the writhe of each component of each link should be zero.}
\label{fig:102_framing}
\end{figure}

\begin{figure}[ht]
\centering
\includegraphics[width=.7\textwidth]{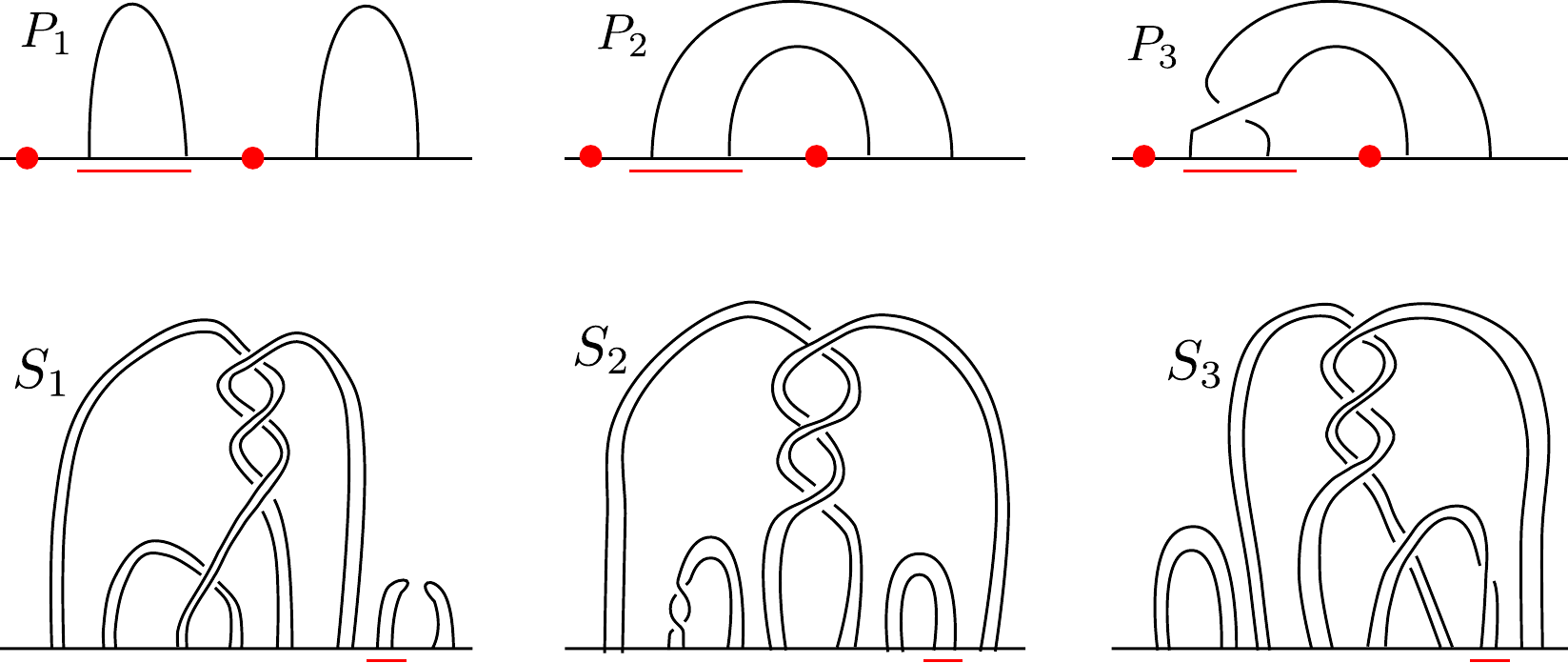}
\caption{(top) A triplane diagram with an axis (red dots) for an unknotted projective plane. (bottom) Satellite surface for $\Sigma(\mathbb{RP}^2,10_2)$.}
\label{fig:fig_rp2_satell}
\end{figure}

Sometimes, the pattern surface may be described as a banded unlink diagram $(L,\nu)$, together with a marked unknot $\ell$. As with triplane diagrams, $(L,\nu)$ describes a surface in $S^4-N(\ell)\approx S^2\times D^2$. When this happens, we can follow the procedure in the proof of Theorem 1.3 of \cite{MZ17Trans} to make the given banded unlink into a triplane diagram in such a way that the loop $\ell$ is an axis for the resulting triplane diagram. See Figure~\ref{fig:whitehead} for an example of this. 
\begin{example}[Whitehead torus pattern]\label{ex:whitehead}
Consider the Whitehead virtual link in Figure~\ref{fig:whitehead} and tube one of its components to get an unknotted torus $W$ in $S^4$. The second component will be a loop $\ell$ in the complement of $W$. This surface is drawn in the second panel of the figure as a banded unlink diagram. Following the discussion on trisecting tube-maps from \cite[Sec 3.4]{Joseph_Classical_knot_theory}, the virtual diagram can be turn into the triplane diagram $\mathcal{W}$ from Figure~\ref{fig:whitehead}(D). The red dots correspond to the curve $\ell$, which is now an axis for the triplane. Figure~\ref{fig:whitehead_satel} is the diagram $\Scal(\mathcal{W},\ell; \Ccal,f)$ for the satellite torus $\Sigma(W,10_2)$.
\end{example}

\begin{figure}[ht]
\centering
\includegraphics[width=.7\textwidth]{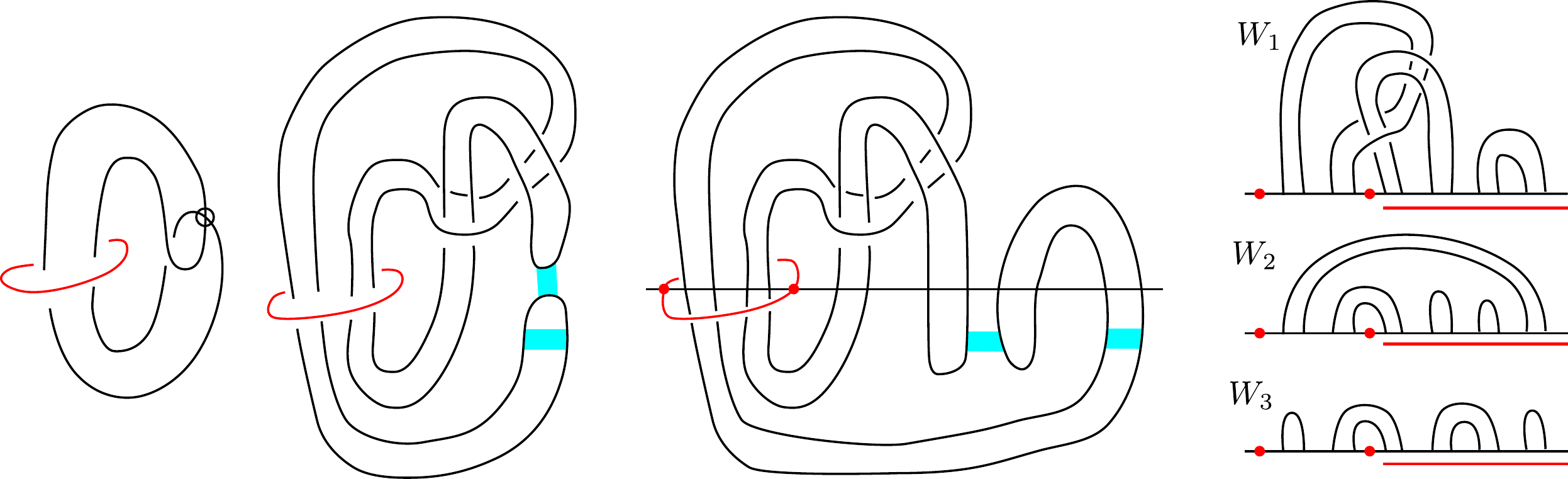}
\caption{Pattern of torus $W$ in $S^2\times D^2$ obtained from a virtual Whitehead link. (B)-(C) Banded unlink diagrams for $F$. (D) Triplane diagram for $W$ obtained from the banded presentation in (C); see \cite[Sec 3.4]{Joseph_Classical_knot_theory} for details.}
\label{fig:whitehead}
\end{figure}

\begin{figure}[ht]
\centering
\includegraphics[width=.9\textwidth]{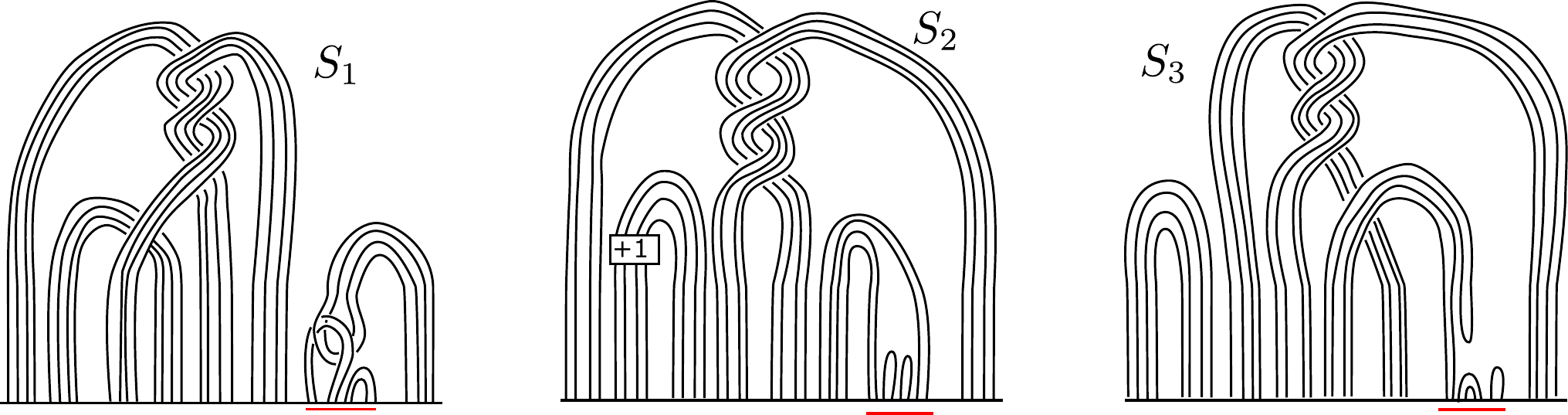}
\caption{Satellite surface for $\Sigma(W,10_2)$ from Example~\ref{ex:whitehead}.}
\label{fig:whitehead_satel}
\end{figure}

Following Kim, a cable 2-knot is a satellite surface with a 2-knot $C$ as a companion and a pattern equal to an unknotted 2-sphere $U$ in $S^4$~\cite{kim20glucktwist}. As before, $U$ is a surface in $S^4$ disjoint from a distinguished loop $\ell$. Cables of $C$ are parametrized by the set of integers~\cite[Lem 2.8]{kim20glucktwist}: the $m$-cable pattern is an unknotted sphere $U$ in which $[\ell]=m\in \pi_1\left(S^4-U\right)$. 

\begin{example}[Triplane of a cable 2-knot]
Figure~\ref{fig:cable} contains an example of a triplane diagram for a 3-cable 2-knot. To see why the trisected surface for the pattern in Figure~\ref{fig:cable}(A) represents a 3-cable pattern, the reader must check that (1) this triplane diagram simplifies with triplane moves to a 1-bridge diagram (thus $U$ is unknotted in $S^4$), and (2) the loop $\ell$ is homotopic to the product of three meridians of $U$ (which are equal to meridians of the tangles in the triplane). This generalizes to any $m$-cable 2-knots. 
\end{example}

\begin{figure}[ht]
\centering
\includegraphics[width=.7\textwidth]{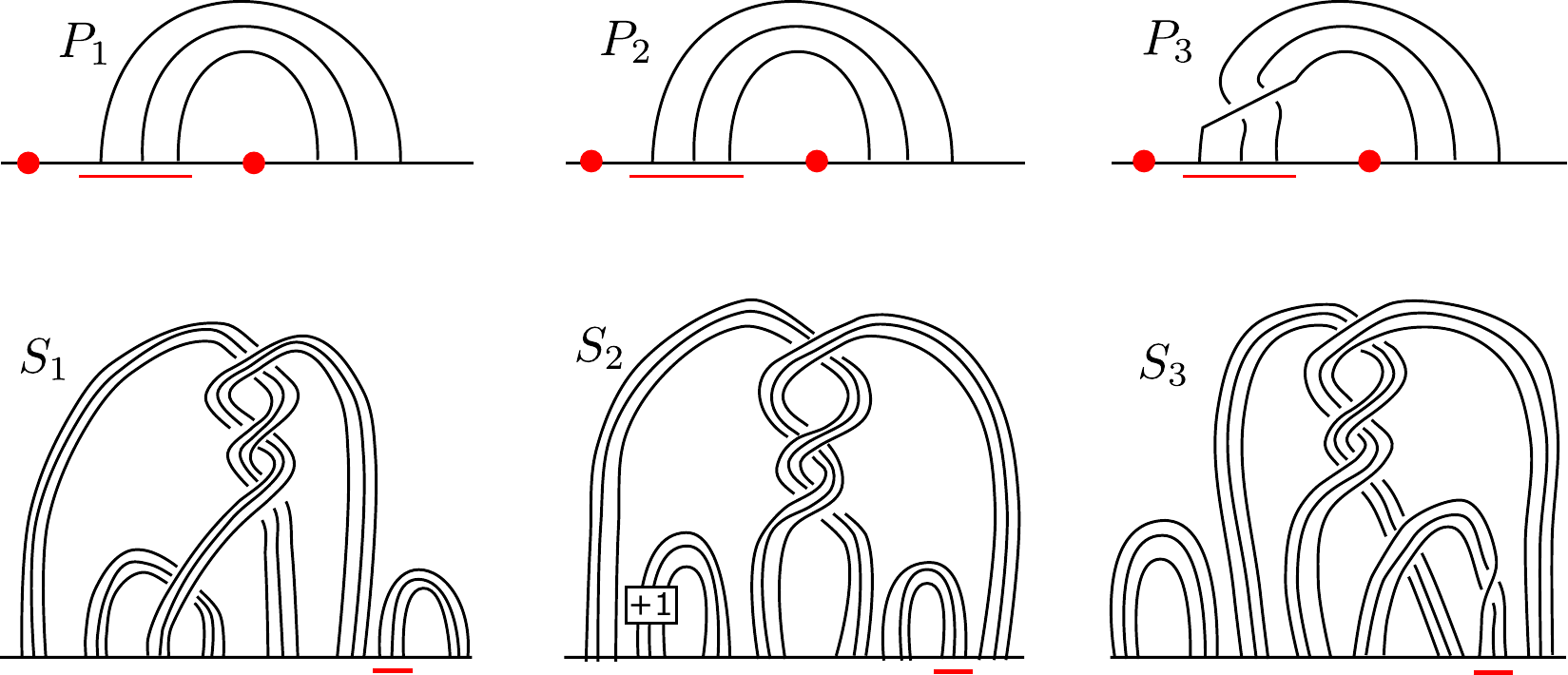}
\caption{(left) A triplane diagram with an axis (red dots) for a 3-cable pattern: this is an unknotted 2-knot $U$ in which the axis $\ell$ winds three times around the meridians of $U$. (right) Triplane diagram for the 3-cable 2-knot of $10_2$, this is the satellite surface $\Sigma(U,10_2)$.}
\label{fig:cable}
\end{figure}

\subsection{Proofs}
\begin{proposition}\label{prop:TPC_is_triplane}
Suppose that $f$ is a good framing. Then $\Scal(\Pcal,\ell;\Ccal,f,x)$ is a triplane diagram. 
\end{proposition}
\begin{proof}
Denote by $(S_1,S_2,S_3)$ the tangles of $\Scal(\Pcal,\ell;\Ccal,f,x)$. We will show that each pairwise union $S_i \cup \overline{S}_{i+1}$ is an unlink. Consider the tri-plane diagram $(C_1,C_2,C_3)$ for the companion. By assumption. $C_i \cup \overline{C}_{i+1}$, is an unlink. In particular, after some mutual braid and interior Reidemeister moves, we can turn each pairwise union $C_i \cup \overline{C}_{i+1}$ into a crossing-less diagram of an unlink $U$. As $S_i\cup \overline{S}_{i+1}$ lies in a neighborhood of $C_i \cup \overline{C}_{i+1}$, this isotopy extends to an ambient isotopy of $S^3$ which we use to move $S_i\cup \overline{S}_{i+1}$ around. As each component of $C_i \cup \overline{C}_{i+1}$ is 0-framed, the framing $f$ becomes the blackboard framing on $U$. In particular, $S_i\cup \overline{S}_{i+1}$ is equal to the distant sum of $P_i\cup \overline{P}_{i+1}$ and some 0-framed cables of the unlink; see Figure~\ref{fig:satellite_2}. As $P_i\cup \overline{P}_{i+1}$ is an unlink, the result follows. 
\end{proof}
\begin{figure}[ht]
\centering
\includegraphics[width=.65\textwidth]{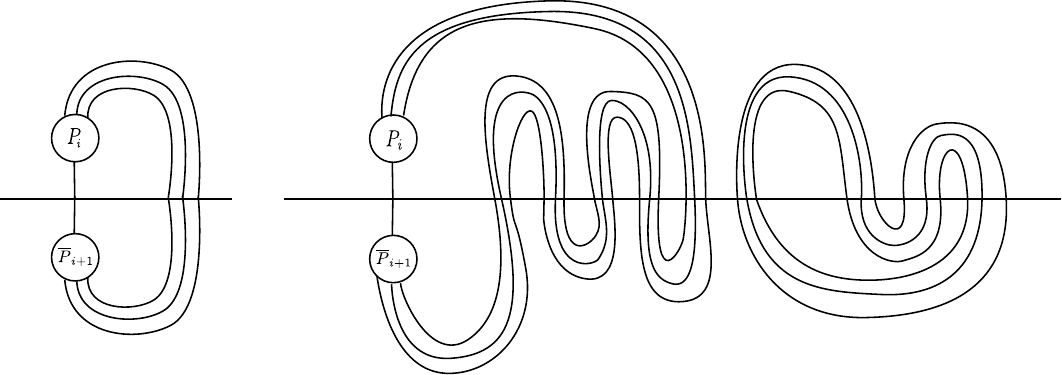}
\caption{(Left) A schematic diagram of $P_i\cup \overline{P}_{i+1}$. (Right) A schematic diagram of $U$: the union of 0-framed cables of the unlink and a 0-framed satellite of the unknot with pattern $P_i\cup \overline{P}_{i+1}$.}
\label{fig:satellite_2}
\end{figure}
\begin{proof}[Proof of Theorem~\ref{thm:sat}]
By Proposition~\ref{prop:TPC_is_triplane}, $\Scal(\Pcal,\ell;\Ccal,f,x)=(S_1,S_2,S_3)$ is a triplane diagram for an embedded surface $S$ in $S^4$. We now explain that $S\subset N(C)$.
One way to visualize an embedding of a trivial disk system bounded by an unlink $L$ is to consider a movie of isotopies that transform $L$ into a crossingless diagram, and then fill in this diagram with the obvious disks it bounds. As we saw in the proof of Proposition~\ref{prop:TPC_is_triplane}, there is an isotopy turning $S_i\cup \overline{S}_{i+1}$ into a crossingless diagram by first performing ambient isotopies that untangle a neighborhood of $C_i\cup \overline{C}_{i+1}$, followed by isotopies that untangle $P_i\cup \overline{P}_{i+1}$. The second round of isotopies (from $P_i\cup \overline{P}_{i+1}$ to a crossingless diagram) can be chosen to occur in a neighborhood of the disks bounded by the crossingless diagram for $C_i\cup \overline{C}_{i+1}$ (see Figure~\ref{fig:satellite_2}). Hence, the disk system bounded by $P_i\cup \overline{P}_{i+1}$ lies inside a neighborhood of the disk system for $C_i\cup \overline{C}_{i+1}$. Therefore, $S\subset N(C)$ and $S$ is a satellite surface with companion $C$. 

To show that $S$ has pattern $P$, we need to exhibit a diffeomorphism from $N(C)$ to $S^2 \times D^2$ sending $S$ to $P$. We do this by finding a map between the spines of $\Scal(\Pcal,l;\Ccal,f,x)$ and $\Pcal$ that sends a meridian of $C$ (containing the strands of $\Pcal$) to the axis $\ell$. As the bridge trisected surfaces are determined by their spines (or triplane diagrams), we can extend our map to the desired diffeomorphism $N(C)\to S^2 \times D^2$.

By Corollary 1.2 of \cite{MTZ_cubic_graphs}, $\Ccal$ can be converted into a triplane diagram $\Ucal_0$ for an unknotted 2-sphere by a sequence of interior Reidemeister moves and crossing changes. By taking the tubular neighborhood of the arcs in $\Ccal$, we can extend such a sequence to a sequence of interior Reidemeister moves and crossing changes that make $\Scal(\Pcal,\ell;\Ccal,f,x)$ into a triplane diagram for $\Scal(\Pcal,\ell;\Ucal_0,f_0,x_0)$. Notice that we may need to choose a different good framing $f_0$ that works for $\Ucal_0$; this is possible by Lemma~\ref{lem:frames_existence}. Now, as $\Ucal_0$ is a triplane for an unknotted 2-sphere, Theorem 1.7 of \cite{MZ17Trans} implies that $\Ucal_0$ can be turned into a one-bridge trisection $\Ucal_1$ by a finite sequence of triplane moves. Again, we can take the tubular neighborhood of the tangles for $\Ucal_0$ and extend this sequence of moves to a sequence of triplane moves taking $\Scal(\Pcal,\ell;\Ucal_0,f_0,x_0)$ into $\Scal(\Pcal,\ell;\Ucal_1,f_1,x_1)$ for some new good framing $f_1$ for $\Ucal_1$. This last diagram $\Scal(\Pcal,\ell;\Ucal_1,f_1,x_1)$ agrees with $\Pcal$. 
\end{proof}

\subsection{Writhes}
The goal of this subsection is to prove the existence of good framings on an orientable triplane diagram. The proof of Lemma~\ref{lem:frames_existence} provides an algorithm to find a good framing.  
\begin{lemma}\label{lem:frames_existence}
    Let $\Tcal$ be a triplane diagram for an orientable surface $F\subset S^4$. There exist framings for the tangles of $\Tcal$ such that each component of $T_i\cup \T_{i+1}$ is a 0-framed unknot. 
\end{lemma}

Let $\Tcal=(T_1,T_2,T_3)$ be a fixed triplane diagram for an orientable surface $F$ and denote by $L_i=T_i\cup \T_{i+1}$. Fix an orientation of $F$ which induces an orientation on the tangles of $\Tcal$ by \cite[Lem 2.1]{MTZ_cubic_graphs}. Recall that the writhe of a link diagram is the sum of the signs of all crossings in the diagram. Crossings between different components of the link also contribute to the linking number of such a pair of circles. Now, as $L_i$ is an unlink, the pairwise linking number between its components is always zero. Thus, we can rewrite the write of $L_i$ as 
\begin{equation}\label{eq:writhes}
    writhe(L_i) = \displaystyle \sum_{j=1}^{c_i} writhe(l_{ij}), 
\end{equation}
where $\{ l_{ij}\}_{j=1}^{c_i}$ are the connected components of $L_i$. 
Consider the blackboard framing of each tangle of $\Tcal$. 
To prove Lemma~\ref{lem:frames_existence}, we will show that we can change the framing of some strands of $\Tcal$ by adding an even number of minor kinks (think of the belt-trick) so that the writhe of each component of each $L_i$ is zero. An immediate obstacle is that changing the framing of one arc in $T_i$ alters the writhe in two circles: one in $L_i$ and one in $L_{i-1}$. 

Consider the abstract surface $F$ together with the cell decomposition induced by its bridge trisection $F=\Dcal_1\cup \Dcal_2\cup \Dcal_3$. The edges are the pairwise intersections of the disks, colored red, blue, and green for $\Dcal_3\cap \Dcal_1$, $\Dcal_1\cap \Dcal_2$, and $\Dcal_2\cap \Dcal_3$, respectively. We assign integer labels to each edge and 2-cell of $F$. The 2-cell label will be equal to the writhe of its corresponding boundary component, $writhe(l_{ij})$. Recall that the edges of $F$ are oriented as $\Tcal$ is, so traversing the edge in the opposite direction will assign the negative of its edge label. We define the \emph{weight of a 2-cell} to be the sum of its 2-cell label and the edge labels of its boundary (traversed using the boundary orientation of the 2-cell).

An edge label of $a\in \Z$ will correspond to adding $a$ full twists to the blackboard framing of the corresponding strand; negative numbers translate to negative twists. This way, choosing new framings on $\Tcal$ amounts to choosing an edge labeling of $F$. The writhe of the newly framed link $L_i$ will be equal to the weight of the corresponding 2-cell of $F$. Equipped with this notation, we are ready to prove Lemma~\ref{lem:frames_existence}.

\begin{proof}[Proof of Lemma~\ref{lem:frames_existence}]
From the discussion above, it is enough to find a set of edge labels for $F$ such that all the 2-cell weights are equal to zero. Start by assigning zero to all edge labels of $F$; this is the same as starting with the blackboard framing on $\Tcal$. We can change labels for some blue edges such that the blue-green disks have zero weight. Similarly, by selecting one green edge per green-red 2-cell, we can modify some of the green edge labels so that the green-red disks have zero weight. We are left with all blue-green and green-red 2-cells having a weight equal to zero. 

Figure~\ref{fig:framings} depicts a modification on the set of red edge labels that preserves the blue-green and green-red weights while changing the weights of the red-blue disks in a controlled way: one disk gains one weight unit while another disk loses one unit. 
Precisely, the modification goes as follows: consider two red-blue disks $D$ and $D'$ that are adjacent (share an edge) to the same disk $D''$. Let $a$ and $b$ be the red labels of one edge of $D\cap D''$ and $D'\cap D''$, respectively. Then, the edge label modification involves replacing $a$ with $a-1$ and $b$ with $b+1$, while keeping the rest of the labels unchanged. 
\begin{figure}[ht]
\centering
\includegraphics[width=.6\textwidth]{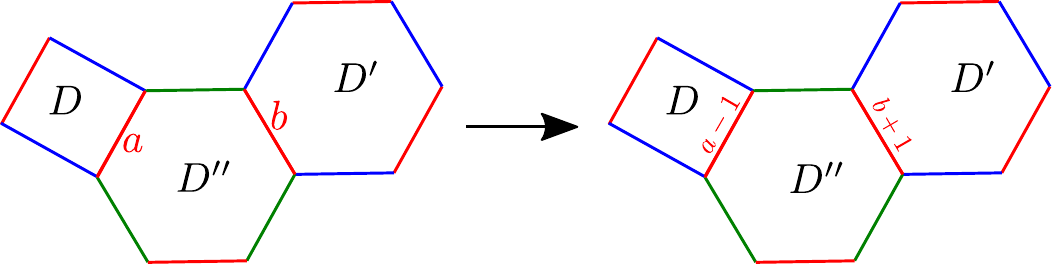}
\caption{How to relabel the edges of $F$.}
\label{fig:framings}
\end{figure}
Using the edge label modification above, we can distribute the red-blue weights to minimize the number of red-blue disks with non-zero weight. On the other hand, from Equation~\eqref{eq:writhes}, we get that 
\[ 
\sum_{i=1}^3 \sum_{D\subset \Dcal_i} weight(D) = \sum_{i=1}^3 writhe(L_i) = e(F) = 0,
\]
where the last equation is Corollary 3.8 of \cite{Joseph_Classical_knot_theory}. In particular, the net-weight of the red-blue disks is equal to zero, 
\[ 
\sum_{D\subset \Dcal_1} weight(D) = -\sum_{D\subset \Dcal_2} weight(D) - \sum_{D\subset \Dcal_3} weight(D) = 0.
\]
We conclude that all the weights of the red-blue 2-cells can made equal to zero, while preserving the blue-green and green-red zero weights. 
\end{proof}

\bibliographystyle{alpha}
\bibliography{biblio}

\newcommand{\etalchar}[1]{$^{#1}$}
\def\cprime{$'$}
\begin{thebibliography}{AAD{\etalchar{+}}23}

\bibitem[AAD{\etalchar{+}}23]{ZupanREU2021}
Wolfgang Allred, Manuel Arag\'on, Zack Dooley, Alexander Goldman, Yucong Lei,
  Isaiah Martinez, Nicholas Meyer, Devon Peters, Scott Warrander, Ana Wright,
  and Alexander Zupan.
\newblock Tri-plane diagrams for simple surfaces in {$S^4$}.
\newblock {\em J. Knot Theory Ramifications}, 32(6):Paper No. 2350041, 28,
  2023.

\bibitem[ACC{\etalchar{+}}on]{bridge_divides}
Román Aranda, Patricia Cahn, Marion Campisi, Agniva Roy, and Melissa Zhang.
\newblock Bridge multisections with divides.
\newblock In preparation.

\bibitem[Ada94]{knotbook}
Colin~C. Adams.
\newblock {\em The knot book}.
\newblock W. H. Freeman and Company, New York, 1994.
\newblock An elementary introduction to the mathematical theory of knots.

\bibitem[AE24]{AE24}
Román Aranda and Carolyn Engelhardt.
\newblock Bridge multisections of knotted surfaces in $s^4$.
\newblock {\em arXiv:2410.01921}, 2024.

\bibitem[APZ23]{BoundsKT2023}
Rom\'an Aranda, Puttipong Pongtanapaisan, and Cindy Zhang.
\newblock Bounds for {K}irby-{T}hompson invariants of knotted surfaces.
\newblock {\em Geom. Dedicata}, 217(6):Paper No. 99, 30, 2023.

\bibitem[Art25]{Artin}
Emil Artin.
\newblock Zur {I}sotopie zweidimensionaler {F}l\"achen im {$R_4$}.
\newblock {\em Abh. Math. Sem. Univ. Hamburg}, 4(1):174--177, 1925.

\bibitem[BB05]{BB_BraidsSurvey}
Joan~S. Birman and Tara~E. Brendle.
\newblock Braids: a survey.
\newblock In {\em Handbook of knot theory}, pages 19--103. Elsevier B. V.,
  Amsterdam, 2005.

\bibitem[Bir74]{Birman}
Joan~S. Birman.
\newblock {\em Braids, links, and mapping class groups}.
\newblock Princeton University Press, Princeton, N.J., 1974.
\newblock Annals of Mathematics Studies, No. 82.

\bibitem[BM92]{BM_theunlink}
Joan~S. Birman and William~W. Menasco.
\newblock Studying links via closed braids. {V}. {T}he unlink.
\newblock {\em Trans. Amer. Math. Soc.}, 329(2):585--606, 1992.

\bibitem[CIM{\etalchar{+}}24]{contactcutgraphweinstein}
Nickolas Castro, Gabriel Islambouli, Jie Min, Sümeyra Sakallı, Laura
  Starkston, and Angela Wu.
\newblock The contact cut graph and a weinstein $\mathcal{L}$-invariant, 2024.

\bibitem[CKS04]{CKS_Surfaces_in_4space}
Scott Carter, Seiichi Kamada, and Masahico Saito.
\newblock {\em Surfaces in 4-space}, volume 142 of {\em Encyclopaedia of
  Mathematical Sciences}.
\newblock Springer-Verlag, Berlin, 2004.
\newblock Low-Dimensional Topology, III.

\bibitem[CS98]{CS:Book}
J.~Scott Carter and Masahico Saito.
\newblock {\em Knotted surfaces and their diagrams}, volume~55 of {\em
  Mathematical Surveys and Monographs}.
\newblock American Mathematical Society, Providence, RI, 1998.

\bibitem[DER21]{diao2021knots}
Yuanan Diao, Claus Ernst, and Philipp Reiter.
\newblock Knots with equal bridge index and braid index.
\newblock {\em Journal of Knot Theory and Its Ramifications}, 30(11):2150075,
  2021.

\bibitem[Fox61]{QuickTrip}
R.~H. Fox.
\newblock A quick trip through knot theory.
\newblock In {\em Topology of 3-manifolds and related topics ({P}roc. {T}he
  {U}niv. of {G}eorgia {I}nstitute, 1961)}, pages 120--167. Prentice-Hall,
  Inc., Englewood Cliffs, NJ, 1961.

\bibitem[Fox62]{Fox61}
Ralph~H. Fox.
\newblock A quick trip through knot theory.
\newblock In {\em Topology of 3-manifolds and related topics ({P}roc. {T}he
  {U}niv. of {G}eorgia {I}nstitute, 1961)}, pages 120--167. Prentice-Hall,
  Englewood Cliffs, N.J., 1962.

\bibitem[Gil82]{Giller}
Cole~A. Giller.
\newblock Towards a classical knot theory for surfaces in {${\bf R}^{4}$}.
\newblock {\em Illinois J. Math.}, 26(4):591--631, 1982.

\bibitem[GK78]{GoldKauff}
Deborah~L. Goldsmith and Louis~H. Kauffman.
\newblock Twist spinning revisited.
\newblock {\em Trans. Amer. Math. Soc.}, 239:229--251, 1978.

\bibitem[GK16]{GK}
David Gay and Robion Kirby.
\newblock Trisecting 4-manifolds.
\newblock {\em Geom. Topol.}, 20(6):3097--3132, 2016.

\bibitem[Hir03]{Hiro03}
Susumu Hirose.
\newblock A four-dimensional analogy of torus links.
\newblock {\em Topology Appl.}, 133(3):199--207, 2003.

\bibitem[IS24]{Islambouli_divides}
Gabriel Islambouli and Laura Starkston.
\newblock Multisections with divides and {W}einstein 4-manifolds.
\newblock {\em J. Symplectic Geom.}, 22(2):223--266, 2024.

\bibitem[Iwa88]{Iwa88}
Zyun'iti Iwase.
\newblock Dehn-surgery along a torus {$T^2$}-knot.
\newblock {\em Pacific J. Math.}, 133(2):289--299, 1988.

\bibitem[JMMZ22]{Joseph_Classical_knot_theory}
Jason Joseph, Jeffrey Meier, Maggie Miller, and Alexander Zupan.
\newblock Bridge trisections and classical knotted surface theory.
\newblock {\em Pacific J. Math.}, 319(2):343--369, 2022.

\bibitem[JP25]{joseph2025meridional}
Jason Joseph and Puttipong Pongtanapaisan.
\newblock Meridional rank and bridge number of knotted 2-spheres.
\newblock {\em Canadian Journal of Mathematics}, 77(1):282--299, 2025.

\bibitem[Kam94]{kamada1994characterization}
Seiichi Kamada.
\newblock A characterization of groups of closed orientable surfaces in
  4-space.
\newblock {\em Topology}, 33(1):113--122, 1994.

\bibitem[Kam02]{KamBook}
Seiichi Kamada.
\newblock {\em Braid and knot theory in dimension four}.
\newblock Mathematical Surveys and Monographs. American Mathematical Society,
  2002.

\bibitem[Kan83]{Kanenobu83}
Taizo Kanenobu.
\newblock Fox's {$2$}-spheres are twist spun knots.
\newblock {\em Mem. Fac. Sci. Kyushu Univ. Ser. A}, 37(2):81--86, 1983.

\bibitem[Ker65]{Kervaire65}
Michel~A. Kervaire.
\newblock Les n\oe uds de dimensions sup\'erieures.
\newblock {\em Bull. Soc. Math. France}, 93:225--271, 1965.

\bibitem[Kim20]{kim20glucktwist}
Seungwon Kim.
\newblock Gluck twist and unknotting of satellite $2$-knots, 2020.

\bibitem[KSS82]{KSS1}
Akio Kawauchi, Tetsuo Shibuya, and Shin'ichi Suzuki.
\newblock Descriptions on surfaces in four-space. {I}. {N}ormal forms.
\newblock {\em Math. Sem. Notes Kobe Univ.}, 10(1):75--125, 1982.

\bibitem[Lev66]{Levine}
J.~Levine.
\newblock Polynomial invariants of knots of codimension two.
\newblock {\em Ann. of Math. (2)}, 84:537--554, 1966.

\bibitem[Lit79]{Litherland}
R.~A. Litherland.
\newblock Deforming twist-spun knots.
\newblock {\em Trans. Amer. Math. Soc.}, 250:311--331, 1979.

\bibitem[Liv95]{Livingston}
Charles Livingston.
\newblock Knotted symmetric graphs.
\newblock {\em Proc. Amer. Math. Soc.}, 123(3):963--967, 1995.

\bibitem[Mei25]{Meier_personal}
Jeff Meier.
\newblock Private communication, Summer 2025.

\bibitem[Mor78]{morton1978infinitely}
Hugh~R Morton.
\newblock Infinitely many fibred knots having the same alexander polynomial.
\newblock {\em Topology}, 17(1):101--104, 1978.

\bibitem[Mor86]{Morton_1986}
H.~R. Morton.
\newblock Threading knot diagrams.
\newblock {\em Mathematical Proceedings of the Cambridge Philosophical
  Society}, 99(2):247–260, 1986.

\bibitem[MTZ23]{MTZ_cubic_graphs}
Jeffrey Meier, Abigail Thompson, and Alexander Zupan.
\newblock Cubic graphs induced by bridge trisections.
\newblock {\em Math. Res. Lett.}, 30(4):1207--1231, 2023.

\bibitem[MZ17]{MZ17Trans}
Jeffrey Meier and Alexander Zupan.
\newblock Bridge trisections of knotted surfaces in {$S^4$}.
\newblock {\em Trans. Amer. Math. Soc.}, 369(10):7343--7386, 2017.

\bibitem[Ros98]{Roseman-diag}
Dennis Roseman.
\newblock Reidemeister-type moves for surfaces in four-dimensional space.
\newblock In {\em Knot theory ({W}arsaw, 1995)}, volume~42 of {\em Banach
  Center Publ.}, pages 347--380. Polish Acad. Sci. Inst. Math., Warsaw, 1998.

\bibitem[Rud83]{Rudolph}
Lee Rudolph.
\newblock Braided surfaces and {S}eifert ribbons for closed braids.
\newblock {\em Comment. Math. Helv.}, 58(1):1--37, 1983.

\bibitem[Yaj64]{YajimaSimple}
Takeshi Yajima.
\newblock On simply knotted spheres in {$R\sp{4}$}.
\newblock {\em Osaka Math. J.}, 1:133--152, 1964.

\bibitem[Yam87]{YamadaSeifert}
Shuji Yamada.
\newblock The minimal number of {S}eifert circles equals the braid index of a
  link.
\newblock {\em Invent. Math.}, 89(2):347--356, 1987.

\bibitem[Yos94]{Yoshikawa}
Katsuyuki Yoshikawa.
\newblock An enumeration of surfaces in four-space.
\newblock {\em Osaka J. Math.}, 31(3):497--522, 1994.

\bibitem[Zee65]{Zeeman}
E.~C. Zeeman.
\newblock Twisting spun knots.
\newblock {\em Trans. Amer. Math. Soc.}, 115:471--495, 1965.

\end{thebibliography}
\end{document}